\documentclass[]{siamart220329}

\usepackage[T1]{fontenc}
\usepackage{xr}
\usepackage[utf8]{inputenc}
\usepackage{amsfonts}
\usepackage{amsmath,amssymb}
\usepackage[caption=false]{subfig} 
\usepackage{tikz,pgfplots}
\usepackage{multirow}
\usepackage{makecell}
\usepackage{tikz-3dplot}
\usetikzlibrary{math}
\usepackage{bm}
\usepackage{algorithm}
\usepackage{algpseudocode}
\usepackage[top=1in,bottom=1in,left=1in,right=1in]{geometry}
\usepackage{graphicx,epstopdf} 

\usepackage{enumerate}

\usepackage{graphicx}

\usepackage{hyperref}
\usepackage{relsize}
\usepackage{adjustbox}
\usepackage{cleveref}
\usepackage{tcolorbox}
\usepackage{dsfont}
\usepackage{accents}

\usepackage{multirow}
\usepackage{booktabs}
\usepackage{graphicx}

\ifpdf
  \DeclareGraphicsExtensions{.eps,.pdf,.png,.jpg}
\else
  \DeclareGraphicsExtensions{.eps}
\fi

\headers{A Black-box fast algorithm for $N$-body problems in $\MakeLowercase{d}$-dimensions}{Applications to integral equations and support vector machines}

\title{HODLR$\MakeLowercase{d}$D: A new Black-box fast algorithm for $N$-body problems in $\MakeLowercase{d}$-dimensions with guaranteed error bounds}

\subtitle{Applications to integral equations and support vector machines}

\author{Ritesh Khan\thanks{Department of Mathematics, IIT Madras, Chennai, India. \email{\lowercase{khanritesh28@gmail.com}}, \email{\lowercase{kandappanva@gmail.com}}}
\and V A Kandappan\footnotemark[1]
\and Sivaram Ambikasaran\thanks{Department of Mathematics and Robert Bosch Centre for Data Science \& Artificial Intelligence, IIT Madras, Chennai, India. \email{\lowercase{sivaambi@alumni.stanford.edu}}}}

\newsiamremark{remark}{Remark}

\hypersetup{
    colorlinks=true,
    urlcolor=blue,
    citecolor = magenta
    }
\usetikzlibrary{arrows} 
\usetikzlibrary{decorations.markings}

\pgfplotsset{compat=1.16} 

\newcommand{\dsum}{\displaystyle\sum}
\newcommand{\dprod}{\displaystyle\prod}

\newcommand{\dcap}{\displaystyle\cap}

\newcommand{\xb}{\pmb{x}}
\newcommand{\yb}{\pmb{y}}

\newcommand{\bkt}[1]{\left(#1\right)}
\newcommand{\abs}[1]{\left\lvert#1\right\rvert}
\newcommand{\magn}[1]{\left\lVert#1\right\rVert}

\newcommand{\ceil}[1]{\left\lceil#1\right\rceil}

\newcommand{\Pc}{\mathcal{P}}
\newcommand{\Rb}{\mathbb{R}}
\newcommand{\Cb}{\mathbb{C}}

\newcommand{\Zb}{\mathbb{Z}}

\newcommand{\cmt}[1]{\iffalse {#1} \fi}

\makeatletter
\newcommand*{\addFileDependency}[1]{
  \typeout{(#1)}
  \@addtofilelist{#1}
  \IfFileExists{#1}{}{\typeout{No file #1.}}
}
\makeatother

\newcommand{\mvp}{matrix-vector product }

\newcommand{\iso}{0.4}
\newcommand{\redcircle}[2]{\draw [fill=red] (#1,#2) circle (0.1);}
\newcommand{\redsquare}[2]{
    \draw [ultra thick] (#1,#2) rectangle (1+#1,1+#2);
    \redcircle{#1}{#2};
    \redcircle{1+#1}{#2};
    \redcircle{#1}{1+#2};
    \redcircle{1+#1}{1+#2};
}
\newcommand{\cube}[2]{
    \draw [ultra thick] (#1,#2) rectangle (1+#1,1+#2);
    \draw [ultra thick, loosely dotted] (\iso+#1,\iso+#2) rectangle (1+\iso+#1,1+\iso+#2);
    \draw [ultra thick] (1+#1,1+#2) -- (1+\iso+#1,1+\iso+#2);
    \draw [ultra thick, loosely dotted] (1+#1,#2) -- (1+\iso+#1,\iso+#2);
    \draw [ultra thick, loosely dotted] (#1,1+#2) -- (\iso+#1,1+\iso+#2);
    \draw [ultra thick, loosely dotted] (#1,#2) -- (\iso+#1,\iso+#2);
    \redcircle{#1}{#2};
    \redcircle{#1+1}{#2};
    \redcircle{#1}{#2+1};
    \redcircle{#1+1}{#2+1};
    \redcircle{#1+\iso}{#2+\iso};
    \redcircle{#1+1+\iso}{#2+\iso};
    \redcircle{#1+\iso}{#2+1+\iso};
    \redcircle{#1+1+\iso}{#2+1+\iso};
}
\DeclareRobustCommand{\cubex}[6]
{
	\draw[#5] (#1,#2,#3) -- (#1,#2+#4,#3) -- (#1+#4,#2+#4,#3) -- (#1+#4,#2,#3) -- cycle;
	\draw[#5] (#1,#2,#3+#4) -- (#1,#2+#4,#3+#4) -- (#1+#4,#2+#4,#3+#4) -- (#1+#4,#2,#3+#4) -- cycle;
	\draw[#5] (#1,#2,#3) -- (#1,#2,#3+#4);
	\draw[#5] (#1,#2+#4,#3) -- (#1,#2+#4,#3+#4);
	\draw[#5] (#1+#4,#2,#3) -- (#1+#4,#2,#3+#4);
	\draw[#5] (#1+#4,#2+#4,#3) -- (#1+#4,#2+#4,#3+#4);

    \node at (#1 + 0.5*#4,#2 + 0.5*#4,#3 + 0.5*#4) {#6};
	
}
\DeclareUnicodeCharacter{2212}{-}

\begin{document}

\maketitle

\begin{abstract}
We study the rank of sub-matrices arising out of kernel functions, $F(\xb,\yb): \Rb^d \times \Rb^d \mapsto \Rb$, where $\xb,\yb \in \Rb^d$ and $F\bkt{\xb,\yb}$ is smooth everywhere except along the line $\yb=\xb$. Such kernel functions are frequently encountered in a wide range of applications such as $N$-body problems, Green's functions, integral equations, geostatistics, kriging, Gaussian processes, radial basis function interpolation, etc. To our knowledge, this is the first work to formally study the rank growth of a wide range of kernel functions in all dimensions.
In this article, we prove new theorems bounding the rank of different sub-matrices arising from these kernel functions. Bounds like these are often useful for analyzing the complexity of various hierarchical matrix algorithms. We also plot the numerical rank growth of different sub-matrices arising out of various kernel functions in $1$D, $2$D, $3$D and $4$D, which, not surprisingly, agrees with the proposed theorems. Another significant contribution of this article is that, using the obtained rank bounds, we also propose a way to extend the notion of \textbf{\emph{weak-admissibility}} for hierarchical matrices in higher dimensions. Based on this proposed \textbf{\emph{weak-admissibility}} condition, we develop a black-box (kernel-independent) fast algorithm for $N$-body problems, hierarchically off-diagonal low-rank matrix in $d$ dimensions (HODLR$d$D), which can perform matrix-vector products with $\mathcal{O}\bkt{pN \log \bkt{N}}$ complexity in any dimension $d$, where $p$ doesn't grow with any power of $N$. More precisely, our theorems guarantee that $p \in \mathcal{O} \bkt{\log \bkt{N} \log^d \bkt{\log \bkt{N}}}$, which implies our HODLR$d$D algorithm scales almost linearly. The $\texttt{C++}$ implementation with \texttt{OpenMP} parallelization of the HODLR$d$D is available at \texttt{\url{https://github.com/SAFRAN-LAB/HODLRdD}}. We also discuss the scalability of the HODLR$d$D algorithm and showcase the applicability by solving an integral equation in $4$ dimensions and accelerating the training phase of the support vector machines (SVM) for the data sets with four and five features.

\end{abstract}


\section{Introduction}
Matrices arising out of kernel functions are encountered frequently in many applications such as integral equations~\cite{ho2013hierarchical}, electromagnetic scattering~\cite{carpentieri2004sparse}, Gaussian process regression ~\cite{gp_book}, machine learning~\cite{gray2000}, radial basis function interpolation~\cite{carr1997surface,radial}, kernel density estimation~\cite{ker_den}, etc. These matrices are typically large and dense, since the kernel functions are not compactly supported. Due to this, storing these matrices and performing matrix operations such as matrix-vector products, solving linear systems, matrix factorizations, etc., are challenging. However, such matrices, especially ones arising out of an underlying application, possess some structure, which is leveraged to store and perform matrix operations. One such structure, which is frequently exploited in the context of matrices arising out of kernel functions, is its rank-structuredness. We do not seek to review the entire literature on the rank-structuredness of matrices arising out of kernel functions. We direct the interested readers to selected important developments on this front~\cite{fmm_ref,barnes1986hierarchical, greengard1987fast,ambikasaran2019hodlrlib,fong2009black,ambikasaran2015fast,foreman2017fast,ambikasaran2013mathcal,ifmm,ambikasaran2013large,cai2018smash,li2014kalman,ambikasaran2014fast,kandappan2022hodlr2d,ambikasaran2015generalized,ying,chandrasekaran2007fast,hackbusch2004hierarchical,randomized_aca}.

The hierarchical low-rank structure is one of the most frequently encountered rank-structuredness for matrices arising out of $N$-body problems. The first works along these lines were the Barnes-Hut algorithm (Treecode)~\cite{barnes1986hierarchical} and the Fast Multipole Method~\cite{greengard1987fast} (from now on abbreviated as FMM), which reduced the computational complexity of performing $N$-body simulations from $\mathcal{O}\bkt{N^2}$ to $\mathcal{O}\bkt{N \log N}$ and $\mathcal{O}\bkt{N}$ respectively, for a given accuracy. One of the main highlights of the FMM is that the \emph{\textbf{far-field}} interactions can be efficiently approximated by a degenerate kernel. The above two algorithms can be easily interpreted as matrix-vector products, and the fact that the \emph{\textbf{far-field}} interactions can be approximated by degenerate kernel approximation is equivalent to stating that sub-matrices corresponding to the \emph{\textbf{far-field}} interactions can be well-represented using low-rank matrices. This interpretation has led to a new class of fast matrix-vector product algorithms leveraging the low-rank structure of certain sub-matrices. Matrices that possess such a hierarchical low-rank structure are termed as hierarchical matrices, and various hierarchical low-rank algorithms have been developed for accelerating matrix-vector products~\cite{greengard1987fast,fong2009black,kandappan2022hodlr2d}. These hierarchical structures have also been leveraged to construct fast direct solvers~\cite{ambikasaran2019hodlrlib,chandrasekaran2007fast,ifmm,ho2012fast}.

However, most of the works on higher dimensional $(d>1)$ hierarchical low-rank matrices so far have only studied the rank of sub-matrices corresponding to the \emph{\textbf{far-field}} interactions (based on the strong admissibility)~\cite{shortFMM,greengard1987fast,fong2009black,ying2004kernel,greengard1997new,engquist2007fast}. \emph{\textbf{In this article, we study the rank of all possible interactions in $d$ dimensions, where $d \in \Zb^+$}}. More precisely, we prove two theorems on \textbf{rank growth of interaction between any two hyper-cubes in $d$-dimension, containing particles in its interior, that share a hyper-surface of at-most $d'$-dimension, where $d' \in \{0,1,2,\ldots,d-1\}$}.
To the best of our knowledge, we believe that this is the first work to formally study the rank of all \emph{\textbf{nearby}} interactions in any dimension for generic kernel functions that are not smooth along the diagonal. Also, based on our theorems, we develop a kernel-independent fast algorithm (HODLR$d$D) for matrix-vector product in $d$ dimensions.  We would like to point out that our earlier work~\cite{kandappan2022hodlr2d} inspired this work, where we obtained rank bounds for \emph{\textbf{nearby}} interactions in two dimensions for the kernel function, $\log\bkt{\magn{\xb-\yb}_2}$. Further, it is also worth noting that some of the hierarchical matrices such as HODLR~\cite{ambikasaran2013fast}, HSS~\cite{xia2010fast,chandrasekaran2007fast}, HBS~\cite{gillman2012direct}, $\mathcal{H}$-matrices~\cite{hackbusch2004hierarchical} with weak admissibility rely on representing certain or all \emph{\textbf{nearby}} interactions as low-rank matrices.
\section{Existing works and novelty of our work}
\label{Existing works and novelty of our work}
The rank of the kernel matrices has been studied many times in the past in the context of many applications. In this section, we discuss some of the existing work that has studied the rank of kernel matrices and those related to our current article. We highlight how our work distinguishes itself from the existing works.

Hackbusch et al.~\cite{borm2003hierarchical,borm2003introduction,grasedyck2003construction} were the pioneers in studying the rank structure of matrices arising out of kernel functions. Even though the tree code~\cite{barnes1986hierarchical}, FMM~\cite{cheng1999fast,greengard1987fast,greengard1988rapid,greengard1997new,carrier1988fast} were based on similar ideas, it was Hackbusch et al.~\cite{borm2003hierarchical,borm2003introduction,grasedyck2003construction} who interpreted them as low-rank representation of appropriate sub-matrices. Hackbusch et al.~\cite{borm2003hierarchical,borm2003introduction,grasedyck2003construction} in their very first works discussed the low-rank representation of kernel sub-matrices arising out of interaction between clusters of particles satisfying the standard (or strong) admissibility criterion, i.e., where the separation distance between two clusters exceeds the diameter of either cluster. The main highlight of FMM and hierarchical matrices is that when the clusters satisfy the standard (or strong) admissibility criterion, the corresponding matrices arising out of kernel functions are low-rank.

In their subsequent work~\cite{hackbusch2004hierarchical}, Hackbusch et al. studied the rank of kernel matrices arising out of $1$D distribution of particles sharing a boundary, i.e., a vertex in $1$D. The article shows that the rank of neighbors in $1$D scale as $\mathcal{O}\bkt{\log\bkt{N}}$, where $N$ is the number of particles in each cluster. The article terms such clusters, i.e., clusters that are neighbors in $1$D, as clusters satisfying weak admissibility, since the rank of interaction between these clusters do not scale with any power of $N$. This work \emph{doesn't discuss about the rank of interaction of neighboring clusters in higher dimension} (though they state that higher dimensions will be considered later and refer to an article in the bibliography. However, to the best of our knowledge and searches the article was never published or is available to the public).

It is to be noted that the notion of \textit{weak admissibility} introduced in~\cite{hackbusch2004hierarchical} is applicable only in the context of one dimension. Our work in this article, \emph{could} be interpreted as the extension of \textit{weak admissibility} to \textbf{higher dimensions}, since we prove that among all nearby interactions that exist for a cluster in $d$-dimensions, the rank of sub-matrices (arising out of a wide range of kernel functions) corresponding to the interaction between neighbors that share a vertex do not scale with any power of $N$. This result, which we present in this article with detailed, rigorous proof, is new to the best of our knowledge. The major highlight of our theorem is that the theorem is applicable in all dimensions \emph{for a wide range of kernel functions} encountered in many applications.

A recent work by Xia~\cite{xia2021multi} attempts to extend the notion of weak admissibility in $2$D and design fast algorithms for matrix operations leveraging this low-rank structure. It~\cite{xia2021multi} \emph{\bf{assumes}} that the underlying kernel admits a multipole like expansion to construct low-rank representation and to obtain bounds on ranks of vertex sharing and edge-sharing interaction in $2$D. Further, the article~\cite{xia2021multi} \emph{\bf{assumes}} that for a given kernel, the rank of interaction between clusters is independent of the cluster size. This may not be true for some kernels. For instance, the error term in multipole expansion of $1/r$ explicitly involves the cluster size and thereby the rank of the corresponding interaction is dependent on the underlying cluster size \cite{greengard1997new} (though the dependence on cluster size is weak, i.e., scales in the logarithm of the cluster size). We would like to emphasize that our work is significantly different from~\cite{xia2021multi} by highlighting the following important differences:
\begin{itemize}
    \item Our work is general for a wide range of kernels in \emph{\bf{all}} dimension.
    \item We \emph{\bf{prove that}} one can obtain a separable expansion for kernels even for neighboring clusters.
    \item We \emph{\bf{construct explicit error bounds}} for the kernel in terms of the separable expansion obtained for neighboring clusters.
    \item In fact, one can see the difference in both our results for $2$D vertex sharing. 
While~\cite{xia2021multi}, \emph{\bf{with its set of restrictive assumptions}}, obtain a rank of $\mathcal{O}\bkt{\log\bkt{N/\epsilon}}$, \emph{\bf{our result, which is applicable to a wide range of kernels}}, provides the bound as $\mathcal{O}\bkt{\log\bkt{N}\log^2\bkt{\dfrac{\log{N}}\epsilon}}$. A slightly stronger bound for the logarithmic kernel scaling as $\mathcal{O}\bkt{\log\bkt{N}\log\bkt{\dfrac{\log{N}}\epsilon}}$ can be found here~\cite{kandappan2022hodlr2d}. Our result in $2$D may not be practically different from the result in~\cite{xia2021multi} (since the term $\log\bkt{\log\bkt{N}}$ is going to be bounded by a constant for problems of interest), but our bounds on the rank are accurate.
\end{itemize}

Another work that discusses the rank of kernel functions is the work by Wang et al.~\cite{wang2018numerical}, where they discuss the rank growth of radial basis functions (RBF) that are globally smooth, i.e., $K(x,y) = f\bkt{\magn{x-y}_2^2}$, where the function $f(x,y)$ is globally smooth. Since these kernels are globally smooth, given an $\epsilon > 0$, one can obtain a separable expansion for $K(x,y)$ as $\dsum_{i=1}^r g_i(x)h_i(y)$ on a compact set $D \subseteq \Rb^d \times \Rb^d$ such that
\begin{equation}
    \abs{K(x,y)-\dsum_{i=1}^r g_i(x)h_i(y)} < \epsilon
\end{equation}
for all $\bkt{x,y} \in D$. This immediately implies that the numerical rank of the associated kernel matrix doesn't scale with $N$. Our current work is drastically different from the work by Wang et al.~\cite{wang2018numerical} on multiple fronts.
\begin{itemize}
    \item In our current work, we are considering kernels that not globally smooth. More specifically, we are considering kernels that are \textbf{\emph{not smooth along $\pmb{y} = \pmb{x}$}} and are smooth everywhere else.
    \item The other significant difference is that the kernels considered in this article are not necessarily a function of the Euclidean distance between $\pmb{y}$, $\pmb{x}$ (though in the examples provided, we consider only such kernels). This difference is significant since in the work by Wang et al.~\cite{wang2018numerical}, the kernel is a radial basis function, which implies that the domain of the radial basis function is essentially one dimension. Our result is applicable to a broader range of functions.
\end{itemize}

The works by Ho \& Greengard~\cite{ho2012fast}, Ho \& Lexing~\cite{ho2013hierarchical} obtain bounds on rank heuristically while discussing the complexity of their algorithms. We believe our work differs substantially from Ho et al. \cite{ho2012fast,ho2013hierarchical} as indicated below.
\begin{itemize}
    \item \cite{ho2012fast,ho2013hierarchical} heuristically discuss the rank of the vertex-sharing case in $1$D, the edge-sharing case in $2$D, and the face-sharing case in $3$D. More precisely, they heuristically discuss the rank of the interaction of two $d$-dimensional hyper-cubes that share a $(d-1)$-dimensional hyper-surface.  On the other hand, in this article, we discuss the rank growth of \textbf{\emph{all interactions}}, i.e., for two hyper-cubes in $d$ dimensions, we obtain bounds on ranks when they share $d'$-dimensional hyper-surface, where $d' \in \{0,1,2,\ldots,d-1\}$. Also, our results on the rank growth for kernel matrices are backed by theorems applicable for a wide-range of kernels in $d$-dimensions.
    \item The second difference is that the initialization cost of our algorithm (HODLR$d$D) is almost linear, even in $d$ dimensions. But the initialization cost of the algorithms in~\cite{ho2012fast,ho2013hierarchical} is high.
\end{itemize}

Another work by Corona et al.~\cite{corona2015n}, which is along the lines of \cite{ho2012fast,ho2013hierarchical,xia2021multi}, presents a fast direct solver for integral equations in two dimensions. For the same reasons as expressed above, the work in this article is different from the work of Corona et al.~\cite{corona2015n}.




\vspace{1cm}
\boxed{\textbf{\textit{Main highlights of this article :}}}
\begin{enumerate}
    \item We prove two new theorems for the rank of kernel matrices arising out of a \textbf{\emph{wide range of kernel functions in all dimensions}}, which give us rigorous bounds on the rank of kernel matrices for different types of interactions in all dimensions.
    \item One of our theorems guarantees that the rank of kernel matrices arising out of \textbf{\emph{vertex-sharing interactions do not scale as any power of $N$}}. This is leveraged to construct an almost linear fast matrix-vector product algorithm, the hierarchically off-diagonal low-rank matrix in $d$ dimensions (from now on abbreviated as HODLR$d$D). We would also like to state that looking at only the vertex-sharing and far-field interaction is the right way to extend the notion of weak admissibility (in the context of hierarchical matrices) to higher dimensions.
    \item We have implemented HODLR$d$D in a user-friendly and dimension-independent fashion, i.e., the user can apply it in any dimension $d$. Also, this algorithm can be used in a black-box (kernel-independent) fashion. As part of this article, we would also like to release the code made available at \texttt{\url{https://github.com/SAFRAN-LAB/HODLRdD}}.
    \item We apply HODLR$d$D to a wide range of problems from solving integral equations in higher dimensions to accelerate the training phase of classical kernel Support Vector Machines (SVM). Also, we compare the performance of HODLR$d$D with $\mathcal{H}$ matrix with strong admissibility (Treecode) and HODLR representation (where all non self-interactions are compressed). 
\end{enumerate}
\textbf{\emph{Outline of the article.}} The rest of the article is organized as follows. In~\Cref{Preliminaries}, we discuss some notations and lemmas, which we will use to prove our theorems. In~\Cref{ddim}, we prove our main theorems with numerical results in $4$D, $3$D, $2$D and $1$D. In~\Cref{weak_admis}, we describe the notion of \emph{weak admissibility} for higher dimension based on our theorems and discuss HODLR$d$D algorithm based on this weak admissibility condition. Further, in \Cref{num_results}, we compare the scalability and performance of the HODLR$d$D algorithm with other hierarchical algorithms and present some applications of the HODLR$d$D algorithm.  

\section{Preliminaries} \label{Preliminaries}
In this section, we state the notations, definitions and lemmas, we will use in this article. We also state the main theorems, which we will prove in this article.

\begin{itemize}
\item \textbf{Bernstein ellipse:} The standard Bernstein ellipse is an ellipse whose foci are at $(-1,0)$, $(1,0)$ and its parametric equation is given by
\begin{equation}
    \mathcal{B}_{\rho} = \text{interior} \bkt{\left\{\dfrac{\rho e^{it}+\rho^{-1} e^{-it}}2 \in \mathbb{C}: t \in [0,2\pi) \right\}}
\end{equation}
for a fixed $\rho > 1$.
\item
Let $\tau_{[\underaccent{\bar}y,\bar{y}]}: \Cb \mapsto \Cb$ such that $\tau_{[\underaccent{\bar}y,\bar{y}]}\bkt{z} = \bkt{\dfrac{\underaccent{\bar}y+\bar{y}}2} + \bkt{\dfrac{\bar{y}-\underaccent{\bar}y}2}z$, where $\underaccent{\bar}y,\bar{y} \in \Cb$. Essentially this map scales and shifts the interval $[-1,1]$ to the interval $[\underaccent{\bar}y,\bar{y}]$.
\item \textbf{Generalized Bernstein ellipse:} Let
\begin{equation}
V = [\underaccent{\bar}y_1,\bar{y}_1] \times [\underaccent{\bar}y_2,\bar{y}_2] \times \cdots \times [\underaccent{\bar}y_d,\bar{y}_d]
\end{equation}
be a hyper-cube with side length $r$, i.e., $\abs{\bar{y}_k-\underaccent{\bar}y_k} = r$ for all $k \in \{1,2,\ldots,d\}$. The generalized Bernstein ellipse on a hyper-cube $V$ is denoted as $\mathcal{B}\bkt{V,\rho'}$ with $\rho' \in \bkt{1,\infty}^d$ and is given by
\begin{equation}
    \mathcal{B}\bkt{V,\rho'} = \tau_{[\underaccent{\bar}y_1,\bar{y}_1]}\bkt{\mathcal{B}_{\rho_1}} \times \tau_{[\underaccent{\bar}y_2,\bar{y}_2]}\bkt{\mathcal{B}_{\rho_2}} \times \cdots \times \tau_{[\underaccent{\bar}y_d,\bar{y}_d]}\bkt{\mathcal{B}_{\rho_d}}
\end{equation}
$\tau_{[\underaccent{\bar}y_k,\bar{y}_k]}\bkt{\mathcal{B}_{\rho_k}}$ scales the standard Bernstein ellipse whose foci are at $-1$ and $1$ to an ellipse whose foci are $\underaccent{\bar}y_k$ and $\bar{y}_k$, where $\rho_k > 1$ for $k \in \{1,2,\ldots,d\}$.
   \item \(T_k(x)\) is Chebyshev polynomial of the first kind, i.e.,
        $T_k(x) = \cos\bkt{k \cos^{-1}(x)}$ for $k \geq 0$.
    \item Chebyshev nodes are the points $ y^k = \cos\bkt{\dfrac{\pi k}{p}}$, $k \in \{0,1,2,\dots,p\}$
    \item \textbf{Polynomial interpolation error in different dimensions.} We will be using interpolation to prove the proposed theorems. The main tool we will rely on is the following \cref{th4}.
    \begin{lemma} \label{th3}
        Let a function \(f\) analytic in \([−1, 1]\) be analytically continuable to the open Bernstein ellipse $\mathcal{B}_{\rho}$ with $\rho > 1$, where it satisfies \(\abs{f (y)} \leq M \) for some $M>0$, then for each $n \geq 0$ its Chebyshev interpolant $\Tilde{f}\bkt{y} = \dsum_{j=0}^{p} c_j T_j(y)$ satisfy 
        \begin{equation}
            \max_{y \in [-1,1]} \abs{{f\bkt{y}-\Tilde{f}\bkt{y} }} = \magn{f-\Tilde{f}}_{\infty} \leq \frac{4M \rho^{-p}}{\rho -1}
        \end{equation}
        where $c_j = \dfrac{2^{\mathds{1}_{0<j<p}}}{p} \dsum_{k=0}^{p}{''} f\bkt{y^k} \cos \bkt{\dfrac{j \pi k}{p}}, j \leq p$ and the notation $\dsum {''}$ indicates that the first and last summands are halved.
    \end{lemma}
    The proof of this lemma is given in~\cite{Trefethen}. Note that $\Pc_p(f)\bkt{y}$ can also be written as
    \begin{equation*}
        \Tilde{f}\bkt{y} = \dsum_{j=0}^p f\bkt{y^j} L_j\bkt{y}
    \end{equation*}
    where $L_j\bkt{y}$ is the Lagrange polynomial given by
       $ L_j\bkt{y} = \dprod_{k=0, k\neq j}^p \bkt{\dfrac{y-y^k}{y^j-y^k}}$.
  Now we discuss a multi-dimension extension of the above lemma given by~\cite{Gab19,glau}. Let $f: V \mapsto \mathbb{R}$ with $V = [-1,1]^d \subset \mathbb{R}^d$. Let $\Bar{p} := (p_1,p_2,\dots, p_d)$ with $p_i \in \mathbb{N}_0$ for $i=1,2,\dots, d$. The interpolation with $\dprod_{i=1}^{d}(p_i +1)$ summands is given by $$\Tilde{f}\bkt{\pmb{y}} := \sum_{\pmb{j} \in J}c_{\pmb{j}} T_{\pmb{j}}(\pmb{y})$$ where $J = \{ (j_1,j_2,\dots,j_d) \in \mathbb{N}_{0}^{d} : j_i \leq p_i\}$, $T_{\pmb{j}}(y_1,y_2,\hdots,y_d) = \dprod_{i=1}^d T_{j_i}(y_i)$ and 
  \begin{equation}
      c_{\pmb{j}} = \bkt{\dprod_{i=1}^d \dfrac{2^{\mathds{1}_{0<j_i<p_i}}}{p_i}} \dsum_{k_1=0}^{p_1}{''} \hdots \dsum_{k_d=0}^{p_d}{''} f(y^{\pmb{k}}) \dprod_{i=1}^{d} \cos \bkt{\dfrac{j_i \pi k_i}{p_i}}
  \end{equation}
  $\pmb{k} = (k_1,k_2,\hdots,k_d) \in J$ and $y^{\pmb{k}} = \bkt{y^{k_1}, y^{k_2}, \hdots, y^{k_d}}$ with $y^{k_i} =  \cos \bkt{\dfrac{\pi k_i}{p_i}} $ , $k_i=0,1,\hdots,p_i$, $i=1,2,\hdots,d$ and $\pmb{j} \in J$. Note that $\abs{J} = \dprod_{i=1}^{d}(p_i +1)$
  
    \begin{lemma} \label{th4}
        Let $f : V \mapsto \Rb$ has an analytic extension to some generalized Bernstein ellipse $\mathcal{B} \bkt{V,\rho'}$, for some parameter vector $\rho' \in \bkt{1,\infty}^d$ and \\ $\magn{f}_{\infty} = \underset{y \in \mathcal{B} \bkt{V,\rho'}}{\max} \abs{f(y)} \leq M$. Then
        \begin{equation}
            \max_{\yb \in V} \abs{f \bkt{\yb}-\Tilde{f} \bkt{\yb}} = \magn{f- \Tilde{f}}_{\infty} \leq  \dsum_{m=1}^{d} 4M \bkt{\dfrac{\rho_{m}^{-p_m}}{\rho_m -1}} + \dsum_{l=2}^{d} 4M \dfrac{\rho_{l}^{-p_l}}{\rho_l -1} \bkt{ 2^{l-1} \dfrac{(l-1) + 2^{l-1}-1}{\prod_{j=1}^{l-1}\bkt{1-\dfrac{1}{\rho_j}}}}
        \end{equation}

    \end{lemma}
    The proof is given by Glau et al.~\cite{glau} using induction. In our case, we will be working with $p_1 = p_2 = \cdots = p_d = p $. Further, for the sake of simplicity, we set $\rho = \min \{ \rho_i : i=1,2,\dots, d\}$.
    \begin{equation} \label{eq0}
        \magn{f- \Tilde{f}}_{\infty} \leq 4M \frac{\rho^{-p}}{\rho -1} \overbrace{\bkt{ \sum_{m=1}^{d}1 +  \sum_{l=2}^{d}  2^{l-1} \dfrac{(l-1) + 2^{l-1}-1}{\bkt{1-\dfrac{1}{\rho}}^{l-1}}}}^{V_d}
    \end{equation}
    We have $V_d = \bkt{d +  \dsum_{l=2}^{d}  2^{l-1} \dfrac{(l-1) + 2^{l-1}-1}{\bkt{1-\dfrac{1}{\rho}}^{l-1}}}$, it does not depend upon $p.$ \\
    Therefore, from~\cref{eq0} we have
    \begin{equation} \label{eq1}
        \magn{f- \Tilde{f}}_{\infty} \leq 4MV_d \frac{\rho^{-p}}{\rho -1}
    \end{equation}
    We use~\cref{eq1} to prove our higher dimensional $(d>1)$ results. Note that $\Tilde{f}\bkt{\pmb{y}}$ can also be written as
    \begin{equation}
        \Tilde{f}\bkt{\pmb{y}} = \dsum_{\pmb{j} \in J} f\bkt{y^{\pmb{j}}} R_{\pmb{j}}\bkt{\pmb{y}}
    \end{equation}
    where $R_{\pmb{j}}\bkt{y}$ is given by
       $ R_{\pmb{j}}\bkt{\pmb{y}} = \dprod_{i=1}^d L_{j_{i}}\bkt{y_i}$
\item
\textbf{Source hyper-cube.} Let $Y \subset \Rb^d$ be a compact hyper-cube in $d$-dimensions with $N=n^d$ charges (or sources or particles) uniformly distributed in its interior, where $n \in \mathbb{N}$. We take uniformly distributed sources to prove our theorems. But our HODLR$d$D algorithm applies to the non-uniform distribution of particles (\cref{k_svm}).
\item
\textbf{Locations of sources.} Let $R_Y$ be the set of locations of these sources, i.e., $R_Y = \{\pmb{y_1},\pmb{y_2},\ldots,\pmb{y_N}\}$ and $R_Y \subset \text{interior}\bkt{Y}$.
\item
\textbf{Target hyper-cube.} Let $X \subset \Rb^d$ be another compact hyper-cube in $d$-dimension, which is identical to the hyper-cube $Y$ and contains $T$ distinct target points in its interior. Further, we will only consider cases where
   $ \text{interior}(X) \dcap \text{interior}(Y) = \emptyset $.
\item
\textbf{Locations of targets.} Let $R_X$ be the set of locations of these targets, i.e.,  $R_X = \{\pmb{x_1},\pmb{x_2},\ldots,\pmb{x_T}\}$ and $R_X \subset \text{interior}\bkt{X}$.
\item \textbf{Kernel function.} The function $F(\xb,\yb): \text{interior}\bkt{X} \times \text{interior}\bkt{Y} \mapsto \Rb$ will be called the \textbf{\emph{kernel function}} throughout this article.
\item $\text{diam}(Y) = \sup\{\magn{\pmb{\Tilde{\alpha}}-\pmb{\Tilde{\beta}}}_2:\pmb{\Tilde{\alpha}},\pmb{\Tilde{\beta}}\in Y\}$
\item For any two hyper-cubes $U$ and $V$ in $\Rb^d$, \text{dist}$(U,V) = \inf\{\magn{\pmb{\Tilde{u}}-\pmb{\Tilde{v}}}_2:\pmb{\Tilde{u}} \in U \text{ and }\pmb{\Tilde{v}}\in V\}$.
\item \textbf{\textit{Analytic continuation assumption of kernel function}}: Let $U,V$ be compact hyper-cubes in $\Rb^d$ such that $\text{dist}\bkt{U,V} > 0$. We say that $F:U \times V \mapsto \Rb$ is analytically continuable to a generalized Bernstein ellipse if there exists a generalized Bernstein ellipse $\mathcal{B}\bkt{V,\rho'}$ and an analytic function $F_a:U \times \mathcal{B}\bkt{V,\rho'} \mapsto \Cb$ such that the restriction of $F_a$ onto $U \times V$ is same as $F$ and $F_a$ is bounded on $U \times \mathcal{B}\bkt{V,\rho'}$. \emph{Throughout the article, we will assume that the kernel function satisfies the analytic continuation assumption (in addition to the fact that the function is not analytic along $\yb=\xb$).}
\item \textbf{Kernel matrix.} The matrix $K \in \Rb^{T \times N}$ will be called the \textbf{\emph{kernel matrix}} throughout this article. The $(i,j)^{th}$ entry of $K$ is given by
   $ K(i,j) = F\bkt{\pmb{x_i},\pmb{y_j}} $.
\item \textbf{Approximation of the kernel matrix via Chebyshev interpolation: } The function $F(\pmb{x},\pmb{y})$ is interpolated using Chebyshev nodes along $\yb$. The interpolant is given by
    \begin{equation}
        \Tilde{F}\bkt{\xb,\yb} = \dsum_{\pmb{j} \in J}c_{\pmb{j}}\bkt{\xb} T_{\pmb{j}}\bkt{\yb}
    \end{equation}
More precisely,
    \begin{equation}
        \Tilde{F} \bkt{\xb,\yb} = \sum_{\pmb{k} \in J} F \bkt{\pmb{x},y^{\pmb{k}}} R_{\pmb{k}} \bkt{\yb}
    \end{equation}
where $R_{\pmb{k}}$ is Lagrange basis. Let $\tilde{K}$ be the matrix corresponding to the function $\tilde{F}\bkt{\xb,\yb}$, i.e., $\tilde{K}\bkt{i,j} = \dsum_{\pmb{k} \in J} F \bkt{\xb_i,y^{\pmb{k}}}R_{\pmb{k}}\bkt{\yb_j} = \tilde{F}\bkt{\pmb{x}_i,\pmb{y}_j}$. The matrix $\tilde{K}$ will be the approximation to the kernel matrix $K$ with the rank of $\Tilde{K}$ is $\abs{J}$.
\end{itemize}

\begin{definition}
    \textbf{The max norm of a matrix.} The \emph{max} norm of a matrix $K \in \Rb^{T \times N}$ is defined by \\ $\magn{K}_{max} = \displaystyle\max_{1 \leq i \leq T, 1 \leq j \leq N} \bigr \{ \abs{K(i,j)} \bigr \}$. 
\end{definition}
\begin{definition} \label{num_rank1}
    \textbf{Numerical max-rank of a matrix}: For a given $\epsilon > 0$, we define the numerical rank of the matrix $K \in \Rb^{T \times N}$ as follows
    \begin{equation}
        r_{\epsilon} =  \min \bigr \{ r : \magn{K - \tilde{K}}_{max} < \epsilon \magn{K}_{max} \bigr \}
    \end{equation}
\end{definition}
where $\tilde{K} \in \Rb^{T \times N}$ is a rank $r$ matrix. In our theorems, we get upper bound on $r_{\epsilon}$.
\begin{definition} \label{num_rank2}
    \textbf{Numerical rank of a matrix}: For a given $\epsilon > 0$, the $\epsilon$-rank of $K \in \Rb^{T \times N}$, denoted by $p_{\epsilon}$, is defined as follows.
    \begin{equation}
    p_{\epsilon} = \min \{k \in \{1,2,\dots, \min\{T,N\} \} : \sigma_{k} < \epsilon \sigma_1 \}
    \end{equation}
\end{definition}
where $\sigma_1 \geq \sigma_2 \geq \dots \geq \sigma_{\min\{T,N\}} \geq 0 $ are the singular values of $K$. In our numerical illustrations, we plot $p_{\epsilon}$ with $\epsilon = 10^{-12}$.

The main theorems that we prove in \Cref{ddim} is given below.

    \begin{tcolorbox}[colback=white!5!white,colframe=black!75!black,title=Main theorems]
        Let $K$ be the kernel interaction matrix. For a given
        $\delta >0$, there exists a matrix $\Tilde{K}$ with rank $p_{\delta}$ such that $\dfrac{\magn{K - \Tilde{K}}_{max}}{\magn{K}_{max}} < \delta$, where $p_\delta\in \mathcal{O}\bkt{\mathcal{R}\bkt{N}\bkt{ \log^{d-d'}\bkt{\frac{\mathcal{R} \bkt{N}}{\delta}}}}$
        \vspace{0.5cm}
        \begin{theorem}
            If \(X\) and \(Y\) share a vertex (vertex-sharing d-hyper-cubes (domains)) then
            $\boxed{\mathcal{R}(N)= \log_{2^d}(N)}$. In this case $d'=0$. \\
            For example, if $Y := [0,r]^d$ and $X := [-r,0]^d$ then the interaction between $X$ and $Y$ is vertex-sharing interaction. By vertex-sharing, we mean the identical hyper-cubes $X$ and $Y$ of side $r$ are connected by the translation $Y = X + r \bkt{1,1,\dots,1}$, where $\bkt{1,1,\dots,1} \in \mathbb{R}^d$
        \end{theorem}
        \vspace{0.5cm}
        \begin{theorem}
            If \(X\) and \(Y\) share a hyper-surface of dim $d' \in \{1,2,\ldots,d-1\}$ ($d'$-hyper-surface-sharing $d$-hyper-cubes (domains)), then
            $\boxed{\mathcal{R}(N)= N^{d' / d}}$. \\
            For example, if $Y := [0,r]^{d}$ and $X := [-r,0]^{d-d'} \times [0,r]^{d'}$ then the interaction between $X$ and $Y$ is a $d'$-hyper-surface-sharing interaction $d' \in \{1,2,\ldots,d-1\}$. For instance, if $d=3$, then $d'=1$ and $d'=2$, will be the edge-sharing and face-sharing cases in 3$D$, respectively. By $d'$ hyper-surface sharing, we mean the identical hyper-cubes $X$ and $Y$ of side $r$ are connected by the translation $Y = X + r (1,1\dots,1,\overbrace{0,0,\dots,0}^{d'})$.
        \end{theorem}
    \end{tcolorbox}
 We also plot the numerical ranks of the different interactions in different dimensions for the following eight kernel functions using MATLAB.
 \begin{enumerate}
    \item $F_1(\xb,\yb) = 1/r$; \hfill Green's function for $3$D Laplace operator
    \item $F_2(\xb,\yb) = \log(r)$; \hfill Green's function for $2$D Laplace operator
    \item $F_3(\xb,\yb) = \dfrac{\exp\bkt{ir}}{r}$; \hfill Green's function for $3$D Helmholtz operator
    \item $F_4(\xb,\yb) = H_{2}^{(1)}\bkt{r}$; \hfill Hankel function
    \item $F_5(\xb,\yb) = r$; \hfill Poly-harmonic radial basis function
    \item $F_6(\xb,\yb) = \sin\bkt{r}$;
    \item $F_7\bkt{\xb,\yb} = \dfrac1{\sqrt{1+r}}$;
    \item $F_8\bkt{\xb,\yb} = \exp\bkt{-r}$; \hfill Matérn covariance kernel with $\nu=1/2$
\end{enumerate}
 where $r = \magn{\xb-\yb}_2$ with $\xb \in \text{interior}\bkt{X} \text{ and } \yb \in \text{interior} \bkt{Y}$.
 \begin{remark}
     We use \emph{max} norm to prove our theorems. But we only report the numerical rank plots based on the \cref{num_rank2} since this numerical rank definition is commonly used. 
 \end{remark}
 \begin{remark}
     Note that the first four kernels have singularity along the line $\yb=\xb$, while the next four are not analytic along the line $\yb=\xb$.
 \end{remark}
 \begin{remark}
     Note that the all the kernels above satisfy the analytic continuation assumption as stated before.
 \end{remark}
 \begin{remark}
 The kernels $F_3$ and $F_4$ are complex-valued kernels. However, our result is still applicable by looking at the real and imaginary parts of these complex-valued kernels.
 \end{remark}
\section{Rank growth of all the interactions in \texorpdfstring{$d$}{d} dimensions} \label{ddim}
    In this section, we discuss and prove our theorems. We also plot the numerical ranks of different interactions in four dimensions $(d=4)$, three dimensions $(d=3)$, two dimensions $(d=2)$, one dimension $(d=1)$.
    
    \subsection{\textbf{\textit{Rank growth of far-field domains}}} The rank growth of far-field is a well-known result but still, for the sake of completeness we will do it once again. 
        \begin{figure}[H]
    \centering  \subfloat[Far-field nodes in $1$D]{\label{far_inter_1}\resizebox{4cm}{!}{
        \begin{tikzpicture}[scale=0.75]
            \draw[|-|,dashed] (0,0) node[anchor=north] {$-r$} -- (2,0);
            \draw[|-|] (-2,0) node[anchor=north] {$-2r$} -- (0,0);
            \draw[|-|] (2,0) -- (4,0);   
            \node[anchor=north] at (2,0) {$0$};
            \node[anchor=north] at (4,0) {$r$};
            
            \node[anchor=north] at (-1,0.75) {X};
            \node[anchor=north] at (3,0.75) {Y};
        \end{tikzpicture}
        }}
        \qquad \qquad \qquad
    \centering   \subfloat[Far-field squares in $2$D]{\label{far_inter_2}\resizebox{4cm}{!}{
		\begin{tikzpicture}[scale=0.75]
			\draw (-1,-1) rectangle (1,1);
			\draw (3,-1) rectangle (5,1);
			\node at (0,0) {$X$};
			\node at (4,0) {$Y$};
			\draw [<->] (-1,-1.25) -- (1,-1.25);
			\draw [<->] (1,-1.25) -- (3,-1.25);
			\draw [<->] (3,-1.25) -- (5,-1.25);
			\draw [<->] (-1.25,-1) -- (-1.25,1);
			\node at (2,-1.5) {$r$};
			\node at (0,-1.5) {$r$};
			\node at (4,-1.5) {$r$};
			\node at (-1.5,0) {$r$};
		\end{tikzpicture}
		}}
        \qquad \qquad \qquad
        \centering   \subfloat[Far-field cubes in $3$D]{\label{far_inter_3}\resizebox{4cm}{!}{
                    	\begin{tikzpicture}[scale=0.5]
         				\draw (2,2,0)--(0,2,0)--(0,2,2)--(2,2,2)--(2,2,0)--(2,0,0)--(2,0,2)--(0,0,2)--(0,2,2);
         		     	\draw
         		     	(2,2,2)--(2,0,2);
         		     	\draw
         		     	(2,0,0)--(0,0,0)--(0,2,0);
         		     	\draw
         		     	(0,0,0)--(0,0,2);
         			    \node at (1,1,1) {$X$};
         		
                        \draw
                        (4,0,2) rectangle (6,2,2);
                        \draw
                        (4,0,0) rectangle (6,2,0);
                        \draw
                        (6,2,0)--(6,2,2);
                        \draw
                        (6,0,0)--(6,0,2);
                        \draw (4,0,0) -- (4,0,2);
                        \draw (4,2,0) -- (4,2,2);
                        \node at (5,1,1) {$Y$};
         			\end{tikzpicture}
        }}
        \caption{Far-field interaction in $d=1,2,3$}
    \end{figure}
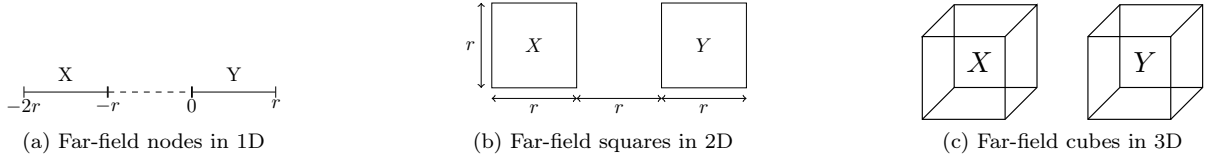
    \begin{lemma} \label{far_th}
        Let $X$ and $Y$ be at least one hyper-cube away and $K$ be the corresponding interaction matrix then for a given
        $\delta >0$, there exists a matrix $\Tilde{K}$ with rank $p_{\delta}$ such that $\dfrac{\magn{K  - \Tilde{K}}_{max}}{\magn{K}_{max}} < \delta$, where $p_\delta\in \mathcal{O}\bkt{\bkt{ \log \bkt{\frac{1}{\delta}}}^d}$
    \end{lemma}
    \begin{proof}
            If \(X\) and \(Y\) are at least one hyper-cube away, i.e., $\text{dist}\bkt{X,Y} \geq \text{diam}(Y)$ then the interaction between them is called far-field interactions. The far-field interaction in different dimensions is shown in~\cref{far_inter_1,far_inter_2,far_inter_3}. Let $Y = [0,r]^{d}$ and $X = [-2r,-r] \times [0,r]^{d-1}$ be two hyper-cubes, which are one hyper-cube away. We choose a $\bkt{p+1}^d$ tensor grid of Chebyshev nodes inside the hyper-cube $Y$ to interpolate $F\bkt{\xb,\yb}$ on the Chebyshev grid along $\yb$. Let $\Tilde{K}$ be the approximation of the matrix \(K\) and $M_i = \displaystyle \sup_{\yb \in \mathcal{B}\bkt{{Y,\rho}}} \abs{F_a\bkt{\xb_i,\yb}}$ ,  $\rho \in (1,\alpha)^d$, for some $\alpha >1$. Then by~\cref{eq1} the absolute error along the $i^{th}$ row is given by (MATLAB notation)
         \begin{equation}\label{far_d}
             \abs{\bkt{K \bkt{i,:} -\Tilde{K} \bkt{i,:}}} \leq 4 M_i V_d \frac{\rho^{-p}}{\rho -1}, \qquad 1 \leq i \leq N
         \end{equation}
         Let $M = \displaystyle \max_{1 \leq i\leq N} \{M_i\}$. Therefore,
         \begin{equation}
             \magn{K  - \Tilde{K}}_{max} \leq 4 M V_d \dfrac{\rho^{-p}}{\rho -1} 
             \implies \dfrac{\magn{K  - \Tilde{K}}_{max}}{\magn{K}_{max}} \leq \dfrac{4M V_d}{\magn{K}_{max}} \dfrac{\rho^{-p}}{\rho -1} = \dfrac{c\rho^{-p}}{\rho -1}
         \end{equation}
    where $c = \dfrac{4M V_d}{\magn{K}_{max}}$. Now, choosing $p$ such that the above relative error to be less than $\delta$ (for some $\delta >0$), we obtain
        $p = \ceil{\frac{ \log \bkt{\frac{c}{\delta (\rho-1)}}}{\log(\rho)}} \implies \dfrac{\magn{K  - \Tilde{K}}_{max}}{\magn{K}_{max}} < \delta$.
     Since, the rank of the matrix $\tilde{K}$ is $\bkt{p+1}^d$, i.e., the rank of $\tilde{K}$ to be bounded above by
        $\bkt{1 + \ceil{\frac{ \log \bkt{\frac{c}{\delta (\rho-1)}}}{\log(\rho)}}}^d$.
    Therefore, the rank of $\Tilde{K}$ scales $\mathcal{O}\bkt{ \log^d \bkt{\frac{1}{\delta}}}$ with $ \dfrac{\magn{K  - \Tilde{K}}_{max}}{\magn{K}_{max}} < \delta$.
    \end{proof}
     The numerical rank growth of the far-field interaction in $4$D, $3$D, $2$D and $1$D of the eight functions as described in~\Cref{Preliminaries} is plotted in~\Cref{fig:far_field} and tabulated in~\cref{tab:res_4d_far_log,tab:res_3d_far_log,tab:res_2d_far_log,tab:res_1d_far_log}.
\begin{figure}[H]
    \centering
    \subfloat[Numerical rank growth of $4$D far-field]{\includegraphics[height=3cm, width=5.5cm]{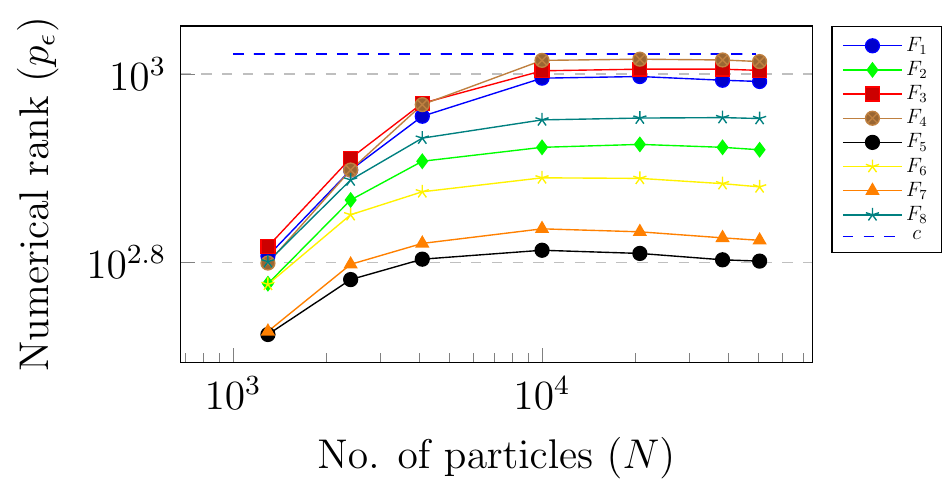}\label{fig:res_4d_far_log}}\qquad \qquad%
    \subfloat[Numerical rank growth of $3$D far-field]{\includegraphics[height=3cm, width=5.5cm]{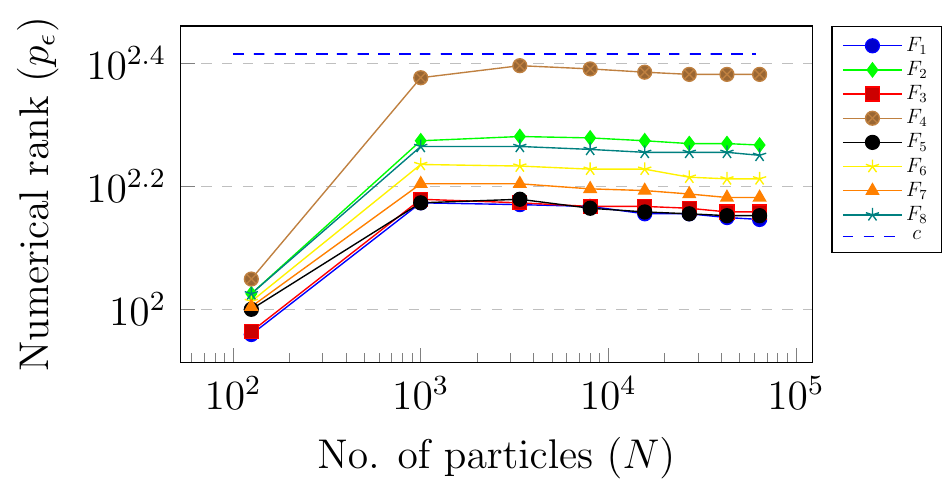}\label{fig:res_3d_far_log}}\qquad \qquad%
    \subfloat[Numerical rank growth of $2$D far-field]{\includegraphics[height=3cm, width=5.5cm]{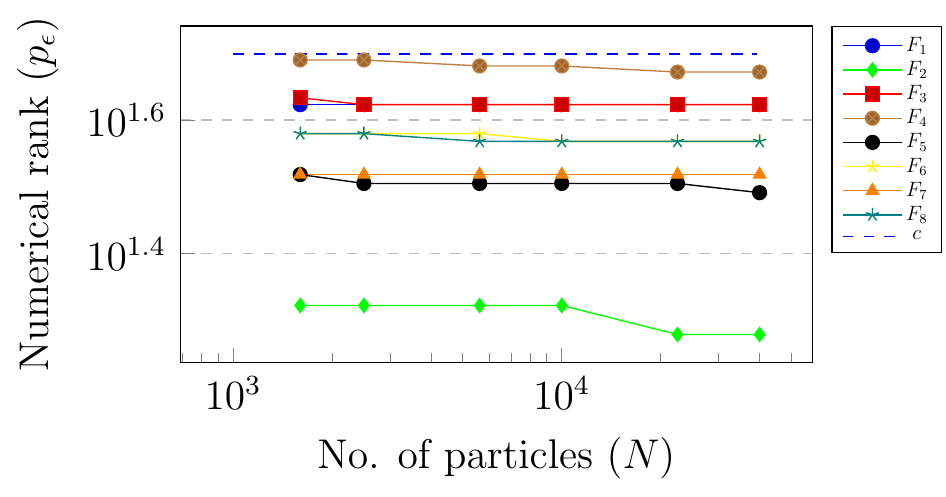}\label{fig:res_2d_far_log}}\qquad \qquad%
    \subfloat[Numerical rank growth of $1$D far-field]{\includegraphics[height=3cm, width=5.5cm]{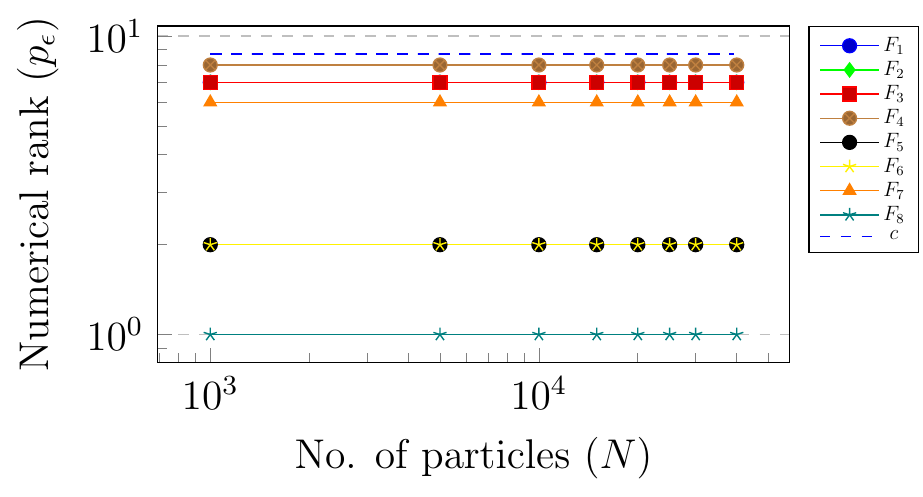}\label{fig:res_1d_far_log}}\qquad%
    \caption{Numerical rank growth with $N$ in different dimension for the far-field interaction}
    \label{fig:far_field}
\end{figure}
\subsection{Rank growth of vertex-sharing domains}
        We are going to prove one of our main theorems on rank growth of \textbf{vertex-sharing} interaction in $d$ dimension. 
           \begin{figure}[H]
        \centering \subfloat[vertex-sharing domain]{\label{1d_ve}\resizebox{4cm}{!}{
        \begin{tikzpicture}
            (2,0)  -- (4,0);
        \draw[|-|]  (0,0) -- (2,0);   
        \draw[|-|]  (4,0) -- (2,0); 
        \node[anchor=north] at (2,0) {$0$};
        \node[anchor=north] at (4,0) {$r$};
        \node[anchor=north] at (0,0) {$-r$};
        \node[anchor=north] at (1,0.7) {X};
        \node[anchor=north] at (3,0.7) {Y};
        \end{tikzpicture}
        }}
        \qquad
        \centering \subfloat[Hierarchical subdivision at level 1]{\label{1d_ver_sub_l1}\resizebox{4cm}{!}{
        \begin{tikzpicture}
        \draw[ |-|,
                decoration={markings,
                    mark= at position 0.5 with {\arrow{|}}
                },
                postaction={decorate}
            ]
            (2,0)  -- (4,0);
        \draw[|-|]  (0,0) -- (2,0);   
        \node[anchor=north] at (2,0) {$0$};
        \node[anchor=north] at (4,0) {$r$};
        \node[anchor=north] at (0,0) {$-r$};
        \node[anchor=north] at (3,0) {$\frac{r}{2}$};
        \node[anchor=north] at (1,0.7) {X};
        \node[anchor=north] at (2.5,0.7) {$Y_2$};
        \node[anchor=north] at (3.5,0.7) {$Y_1$};
        \end{tikzpicture}
        }}
        \qquad
        \centering \subfloat[Hierarchical subdivisions]{\label{1d_ver_sub}\resizebox{4cm}{!}{
        \begin{tikzpicture}
        \draw[ |-|,
                decoration={markings,
                    mark= at position 0.0625 with {\arrow{|}},
                    mark= at position 0.125 with {\arrow{|}},
                    mark= at position 0.25 with {\arrow{|}},
                    mark= at position 0.5 with {\arrow{|}}
                },
                postaction={decorate}
            ]
            (2,0)  -- (4,0);
        \draw[|-|]  (0,0) -- (2,0);   
        \node[anchor=north] at (2,0) {$0$};
        \node[anchor=north] at (4,0) {$r$};
        \node[anchor=north] at (0,0) {$-r$};
        \node[anchor=north] at (3,0) {$\frac{r}{2}$};
        \node[anchor=north] at (2.5,0) {$\frac{r}{4}$};
        \node[anchor=north] at (1,0.7) {X};
        \end{tikzpicture}
        }}
        \caption{vertex-sharing domain and subdivisions in $1$D}
        \label{1d_ver}
    \end{figure}
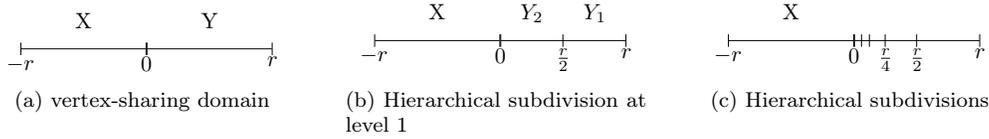 
\begin{figure}[H]
    \centering \subfloat[vertex-sharing squares]{\label{2d_vertex}\resizebox{3.8cm}{!}{
        \begin{tikzpicture}
            \draw (-1,-1) rectangle (1,1);
            \draw [<->] (-1,-1.25) -- (1,-1.25);
            \node at (0,-1.5) {$r$};
            \draw [<->] (-1.25,-1) -- (-1.25,1);
            \node at (-1.5,0) {$r$};
            \draw [<->] (1,.75) -- (3,.75);
            \node[anchor=south] at (2,.25) {$r$};
            \draw [<->] (3.25,1) -- (3.25,3);
            \node at (3.5,2) {$r$};
            \draw (1,1) rectangle (3,3);
            \node at (0,0) {$X$};
         	\node at (2,2) {$Y$};
            \end{tikzpicture}
            }}
        \qquad
        \centering \subfloat[subdivision at level 1]{\label{2d_vertex_divide_level1}\resizebox{3.8cm}{!}{
        \begin{tikzpicture}
            \draw (-1,-1) rectangle (1,1);
            \node at (0,0) {$X$};
            
            \draw (1,1) rectangle (3,3);
            \draw (1,1) grid (3,3);
            \node at (2.5,1.5) {\small $Y_{1,3}$};
            \node at (1.5,2.5) {\small $Y_{1,1}$};
            \node at (2.5,2.5) {\small $Y_{1,2}$};
            \node at (1.5,1.5) {\small $Y_{2}$};
            \draw [<->] (-1,-1.25) -- (1,-1.25);
            \node at (0,-1.5) {$r$};
            \draw [<->] (-1.25,-1) -- (-1.25,1);
            \node at (-1.5,0) {$r$};
            \draw [<->] (1,.75) -- (3,.75);
            \node[anchor=south] at (2,.25) {$r$};
            \draw [<->] (3.25,1) -- (3.25,3);
            \node at (3.5,2) {$r$};
            
            \end{tikzpicture}
            }}
    \qquad 
    \centering \subfloat[Hierarchical subdivisions]{\label{2d_vertex_divide}\resizebox{3.8cm}{!}{
        \begin{tikzpicture}
            \draw (-1,-1) rectangle (1,1);
            \node at (0,0) {$X$};
            \draw (1,1) rectangle (3,3);
            \draw (1,1) grid (3,3);
            \draw[step=0.5] (1,1) grid (2,2);
            \node at (2.5,1.5) {\small $Y_{1,3}$};
            \node at (1.5,2.5) {\small $Y_{1,1}$};
            \node at (2.5,2.5) {\small $Y_{1,2}$};
            \draw[step=0.25] (1,1) grid (1.5,1.5);
            \node at (1.75,1.75) {\tiny $Y_{2,2}$};
            \node at (1.75,1.25) {\tiny $Y_{2,1}$};
            \node at (1.25,1.75) {\tiny $Y_{2,3}$};
            \draw[step=0.125] (1,1) grid (1.25,1.25);
            \draw [<->] (-1,-1.25) -- (1,-1.25);
            \node at (0,-1.5) {$r$};
            \draw [<->] (-1.25,-1) -- (-1.25,1);
            \node at (-1.5,0) {$r$};
            \draw [<->] (1,.75) -- (3,.75);
            \node[anchor=south] at (2,.25) {$r$};
            \draw [<->] (3.25,1) -- (3.25,3);
            \node at (3.5,2) {$r$};
            
            \end{tikzpicture}
            }}
            \caption{vertex-sharing squares and its subdivisions in $2$D}
            \label{2d_ver}
\end{figure}
\begin{figure}[H]
    \centering
    \subfloat[subdivision at level $1$]{\resizebox{3.2cm}{!}{
                \tdplotsetmaincoords{75}{210}
            \begin{tikzpicture}[tdplot_main_coords]
                \cubex{0}{0}{0}{2}{line width=0.5mm,black}{};
	            \cubex{1}{1}{1}{-1.2}{draw=none,line width=0.5mm,black}{$\color{blue}Y_1$};
	            \cubex{2}{2}{2}{-1}{line width=0.25mm,black!60}{$Y_2$};
                \cubex{2}{2}{2}{2}{line width=0.5mm,red}{$X$};
         	\end{tikzpicture}
        }}
        \qquad \qquad
    \subfloat[subdivision at level $2$]
    {\resizebox{3.2cm}{!}{
	        \label{3d_v} 
	           \tdplotsetmaincoords{75}{210}
	        \begin{tikzpicture}[tdplot_main_coords]
            	\cubex{0}{0}{0}{2}{line width=0.5mm,black}{};
	            \cubex{1}{1}{1}{-1.2}{draw=none,line width=0.5mm,black}{$\color{blue}Y_1$};
	            \cubex{2}{2}{2}{-1}{line width=0.25mm,black!60}{};
	            \cubex{2}{2}{2}{-0.5}{line width=0.15mm,black!60}{$Y_3$};
                \cubex{2}{2}{2}{2}{line width=0.5mm,red}{$X$};
            \end{tikzpicture}
        }
    }
    \caption{vertex-sharing cubes and its subdivisions in $3$D}
    \label{3d_ver}
\end{figure}   
        \begin{theorem} \label{ver_th}
            Let $X$ and $Y$ be hyper-cubes that share a vertex and $K$ be the corresponding interaction matrix, then for a given
            $\delta >0$, there exists a matrix $\Tilde{K}$ with rank $p_{\delta}$ such that $\dfrac{\magn{K  - \Tilde{K}}_{max}}{\magn{K}_{max}} < \delta$, where $p_\delta\in \mathcal{O}\bkt{\log_{2^d}(N) \log^d \bkt{\frac{ \log_{2^d}(N)}{\delta}}}$
        \end{theorem}

        \begin{proof}
                    Consider two vertex-sharing $d$-hyper-cubes $Y = [0,r]^d$ and $X = [-r,0]^d$. We subdivide the hyper-cube \(Y\) hierarchically using an adaptive $2^d$ tree as shown in~\cref{1d_ver,2d_ver,3d_ver}. \href{https://sites.google.com/view/dom3d/vertex-sharing-domains/}{\textbf{The link here provides a better $3$D view}}. At level 1, we have total $2^d$ finer hyper-cubes. We again recursively subdivide the finer hype-cube adjacent to the vertex using the $2^d$ tree until the finest hyper-cube ($Y_{\kappa+1}$), which shares the vertex with $X$ contains one particle. This gives,
        \begin{equation}
             Y = \overbrace{Y_2 \bigcup_{j=1}^{2^{d}-1} Y_{1,j}}^{\text{At level 1}} = \overbrace{\left[0, \dfrac{r}{2} \right]^d \bigcup_{j=1}^{2^{d}-1} Y_{1,j}}^{\text{At level 1}}   = \cdots =   \overbrace{\left[0, \dfrac{r}{2^{\kappa}} \right]^d\bigcup_{k=1}^{\kappa} \bigcup_{j=1}^{2^{d}-1} Y_{k,j}}^{\text{At level } \kappa} = \overbrace{Y_{\kappa + 1} \bigcup_{k=1}^{\kappa} \bigcup_{j=1}^{2^{d}-1} Y_{k,j}}^{\text{At level } \kappa}
             \label{eqn_vertex_subdivide}
        \end{equation}
        where $Y_{k}$ to be the subdivision of domain $Y$ at level $k$ in the tree and $Y_{k,j}$ be the $j^{th}$ subdivision of $Y_{k}$ $\bkt{Y_{k,j} \text{ is a child of } Y_k}$ with $Y_{a,r} \bigcap Y_{b,s} \neq \emptyset$ iff $a=b$ $\&$ $r=s$. We have $Y = \displaystyle Y_{\kappa + 1} \bigcup_{k=1}^{\kappa} \bigcup_{j=1}^{2^{d}-1} Y_{k,j}$, \quad $Y_{\kappa + 1} = Y \setminus \displaystyle \bigcup_{k=1}^{\kappa} Y_{k}$ with $Y_{\kappa +1}$ having one particle and $\kappa \sim \log_{2^d} \bkt{N}$. Further, $\text{dist}(X,Y_k) >0$ for $k \in \{1,2,\dots, \kappa\}$ and $\text{dist}(X,Y_{\kappa+1}) =0$.  
        Let us define the matrix $K_{k} \in \Rb^{T \times N}$ by (MATLAB notaion):
            $$K_{k}(u,v) = \begin{cases}
            K(u,v) & \text{ if }\yb_v \in Y_{k}\\
            0 & \text{ otherwise}
            \end{cases}$$
        where $k \in \{1,2,\hdots,\kappa, \kappa +1 \}$
    and the matrix $K_{k,j} \in \Rb^{T \times N}$ by:
    $$K_{k,j}(u,v) = \begin{cases}
    K(u,v) & \text{ if }\yb_v \in Y_{k,j}\\
    0 & \text{ otherwise}
    \end{cases}$$
    where $k \in \{1,2,\hdots,\kappa \}$ is the level and $j$ is indicative of the $j^{th}$ subdivision at level $k$.
        Let $\Tilde{K}_{k,j}$ be the approximation of the matrix ${K_{k,j}}$, then approximation of matrix $K$ is given by
        \begin{equation}
             \Tilde{K} = \sum_{k=1}^{\kappa} \sum_{j=1}^{2^d -1} \Tilde{K}_{k,j} + K_{\kappa+1}
        \end{equation}
        We choose a $\bkt{p_{k,j}+1}^d$ tensor product grid on Chebyshev nodes in the hyper-cube $Y_{k,j}$ to obtain the polynomial interpolation of $F(\xb,\yb)$ along $\yb$. Let $M^{(i)}_{k,j} = \displaystyle \sup_{\yb \in \mathcal{B} \bkt{{Y_{k,j},\rho}}} \abs{F_a\bkt{\xb_i,\yb}}$,  $\rho \in (1,\alpha)^d$, for some $\alpha > 1$. Then using~\cref{eq1} the absolute error along the $i^{th}$ row of the matrix $K_{k,j}$ is given by (MATLAB notation)
        \begin{equation}\label{ver_d}
             \Bigl\lvert \bkt{K_{k,j} \bkt{i,:} -\Tilde{K}_{k,j} \bkt{i,:}} \Bigl\lvert \leq 4 M^{(i)}_{k,j} V_d \frac{\rho^{-p_{k,j}}}{\rho -1}, \qquad 1 \leq i \leq N, \quad 1 \leq k \leq \kappa, \quad 1 \leq j \leq 2^d-1
        \end{equation}
        Let $M = \max \{ M^{(i)}_{k,j} \} $. Therefore,
        \begin{equation}
            \magn{\bkt{K_{k,j} -\Tilde{K}_{k,j}}}_{max} \leq 4 M V_d \frac{\rho^{-p_{k,j}}}{\rho -1} \implies \dfrac{\magn{\bkt{K_{k,j} -\Tilde{K}_{k,j}}}_{max}}{\magn{K}_{max}} \leq \dfrac{4 M V_d}{\magn{K}_{max}} \frac{\rho^{-p_{k,j}}}{\rho -1} = \frac{c\rho^{-p_{k,j}}}{\rho -1}
        \end{equation}
    where $c = \dfrac{4 M V_d}{\magn{K}_{max}}$. Now choosing $p_{k,j}$ such that the above error is less than $\delta_1$ (for some $\delta_1>0$), we obtain $p_{k,j} = \ceil{\frac{ \log \bkt{\frac{c}{\delta_1 (\rho-1)}}}{\log(\rho)}} \implies \dfrac{\magn{\bkt{K_{k,j} -\Tilde{K}_{k,j}}}_{max}}{\magn{K}_{max}} < \delta_1$ 
        with rank of $\Tilde{K}_{k,j} $ is $ (1 + p_{k,j})^d$. \\ Let $p_{l,m} = \max \{p_{k,j} : k=1,2,\hdots, \kappa \text{ and } j = 1,2,\hdots, 2^d -1 \}$, which corresponds to the matrix $\Tilde{K}_{l,m}$. So, at level $k$ the absolute error of matrix $K_k$ is
        \begin{align}
            \magn{\bkt{K_{k} -\Tilde{K}_{k}}}_{max} =\magn{ \sum_{j=1}^{2^d -1}  \bkt{K_{k,j} -\Tilde{K}_{k,j}}}_{max} \leq  \sum_{j=1}^{2^d -1} \magn{ \bkt{K_{k,j}-\Tilde{K}_{k,j}}}_{max} < \bkt{2^d -1} \delta_1 \magn{K}_{max}
        \end{align}
        with rank of $\Tilde{K}_{k} $ is bounded above by $ \bkt{2^d -1} (1 + p_{l,m})^d$.
        Therefore, we get 
        \begin{align}
            \begin{split}
                \magn{\bkt{K - \Tilde{K}}}_{max} =\magn{ \bkt{K_1+ K_2+\cdots + K_\kappa} - \bkt{\Tilde{K_1}+\Tilde{K_2}+\cdots + \Tilde{K_\kappa}}}_{max} \\
            \leq \dsum_{k=1}^{\kappa} \magn{\bkt{K_k -\Tilde{K}_k }}_{max} < \bkt{2^d -1} \kappa \delta_1 \magn{K}_{max}
            \end{split}
        \end{align}
        with the rank of $\Tilde{K} $ is bounded above by $ 1+ \bkt{2^d -1} \kappa (1 + p_{l,m})^d = 1+ \bkt{2^d -1} \kappa \bkt{1 + \ceil{\frac{ \log \bkt{\frac{c}{\delta_1 (\rho-1)}}}{\log(\rho)}}}^d$.
        Note that, the rank of $K_{\kappa +1}$ is one. If we set $\delta_1 = \dfrac{\delta}{\bkt{2^d -1} \kappa}$, then $\dfrac{ \magn{K - \Tilde{K}}_{max}}{\magn{K}_{max}} < \delta$ with rank of $\Tilde{K}$ is bounded above by $ 1 + \bkt{2^d -1} \log_{2^d}(N) \bkt{1+\ceil{\frac{ \log \bkt{\frac{c \bkt{2^d -1} \log_{2^d}(N)}{\delta (\rho-1)}}}{\log(\rho)}}}^d$.
        
        Hence, the rank of $\Tilde{K}$ scales $\mathcal{O}\bkt{\log_{2^d}(N) \log^d \bkt{\frac{ \log_{2^d}(N)}{\delta}}}$ with $ \dfrac{\magn{K - \Tilde{K}}_{max}}{\magn{K}_{max}} < \delta$. 
        \end{proof}
The numerical rank plots of the $4$D, $3$D, $2$D and $1$D vertex-sharing interaction of the eight functions as described in~\Cref{Preliminaries} are shown in~\cref{fig:vertex_rank} and tabulated in~\cref{tab:res_4d_ver_log,tab:res_3d_ver_log,tab:res_2d_ver_log,tab:res_1d_vert_log}.
\begin{figure}[H]
    \centering
    \subfloat[Numerical rank growth of $4$D vertex-sharing]{\includegraphics[height=3cm, width=5.5cm]{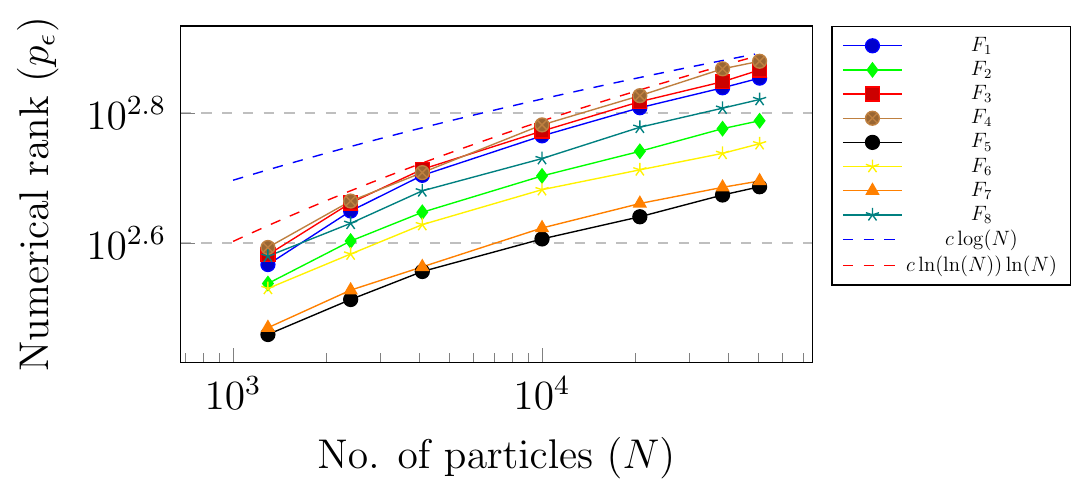}\label{fig:res_4d_ver_log}}\qquad \qquad%
    \subfloat[Numerical rank growth of $3$D vertex-sharing]{\includegraphics[height=3cm, width=5.5cm]{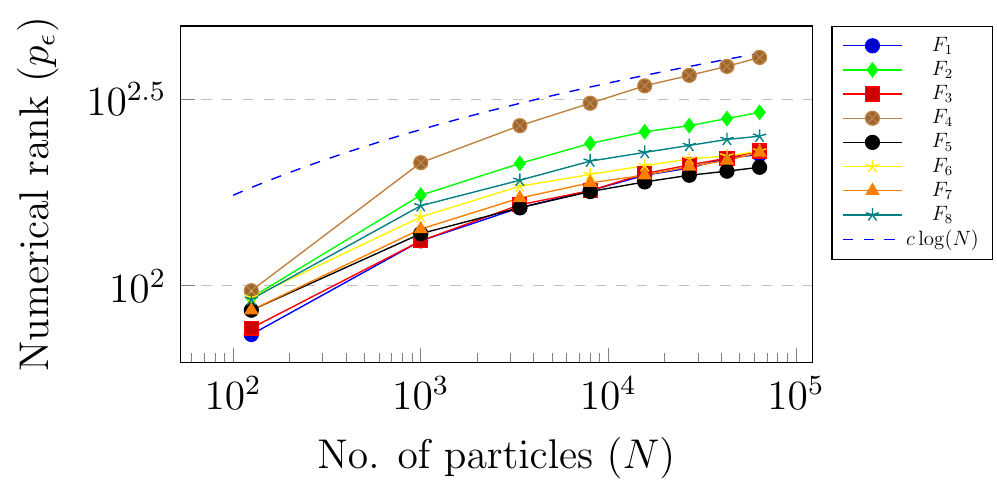}\label{fig:res_3d_ver_log}}\qquad \qquad%
    \subfloat[Numerical rank growth of $2$D vertex-sharing]{\includegraphics[height=3cm, width=5.5cm]{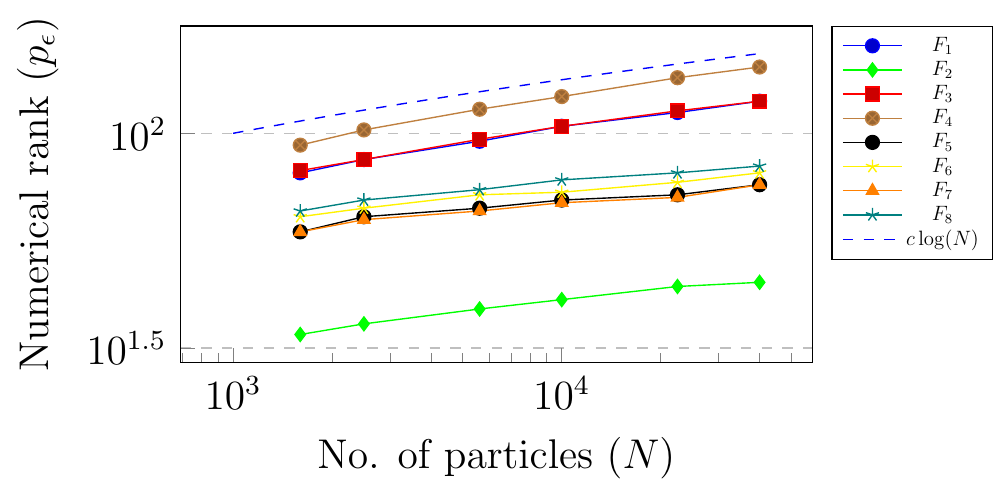}\label{fig:res_2d_ver_log}}\qquad \qquad%
    \subfloat[Numerical rank growth of $1$D vertex-sharing]{\includegraphics[height=3cm, width=5.5cm]{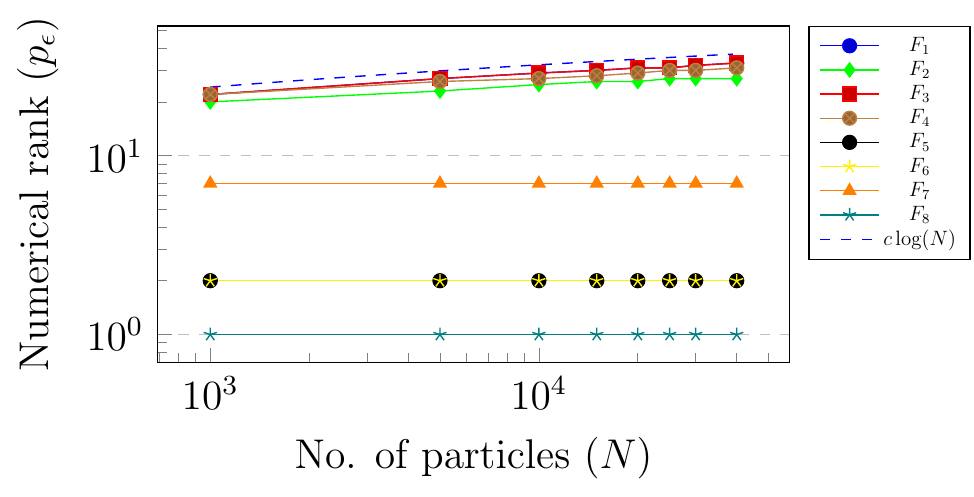}\label{fig:res_1d_ver_log}}\qquad %
    \caption{Numerical rank growth with $N$ in different dimension for the vertex-sharing interaction}
    \label{fig:vertex_rank}
\end{figure}

    \subsection{\textit{Rank growth of different hyper-surface interaction}} In this sub-section, we prove our other theorem on the rank growth of \textbf{all} the near fields (other than vertex-sharing) interaction.
    \begin{figure}[H]
    \centering
    \subfloat[Edge sharing squares]{\label{2d_ed}\resizebox{4cm}{!}{
                    \begin{tikzpicture}[scale=1.2]
         				\draw (-1,-1) rectangle (1,1);
         				\draw (1,-1) rectangle (3,1);
         				\node at (0,0) {$X$};
         				\node at (2,0) {$Y$};
         				\draw [<->] (-1,-1.25) -- (1,-1.25);
         				\node at (0,-1.5) {$r$};
         				\draw [<->] (1,-1.25) -- (3,-1.25);
         				\node at (2,-1.5) {$r$};
         				\draw [<->] (-1.25,-1) -- (-1.25,1);
         				\node[anchor=east] at (-1.25,0) {$r$};
         			\end{tikzpicture}
        }}
    \qquad
    \subfloat[subdivision at level 1]
    {\resizebox{4cm}{!}{
	        \label{2d_edge_divide_l1} 
	        \begin{tikzpicture}[scale=1.2]
         				\draw (-1,-1) rectangle (1,1);
         				\draw (1,-1) rectangle (3,1);
         				
         				\draw (1,-1) grid (3,1);
         				\node at (2.5,-0.5) {$Y_{1,1}$};
         				\node at (2.5,0.5) {$Y_{1,2}$};
         				\node at (1.5,0) {$Y_{2}$};
         				
         				   
         				


 
         				\node at (0,0) {\small $X$};
         				\draw [<->] (-1,-1.25) -- (1,-1.25);
         				\node at (0,-1.5) {$r$};
         				\draw [<->] (1,-1.25) -- (3,-1.25);
         				\node at (2,-1.5) {$r$};
         				\draw [<->] (-1.25,-1) -- (-1.25,1);
         				\node[anchor=east] at (-1.25,0) {$r$};
            \end{tikzpicture}
        }
    }
        \qquad 
    \subfloat[Hierarchical subdivision]
    {\resizebox{4cm}{!}{
	        \label{2d_edge_divide} 
	        \begin{tikzpicture}[scale=1.2]
         				\draw (-1,-1) rectangle (1,1);
         				\draw (1,-1) rectangle (3,1);
         				
         				\draw (1,-1) grid (3,1);
         				\node at (2.5,-0.5) {$Y_{1,1}$};
         				\node at (2.5,0.5) {$Y_{1,2}$};
         				
         				\foreach \i in {0,1,2,3}
         				    \draw (1.5,-1+0.5*\i) rectangle (2,-0.5+0.5*\i);
         				   
         				\node at (1.75,-0.75) {\small $Y_{2,1}$};
         				\node at (1.75,-0.25) {\small $Y_{2,2}$};
         				\node at (1.75,0.25) {\small $Y_{2,3}$};
         				\node at (1.75,0.75) {\small $Y_{2,4}$};
         				
         				\foreach \i in {0,1,2,...,7}
         				    \draw (1.25,-1+0.25*\i) rectangle (1.5,-0.75+0.25*\i);

                        \foreach \i in {0,1,2,...,15}
         				    \draw (1.125,-1+0.125*\i) rectangle (1.25,-0.875+0.125*\i);

                        \foreach \i in {0,1,2,...,31}
         				    \draw (1.0625,-1+0.0625*\i) rectangle (1.125,-0.9375+0.0625*\i);
 
         				\node at (0,0) {\small $X$};
         				\draw [<->] (-1,-1.25) -- (1,-1.25);
         				\node at (0,-1.5) {$r$};
         				\draw [<->] (1,-1.25) -- (3,-1.25);
         				\node at (2,-1.5) {$r$};
         				\draw [<->] (-1.25,-1) -- (-1.25,1);
         				\node[anchor=east] at (-1.25,0) {$r$};
            \end{tikzpicture}
        }
    }
    \caption{Edge sharing boxes and its subdivision in $2$D, $d=2, d'=1$}
    \label{2d_edge}
\end{figure}
    \begin{figure}[H]
    \centering
    \subfloat[at level $1$]{\resizebox{5cm}{!}{
    \tdplotsetmaincoords{75}{215}
    \begin{tikzpicture}[tdplot_main_coords]
	\cubex{0}{0}{0}{2}{line width=0.5mm,black}{};
	\foreach \i in {0,1}
		\cubex{0}{1*\i}{0}{1}{line width=0.25mm,black!40}{};
	\foreach \i in {0,1}
		\cubex{0}{1*\i}{1}{1}{line width=0.25mm,black!40}{};
	\cubex{1}{0}{1}{1}{line width=0.25mm,black!40}{};
	\cubex{1}{0}{0}{1}{line width=0.25mm,black!40}{};
    \cubex{2}{2}{0}{2}{line width=0.5mm,red}{$X$};
    \end{tikzpicture}
    }}
    \caption{Edge sharing cubes and its subdivision in $3$D, $d=3,d'=1$}
    \label{3d_edg}
\end{figure}
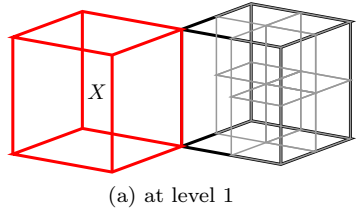
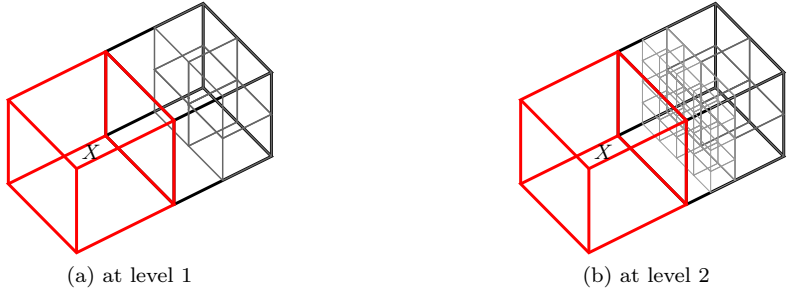
\begin{figure}[H]
    \centering
    \subfloat[at level $1$]{\resizebox{4cm}{!}{
    \tdplotsetmaincoords{45}{145}
    \begin{tikzpicture}[tdplot_main_coords]
	\cubex{0}{0}{0}{2}{line width=0.5mm,black};

    \cubex{0}{0}{0}{1}{line width=0.25mm,black!60}{};
    \cubex{0}{1}{0}{1}{line width=0.25mm,black!60}{};
    \cubex{0}{1}{1}{1}{line width=0.25mm,black!60}{};
    \cubex{0}{0}{1}{1}{line width=0.25mm,black!60}{};


    \cubex{2}{0}{0}{2}{line width=0.5mm,red}{$X$};
    
    \end{tikzpicture}
        }}
        \qquad \qquad \qquad \qquad
    \subfloat[at level $2$]
    {\resizebox{4cm}{!}{
	 \tdplotsetmaincoords{45}{145}
    \begin{tikzpicture}[tdplot_main_coords]
	\cubex{0}{0}{0}{2}{line width=0.5mm,black};

    \cubex{0}{0}{0}{1}{line width=0.25mm,black!60}{};
    \cubex{0}{1}{0}{1}{line width=0.25mm,black!60}{};
    \cubex{0}{1}{1}{1}{line width=0.25mm,black!60}{};
    \cubex{0}{0}{1}{1}{line width=0.25mm,black!60}{};

	\foreach \i in {0,1,2,3}
		\cubex{1}{0.5*\i}{0}{0.5}{line width=0.15mm,black!40}{};
	\foreach \i in {0,1,2,3}
		\cubex{1}{0.5*\i}{1}{0.5}{line width=0.15mm,black!40}{};
	\foreach \i in {0,1,2,3}
		\cubex{1}{0.5*\i}{0.5}{0.5}{line width=0.15mm,black!40}{};
	\foreach \i in {0,1,2,3}
		\cubex{1}{0.5*\i}{1.5}{0.5}{line width=0.15mm,black!40}{};

    \cubex{2}{0}{0}{2}{line width=0.5mm,red}{$X$};
    
    \end{tikzpicture}
        }
    }
    \caption{face-sharing cubes and its subdivision in $3$D, $d=3,d'=2$}
    \label{3d_fac}
\end{figure}  
    \begin{theorem} \label{hyper_th}
           Let $X$ and $Y$ be two hyper-cubes that share a hyper-surface of dim $d'$ and $K$ be the corresponding interaction matrix, then for a given
            $\delta >0$, there exists a matrix $\Tilde{K}$ with rank $p_{\delta}$ such that $\dfrac{\magn{K  - \Tilde{K} }_{max}}{\magn{K}_{max}} < \delta$, where $p_\delta\in \mathcal{O}\bkt{N^{d'/d} \bkt{\log^{d-d'}\bkt{\frac{N}{\delta}}}}$
    \end{theorem}

    \begin{proof}
            Consider two identical hyper-cubes $Y = [0,r]^{d}$ and $X = [-r,0]^{d-d'} \times [0,r]^{d'}$ of dim \(d\) which share a hyper-surface of dim \(d'\). In contrast, if $d=3$ then $d' = 1$ implies \emph{edge-sharing blocks} (edge of a cube is a line segment which has \emph{dim 1}) and $d' = 2$ implies \emph{face-sharing blocks} (face of a cube is a square plane which has \emph{dim 2}). In the previous sub-section, we have discussed the case $d' = 0,  i.e., \emph{ the vertex-sharing case}.$ Now, here we discuss the cases $1 \leq d' \leq d-1$. We subdivide the hyper-cube \(Y\) hierarchically using an adaptive $2^d$ tree as shown in~\cref{2d_edge,3d_edg,3d_fac}. \href{https://sites.google.com/view/dom3d/}{\textbf{The link here provides a better $3$D view}}. At level $1$, we have $\bkt{2^d - 2^{d'}}$ finer hyper-cubes that do not share the \(d'\) dimensional hyper-surface with $X$. We recursively subdivide the finer hyper-cubes adjacent to $X$ using the $2^d$ tree. In general, at level $k$, we have $\bkt{2^d-{2^{d'}}}\bkt{2^{d'}}^{k-1}$ finer hyper-cubes that do not share the \(d'\) dimensional hyper-surface with $X$. So, we can write $Y$ as
    \begin{equation}
        Y = \overbrace{Y_2  \bigcup_{j=1}^{\bkt{2^d-{2^{d'}}}} Y_{1,j}}^{\text{At level 1}} = \dots = \overbrace{Y_{\kappa + 1} \bigcup_{k=1}^{\kappa} \bigcup_{j=1}^{\bkt{2^d-{2^{d'}}}\bkt{2^{d'}}^{k-1}} Y_{k,j}}^{\text{At level } \kappa} = \overbrace{Y_{\kappa + 1} \bigcup_{k=1}^{\kappa} \bigcup_{j=1}^{\bkt{2^{d-d'}-1}\bkt{2^{d'}}^k} Y_{k,j}}^{\text{At level } \kappa}
    \end{equation}
            where $Y_{k}$ to be the subdivision of domain $Y$ at level $k$ in the tree and $Y_{k,j}$ be the $j^{th}$ subdivision of $Y_{k}$ $\bkt{Y_{k,j} \text{ is a child of } Y_k}$ with $Y_{a,r} \bigcap Y_{b,s} \neq \emptyset$ iff $a=b$ $\&$ $r=s$. We have $Y = \displaystyle Y_{\kappa + 1} \bigcup_{k=1}^{\kappa} \bigcup_{j=1}^{\bkt{2^{d-d'}-1}\bkt{2^{d'}}^k} Y_{k,j}$, \quad $Y_{\kappa + 1} = Y \setminus \displaystyle \bigcup_{k=1}^{\kappa} Y_{k}$ with $Y_{\kappa +1}$ having $N^{d'/d}$ particles and $\kappa \sim \log_{2^d} \bkt{N}$. Further, $\text{dist}(X,Y_k) >0$ for $k \in \{1,2,\dots, \kappa\}$ and $\text{dist}(X,Y_{\kappa+1}) =0$.
        Let us define the matrix $K_{k} \in \Rb^{T \times N}$ by (MATLAB notaion):
            $$K_{k}(u,v) = \begin{cases}
            K(u,v) & \text{ if }\yb_v \in Y_{k}\\
            0 & \text{ otherwise}
            \end{cases}$$
        where $k \in \{1,2,\hdots,\kappa, \kappa +1 \}$
    and the matrix $K_{k,j} \in \Rb^{T \times N}$ by:
    $$K_{k,j}(u,v) = \begin{cases}
    K(u,v) & \text{ if }\yb_v \in Y_{k,j}\\
    0 & \text{ otherwise}
    \end{cases}$$
    where $k \in \{1,2,\hdots,\kappa \}$ is the level and $j$ is indicative of the $j^{th}$ subdivision at level $k$.
    Then approximation of the matrix $K$ is given by
    \begin{equation}
        \Tilde{K} = \displaystyle \sum_{k=1}^{\kappa} \sum_{j=1}^{\beta \bkt{{2^{d'}}}^k} \Tilde{K}_{k,j} + K_{\kappa+1}
    \end{equation}
    where $\beta = \bkt{2^{d-d'}-1}$.
    We choose a $\bkt{p_{k,j}+1}^{d-d'}$ tensor product grid on Chebyshev nodes in the hyper-cube $Y_{k,j}$ to obtain the polynomial interpolation of $F(\xb,\yb)$ along $\yb$ by interpolating the function only along the direction of \emph{non-smoothness}. Let $M^{(i)}_{k,j} = \displaystyle \sup_{\yb \in \mathcal{B} \bkt{{Y_{k,j},\rho}}} \abs{F_a\bkt{\xb_i,\yb}}$,  $\rho \in (1,\alpha)^d$, for some $\alpha > 1$. Then by~\cref{eq1} the absolute error along the $i^{th}$ row of the matrix $K_{k,j}$ is given by (MATLAB notation)
        \begin{equation}\label{hyp_d}
             \Bigl\lvert \bkt{K_{k,j} \bkt{i,:} -\Tilde{K}_{k,j} \bkt{i,:} } \Bigl\lvert \leq 4 M_{k,j}^{(i)} V_d \frac{\rho^{-p_{k,j}}}{\rho -1} , , \qquad 1 \leq i \leq N, \quad 1 \leq k \leq \kappa, \quad 1 \leq j \leq \beta \bkt{{2^{d'}}}^k
        \end{equation}
Let $M = \max \{ M^{(i)}_{k,j} \} $. Therefore,
        \begin{equation}
            \magn{\bkt{K_{k,j} -\Tilde{K}_{k,j}}}_{max} \leq 4 M V_d \frac{\rho^{-p_{k,j}}}{\rho -1} \implies \dfrac{\magn{\bkt{K_{k,j} -\Tilde{K}_{k,j}}}_{max}}{\magn{K}_{max}} \leq \dfrac{4 M V_d}{\magn{K}_{max}} \frac{\rho^{-p_{k,j}}}{\rho -1} = \frac{c\rho^{-p_{k,j}}}{\rho -1}
        \end{equation}
         where $c = \dfrac{4 M V_d}{\magn{K}_{max}}$. Now choosing $p_{k,j}$ such that the above error is less than $\delta_1$ (for some $\delta_1>0$), we obtain \\
           $ p_{k,j}  = \bkt{1 + \ceil{\frac{ \log \bkt{\frac{c}{\delta_1 (\rho-1)}}}{\log(\rho)}}} \implies \dfrac{\magn{\bkt{K_{k,j} -\Tilde{K}_{k,j}}}_{max}}{\magn{K}_{max}} < \delta_1$ with rank of $\Tilde{K}_{k,j}$ is bounded above by $\bkt{1 + p_{k,j}}^{d-d'}$.
           Let $p_{l,m} = \max \Bigl \{p_{k,j} : k=1,2,\hdots, \kappa \text{ and } j = 1,2,\hdots, \beta \bkt{2^{d'}}^k \Bigl \}$, which corresponds to $\Tilde{K}_{l,m}$.
        Note that the rank of $K_{\kappa +1}$ is $N^{d'/d}$. So, the rank of $\Tilde{K}$ is bounded above by
        \begin{align} \label{eq5}
        \begin{split}
            N^{d'/d} + \mathlarger{\sum_{k=1}^{\kappa} \beta \bkt{{2^{d'}}}^k \bkt{1 + \ceil{\frac{ \log \bkt{\frac{c}{\delta_1 (\rho-1)}}}{\log(\rho)}}}^{d-d'} }
            \\ = N^{d'/d} + \beta \bkt{\frac{2^{d'} \bkt{\bkt{2^{d'}}^\kappa -1}}{\bkt{2^{d'}-1}}}  \bkt{1 + \ceil{\frac{ \log \bkt{\frac{c}{\delta_1 (\rho-1)}}}{\log(\rho)}}}^{d-d'} 
            \\
            \leq N^{d'/d} + \beta \bkt{\frac{2^{d'} \bkt{2^{d'}}^\kappa }{\bkt{2^{d'}-1}}} \bkt{1 + \ceil{\frac{ \log \bkt{\frac{c}{\delta_1 (\rho-1)}}}{\log(\rho)}}}^{d-d'}  \\
            = N^{d'/d} + \beta \dfrac{2^{d'}N^{d'/d}}{\bkt{2^{d'}-1}} \bkt{1 + \ceil{\frac{ \log \bkt{\frac{c}{\delta_1 (\rho-1)}}}{\log(\rho)}}}^{d-d'}  
        \end{split}
        \end{align}
    Therefore, the rank of $\Tilde{K} \in$
    $\mathcal{O}\bkt{N^{d'/d} \bkt{\log\bkt{\frac{1}{\delta_1}}}^{d-d'}}$
    and the absolute error is given by
        \begin{equation}
             \magn{\bkt{K - \Tilde{K}}}_{max} <  \sum_{k=1}^{\kappa} \beta \bkt{{2^{d'}}}^k \delta_1 \magn{K}_{max} = \beta \bkt{\frac{2^{d'} \bkt{\bkt{2^{d'}}^\kappa -1}}{\bkt{2^{d'}-1}}}  \delta_1 \magn{K}_{max}  < \frac{\beta 2^{d'}  N^{d'/d}}{\bkt{2^{d'}-1}} \delta_1 \magn{K}_{max} 
        \end{equation}
        Pick $\delta_1 = \frac{\delta \bkt{2^{d'}-1}}{\beta 2^{d'}  N^{d'/d}}$, then
           $  \dfrac{\magn{K -\Tilde{K}}_{max}}{\magn{K}_{max}} < \delta $ and the rank of $\Tilde{K} \in$ $\mathcal{O}\bkt{N^{d'/d} \bkt{\log^{d-d'} \bkt{\frac{\beta 2^{d'} N^{d'/d}}{\bkt{2^{d'}-1} \delta}}}}$
           
        Hence, the rank of $\Tilde{K}$ scales $\mathcal{O}\bkt{N^{d'/d} \bkt{\log^{d-d'}\bkt{\frac{N}{\delta}}}}$ with $\dfrac{\magn{K -\Tilde{K}}_{max}}{\magn{K}_{max}} < \delta $. 
    \end{proof}
    The numerical rank plots of the different $4$D, $3$D, $2$D hyper-surface-sharing interaction of the eight functions as described in~\Cref{Preliminaries} are shown in~\cref{fig:hypersurface1_rank,fig:hypersurface2_rank,fig:hypersurface3_rank} and tabulated in~\cref{tab:res_4d_h1_log,tab:res_3d_edge_log,tab:res_2d_edge_log,tab:res_4d_h2_log,tab:res_3d_face_log,tab:res_4d_h3_log}.

\begin{figure}[H]
    \centering
    \subfloat[Numerical rank growth of $4$D 1-hyper-surface-sharing $(d=4, d'=1)$]{\includegraphics[height=3cm, width=5.2cm]{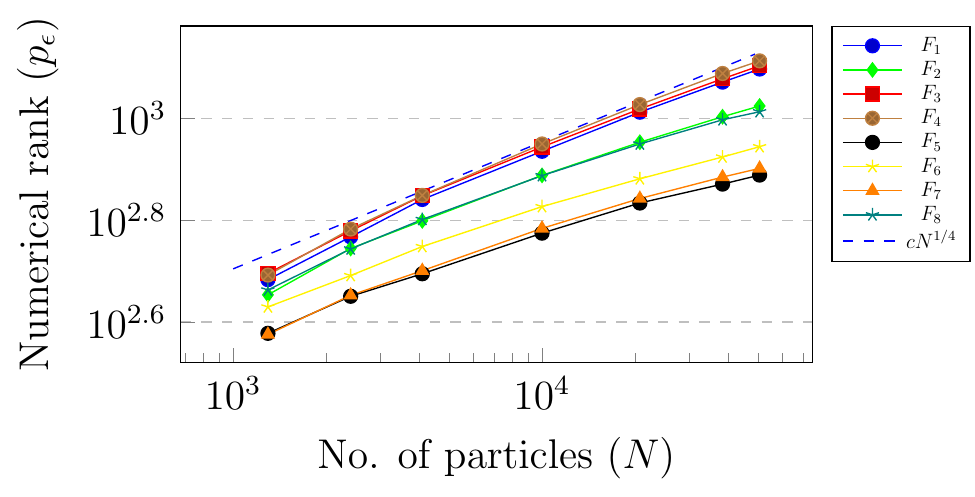}\label{fig:res_4d_h1_log}}\quad%
    \subfloat[Numerical rank growth of $3$D edge-sharing $(d=3, d'=1)$]{\includegraphics[height=3cm, width=5.2cm]{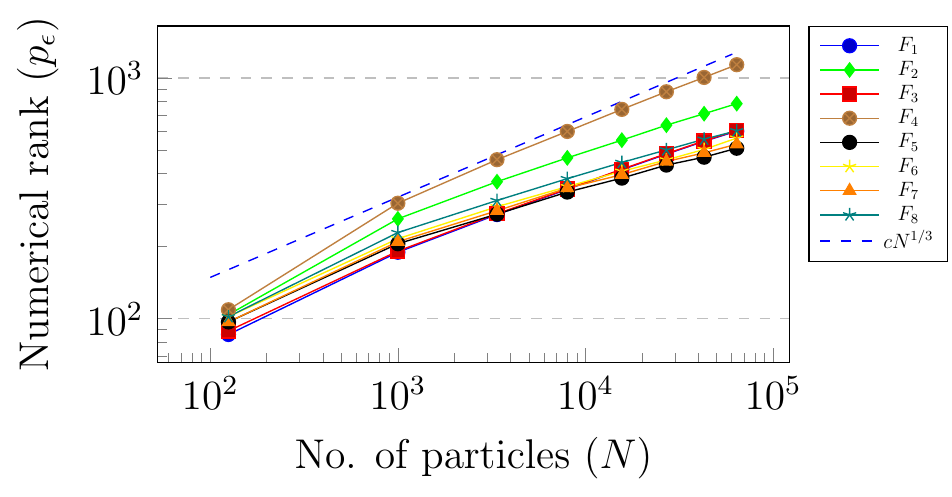}\label{fig:res_3d_edge_log}}\quad%
    \subfloat[Numerical rank growth of $2$D edge-sharing $(d=2, d'=1)$]{\includegraphics[height=3cm, width=5.2cm]{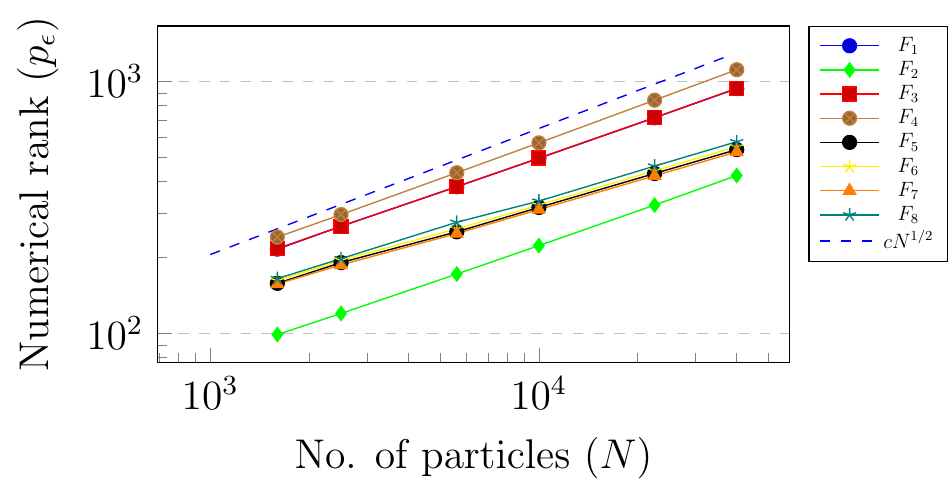}\label{fig:res_2d_edge_log}}\quad%
    \caption{Numerical rank growth with $N$ in different dimension for the  1-hyper-surface-sharing interaction}
    \label{fig:hypersurface1_rank}
\end{figure}

\begin{figure}[H]
    \centering
    \subfloat[Numerical rank growth of $4$D 2-hyper-sharing $(d=4, d'=2)$]{\includegraphics[height=3cm, width=6cm]{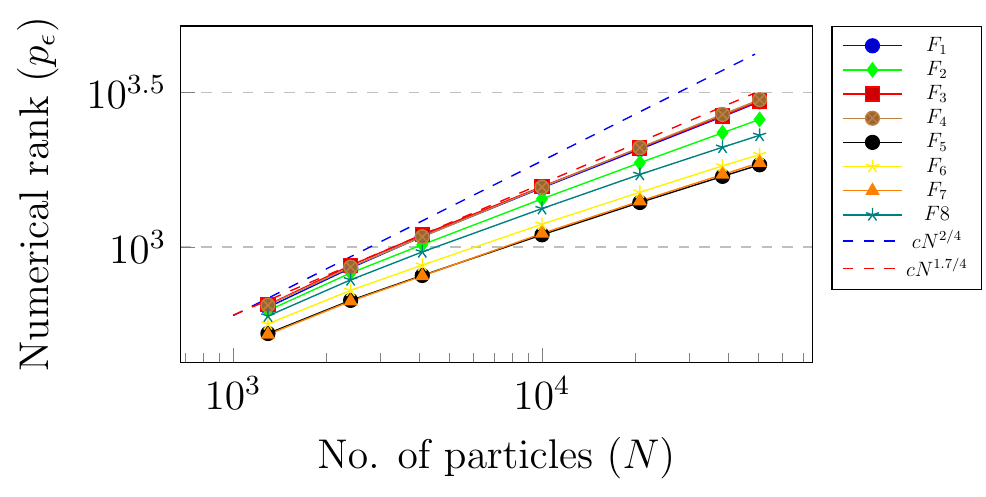}\label{fig:res_4d_h2_log}} \qquad%
    \subfloat[Numerical rank growth of $3$D face-sharing $(d=3, d'=2)$]{\includegraphics[height=3cm, width=6cm]{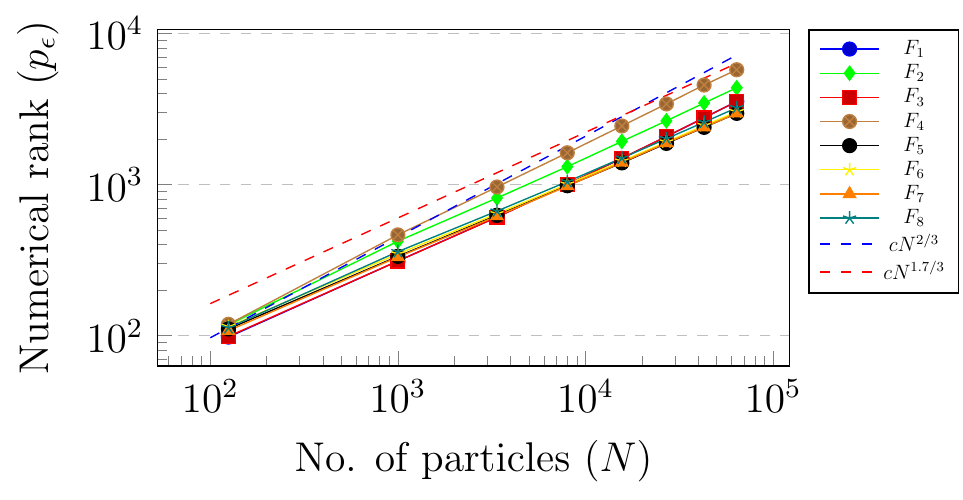}\label{fig:res_3d_face_log}} \qquad%
    \caption{Numerical rank growth with $N$ in different dimension for the 2-hyper-surface-sharing interaction}
    \label{fig:hypersurface2_rank}
\end{figure}

\begin{figure}[H]
    \centering
    \subfloat[Numerical rank growth of $4$D 3-hyper-surface-sharing $(d=4, d'=3)$]{\includegraphics[height=3cm, width=6cm]{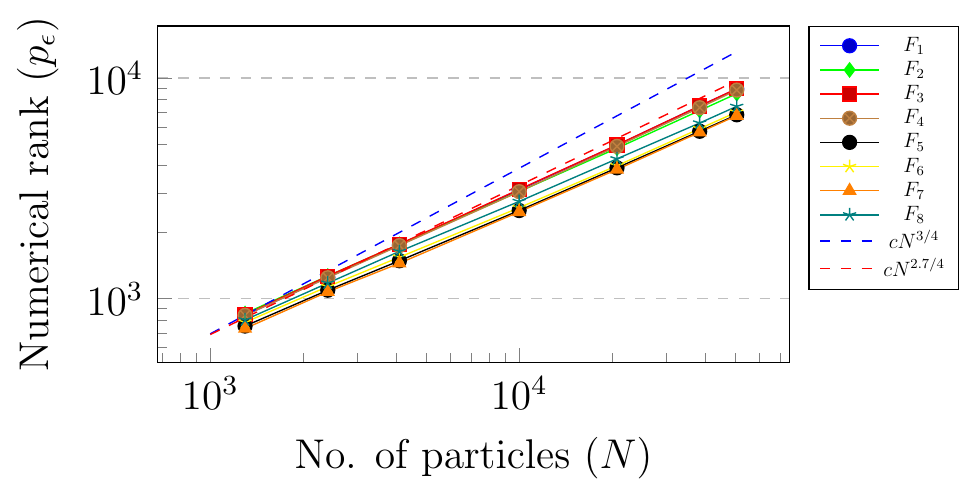}\label{fig:res_4d_h3_log}}\quad%
    \caption{Numerical rank growth with $N$ in different dimension for the  3-hyper-surface-sharing interaction}
    \label{fig:hypersurface3_rank}
\end{figure}

\begin{remark}
    In the proofs, we need a hierarchy of hyper-cubes so that at each level, the interaction is at least one hyper-cube away. This ensures that $\rho$ is constant at all levels.
\end{remark}
\begin{remark}
    One may find a tighter upper bound for the error of specific kernels (therefore, tighter upper bound for the rank) if we have a multipole/eigen function expansion as in~\cite{shortFMM,kandappan2022hodlr2d}. Our theorems are general and applicable to a wide range of kernel functions under reasonably mild assumptions.
\end{remark}
\begin{remark}
    We also perform an optimal curve fitting in \Cref{fig:hypersurface2_rank,fig:hypersurface3_rank}, where the power of $N$ is close to the proposed bound (but for these plots $N \leq 64000$). We also plot the maximum rank obtained from the ACA technique for a larger value of $N$ in \cref{s_rank} (for the $\log(\pmb{r})$ kernel), where the plot follows our theoretical bounds.
\end{remark}
\begin{remark}
    \textbf{The main point we want to highlight is that if we compress sub-matrices corresponding to $d'$-hyper-surface (where $d' \geq 1$) sharing clusters, the rank grows as a positive power of $N$. Hence, in the context of hierarchical matrices with weak admissibility criteria, it is preferable to include only \textbf{\emph{vertex sharing clusters}} and far-field clusters as admissible clusters}.
\end{remark}


 \subsection{Specific results for $1$D, $2$D and $3$D}
 One can easily deduce the following results as shown in~\cref{tab:allD} from our theorems. 
 \begin{table}[H]
     \centering
     \resizebox{\textwidth}{!}{\begin{tabular}{|c|c|c|c|}
     \hline
          \multirow{2}{*}{Interaction type} & \multicolumn{3}{|c|}{Rank growth} \\
     \cline{2-4}
     & $1$D $(d=1)$ & $2$D $(d=2)$ & $3$D $(d=3)$\\
     \hline
          Far-field & $\mathcal{O}\bkt{\log\bkt{1/\delta}}$ & $\mathcal{O}\bkt{\log^2\bkt{1/\delta}}$ & $\mathcal{O}\bkt{\log^3\bkt{1/\delta}}$\\
     \hline
          Vertex-sharing $(d'=0)$ & $\mathcal{O}\bkt{\log_2\bkt{N}\log\bkt{\log_2\bkt{N/\delta}}}$ & $\mathcal{O}\bkt{\log_4\bkt{N}\log\bkt{\log_4\bkt{N/\delta}}}$ & $\mathcal{O}\bkt{\log_8\bkt{N}\log^3\bkt{\log_8\bkt{N/\delta}}}$\\
     \hline
        Edge-sharing $(d'=1)$ &
        NA &
        $\mathcal{O}\bkt{N^{1/2}\log\bkt{N/\delta}}$ & $\mathcal{O}\bkt{N^{1/3}\log^2\bkt{N/\delta}}$\\
    \hline
        Face-sharing $(d'=2)$ &
        NA & 
        NA & 
        $\mathcal{O}\bkt{N^{2/3}\log\bkt{N/\delta}}$\\
    \hline
     \end{tabular}}
     \caption{Rank growth of different interactions in $1$D, $2$D and $3$D}
     \label{tab:allD}
 \end{table}
\section{HODLR$d$D: Weak admissibility condition in higher dimensions}\label{weak_admis} In this section, we discuss a possible way to extend the notion of weak admissibility to higher dimensions and also describe our HODLR$d$D algorithm for particle simulations in $d$ dimensions.
\subsection{Weak admissibility condition in higher dimensions} \label{dweak} Hackbusch et al. \cite{hackbusch2004hierarchical} first introduced the notion of weak admissibility in one dimension, where they compress the sub-matrices corresponding to the adjacent intervals. Equivalently in terms of matrix setting, this represents all the off-diagonal sub-matrices as low-rank. However, if we extend this notion of weak admissibility (i.e., compressing sub-matrices corresponding to adjacent clusters) to higher dimensions $(d>1)$ and compress \emph{\textbf{all}} off-diagonal sub-matrices, then it will not result in an almost linear complexity matrix-vector product algorithm. This is because, except for the vertex-sharing interaction the rank of all other adjacent interactions grows with some positive power of $N$. In \Cref{num_results}, we develop the HODLR algorithm in $4$D based on their weak admissibility condition (i.e., except  for the self-interaction, we compressed all other interactions) and its complexity is roughly $\mathcal{O} \bkt{N^{7/4} \log{N}}$. Hence, it is important to note that the notion of weak admissibility introduced by Hackbusch et al. \cite{hackbusch2004hierarchical} is fine for one dimensional problems but needs to be revised for higher-dimensional problems. As part of this article, in this section, we propose a possible way to extend the notion of weak admissibility condition to higher dimensions. Our theorems show that the rank of sub-matrices corresponding to not just the far-field interactions but also vertex-sharing interactions do not scale with any power of $N$. Therefore, treating the clusters corresponding to far-field and \textbf{vertex-sharing interactions} as admissible clusters could be considered the right way to extend the notion of \emph{weak admissibility} for higher dimensions.
\subsection{Construction of the tree to build $d$ dimensional hierarchical matrix} We take a $d$ dimensional hyper-cube $C \subset \Rb^d$ as the computational domain, that contains the particles. We divide the hyper-cube $C$ using the $2^d$ tree. At level 0 (root level) of the tree is the hyper-cube $C$ itself. A hyper-cube at level $k$ is subdivided into $2^d$ finer hyper-cubes that belong to the level $k+1$ of the tree. The former is called the parent of the latter, and the latter are the children of the former. We stop the sub-division at a level $\kappa$ of the tree, when each finest hyper-cube contains at most $N_{max}$ particles. $N_{max}$ is a user-specified threshold that defines the maximum number of particles at a finest hyper-cube.  The $2^{d \kappa}$ finest hyper-cubes at level $\kappa$ are called the leaves. In this article, we consider the balanced tree for simplicity. An adaptive $2^d$ tree or K-d tree can be considered too. Let $N$ be the total number of particles inside the hyper-cube $C$ then $N \leq N_{max} 2^{d \kappa} \implies \kappa =  \ceil{\log_{2^d} \bkt{N/N_{max}}}$

\subsection{HODLR$d$D, a black-box fast algorithm for matrix-vector product} \label{HODLR$d$D_algo}
 We develop our \\ HODLR$d$D algorithm based on the weak admissibility condition (admissibility of far-field and vertex-sharing interactions) as mentioned in \Cref{dweak}. The complexity of our HODLR$d$D algorithm is $\mathcal{O}\bkt{pN \log \bkt{N}}$ and our theorems guarantee that $p \in \mathcal{O} \bkt{\log \bkt{N} \log^d \bkt{\log \bkt{N}}}$, i.e., it is an almost linear complexity algorithm for matrix-vector product in any dimension $d$. The article~\cite{kandappan2022hodlr2d} by our group discusses the rank bounds of $2$D Green's function and the HODLR2D algorithm. But in this article, rank bounds obtained from our theorems are applicable for a wide range of kernel functions in any dimension $d$. In contrast, we have generalized the algorithm in any dimension $d$ based on our theorems of rank bounds.
This section introduces a few notations (\cref{tab:HOD_notation}) to explain our HODLR$d$D algorithm.
 \begin{table}[H]
     \centering
     \resizebox{\textwidth}{!}{\begin{tabular}{|c|c|}
     \hline
        $\mathcal{C}_i^{\bkt{l}}$ &  Cluster of particles inside the hyper-cube $i$ at level $l$ of the $2^d$ tree (we subdivide the hyper-cube $C$ using $2^d$ tree) \\
    \hline
        $\mathcal{F}_{\mathcal{C}_i^{\bkt{l}}}$  &  Set of hyper-cubes that are in far-field of $\mathcal{C}_i^{\bkt{l}}$ i.e., $\{\mathcal{C}_{j}^{(l)}: \text{hyper-cubes $i$ and $j$ are at least one hyper-cube away}\}$  \\
    \hline
        $\mathcal{HS}_{\mathcal{C}_i^{\bkt{l}},d'}$  &  \makecell{Set of hyper-cubes that share $d'$ hyper-surface with $\mathcal{C}_i^{\bkt{l}}$ i.e., $\{\mathcal{C}_{j}^{(l)}: \text{hyper-cubes $i$ and $j$ share $d'$ hyper-surface}\}$ \\ e.g., $\mathcal{HS}_{\mathcal{C}_i^{\bkt{l}},0}$ is the set of hyper-cubes that share a vertex with $\mathcal{C}_i^{\bkt{l}}$}  \\
    \hline
        $\text{child} \bkt{\mathcal{C}_i^{\bkt{l}}}$  &  $\{\mathcal{C}_{j}^{(l+1)}: \text{hyper-cubes $j$ is a child of coarser hyper-cube $i$} \}$  \\
    \hline
    $\text{parent} \bkt{\mathcal{C}_i^{\bkt{l}}}$  &  $\{\mathcal{C}_{j}^{(l-1)}, \text{hyper-cubes $i$ is a child of coarser hyper-cube $j$} \} $  \\
    \hline
    $\text{siblings} \bkt{\mathcal{C}_i^{\bkt{l}}}$  &  $\text{child}\bkt{\text{parent} \bkt{\mathcal{C}_i^{\bkt{l}}}} \setminus \mathcal{C}_i^{\bkt{l}} $  \\
    \hline
    $\text{clan} \bkt{\mathcal{C}_i^{\bkt{l}}}$  &  $ \{ \text{siblings} \bkt{\mathcal{C}_i^{\bkt{l}}} \} \bigcup \{ \text{child} \bkt{P} : P \in \bigcup\limits_{d'=0}^{d-1} \mathcal{HS}_{\text{parent} \bkt{\mathcal{C}_i^{\bkt{l}}},d'} \}$ \\
    \hline
    $\mathcal{IL} \bkt{\mathcal{C}_i^{\bkt{l}}}$  & The interaction list of the cluster $\mathcal{C}_i^{\bkt{l}}$ is defined as  $\text{clan} \bkt{\mathcal{C}_i^{\bkt{l}}} \bigcap \bkt{\mathcal{HS}_{\mathcal{C}_i^{\bkt{l}},0} \bigcup \mathcal{F}_{\mathcal{C}_i^{\bkt{l}}}}$ \\
    \hline
     \end{tabular}}
     \caption{Notations for this section (\Cref{HODLR$d$D_algo})}
     \label{tab:HOD_notation}
 \end{table}

\vspace{1cm}

\begin{definition} 
    \textbf{HODLR$d$D admissibility condition}: Two different clusters $\mathcal{C}_i^{\bkt{l}}$ and $\mathcal{C}_j^{\bkt{l}}$ at level $l$ of the $2^d$ tree are admissible iff either they do not share $d'$ hypersurface ($d'>0$) or they share at the most a vertex.
\end{definition}
We illustrate the image of HODLR$d$D matrix in $4$D (i.e. HODLR4D) at level $1$ and level $2$ in \cref{fig:HODLR4D_matrix}.
\begin{figure}[H]
    \centering
    \subfloat[HODLR4D at level 1]{
        \includegraphics[scale=0.5]{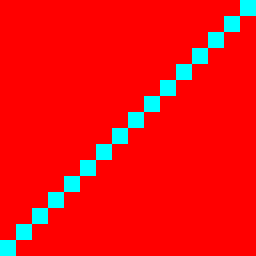}
        }\quad \quad
    \subfloat[HODLR4D at level 2]{
        \includegraphics[scale=0.5]{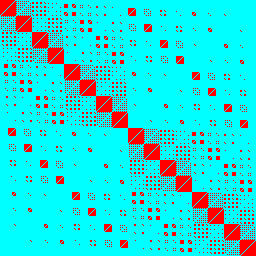}
        }\quad
    \subfloat{
        \begin{tikzpicture}
            [
            box/.style={rectangle,draw=black, minimum size=0.25cm},scale=0.2
            ]
            \node[box,fill=red,,font=\tiny,label=right:Full-rank block,  anchor=west] at (0,6){};
            \node[box,fill=cyan,,font=\tiny,label=right:Low-rank block,  anchor=west] at (0,2){};
        \end{tikzpicture}
    }
    \caption{HODLR4D matrix at level 1 and 2}
    \label{fig:HODLR4D_matrix}
\end{figure}
\begin{figure}[H]
    \centering
    \subfloat[$\mathcal{H}$ matrix at level 1]{
        \includegraphics[scale=0.5]{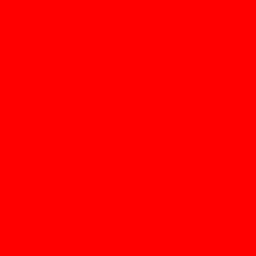}
        }\quad \quad
    \subfloat[$\mathcal{H}$ matrix at level 2]{
        \includegraphics[scale=0.5]{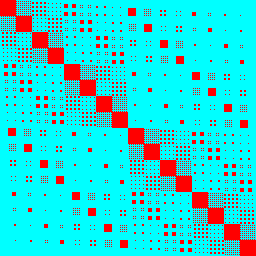}
        }\quad
    \subfloat{
        \begin{tikzpicture}
            [
            box/.style={rectangle,draw=black, minimum size=0.25cm},scale=0.2
            ]
            \node[box,fill=red,,font=\tiny,label=right:Full-rank block,  anchor=west] at (0,6){};
            \node[box,fill=cyan,,font=\tiny,label=right:Low-rank block,  anchor=west] at (0,2){};
        \end{tikzpicture}
    }
    \caption{$\mathcal{H}$ matrix with standard/strong admissibility at level 1 and 2 in $4$D}
    \label{fig:H_matrix}
\end{figure}
\begin{figure}[H]
    \centering
    \subfloat[HODLR in $4$D at level 1]{
        \includegraphics[scale=0.5]{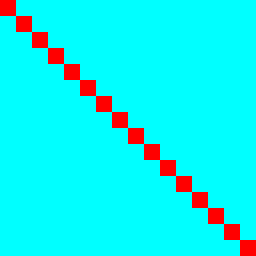}
        }\quad \quad
    \subfloat[HODLR in $4$D at level 2]{
        \includegraphics[scale=0.5]{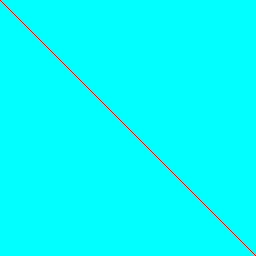}
        }\quad
    \subfloat{
        \begin{tikzpicture}
            [
            box/.style={rectangle,draw=black, minimum size=0.25cm},scale=0.2
            ]
            \node[box,fill=red,,font=\tiny,label=right:Full-rank block,  anchor=west] at (0,6){};
            \node[box,fill=cyan,,font=\tiny,label=right:Low-rank block,  anchor=west] at (0,2){};
        \end{tikzpicture}
    }
    \caption{HODLR in $4$D at level 1 and 2}
    \label{fig:HODLR_matrix}
\end{figure}
We will discuss our HODLR$d$D algorithm in two phases. The first phase is the initialization of HODLR$d$D hierarchical structure, as described in~\cref{alg:initialization}. One of the issues with the higher dimensional problems is the scaling of the number of particles (curse of dimensionality). For example, if we consider $100$ particles along each dimension, then for a $d$ dimensional problem, the number of particles, $N=100^d$. Hence, storing all the low-rank factors of the interaction matrices might exhaust the RAM for large $N$. To address this issue, we store the factors in a memory-efficient way. We rely on algebraic technique for the low-rank compression. For each pair of clusters $\mathcal{C}_i^{\bkt{l}}$ and $\mathcal{C}_j^{\bkt{l}} \in \mathcal{IL} \bkt{\mathcal{C}_i^{\bkt{l}}}$ at all levels of the $2^d$ tree, with index sets $I$ and $J$ respectively, the interaction matrix $K\bkt{I,J}$ is compressed using ACA \cite{bebendorf2003adaptive,tyrtyshnikov2000incomplete} with a tolerance $\epsilon$ as follows.
\begin{align} \label{keq1}
    K\bkt{I,J} \approx K \bkt{I,\tau_{IJ}} K\bkt{\sigma_{IJ},\tau_{IJ}}^{-1}  K\bkt{\sigma_{IJ},J}
\end{align}
where $\sigma_{IJ}$ and $\tau_{IJ}$ are the pivots of ACA. We compute the LU factors of $K\bkt{\sigma_{IJ},\tau_{IJ}}$, $L_{IJ}$ and $U_{IJ}$ such that $K\bkt{\sigma_{IJ},\tau_{IJ}} = L_{IJ} U_{IJ}$. Hence, \cref{keq1} can be re-written as
\begin{align} \label{keq2}
    K\bkt{I,J} \approx K \bkt{I,\tau_{IJ}} U_{IJ}^{-1} L_{IJ}^{-1} K\bkt{\sigma_{IJ},J}
\end{align}
The LU factors of $K\bkt{\sigma_{IJ},\tau_{IJ}}$ can be obtained as a by-product of the ACA routine~\cite{bebendorf2009recompression}. Also, in the ACA routine, we only store the pivots $\sigma_{IJ}, \tau_{IJ}$ and the matrices $L_{IJ}$ and $U_{IJ}$. We do not store the matrices $ K \bkt{I,\tau_{IJ}}$ and $ K\bkt{\sigma_{IJ},J} $.
The second phase is to perform the HODLR$d$D matrix-vector product as described in~\cref{alg:matvec}.
\begin{remark}
    When we apply the inverse of upper triangular (lower triangular) matrices, we perform backward (forward) substitution.
\end{remark}
 \begin{algorithm}[H]
	\caption{HODLR$d$D Initialization}\label{alg:initialization}
	\begin{algorithmic}[1]
		\Procedure{Initialize}{$N_{\max}$,$\epsilon$}
		\State{} \Comment{$N_{\max}$ is the maximum number of particles at leaf level}
		\State{} \Comment{$\epsilon$ is the tolerance for the ACA}
		\State {Form the $2^d$ tree with number of levels $\kappa$, \\ where $\kappa= \min \left\{l : \abs{\mathcal{C}^{(l)}_{i}} < N_{\max};\forall i \in \{0,1,2,\ldots,2^{dl}-1\}\right\}$}
\For{\texttt{$l=1:\kappa$}} 
        \Comment{Low-rank compression of interaction list with memory-efficient ACA}
				\For{\texttt{$i=0:2^{dl}-1$}}
					\State $I\gets \text{Index set of } \mathcal{C}_{i}^{(l)}$
					\For{\texttt{j in $\mathcal{IL}_{\mathcal{C}_{i}^{(l)}}$}}
						\State $J\gets \text{Index set of } \mathcal{C}_{j}^{(l)}$
						\State Perform memory-efficient ACA with tolerance $\epsilon$ on the sub-matrix $K(I,J)  \approx K \bkt{I,\tau_{IJ}} U_{IJ}^{-1} L_{IJ}^{-1} K\bkt{\sigma_{IJ},J}$ 
					\EndFor
				\EndFor
			\EndFor
		\EndProcedure
	\end{algorithmic}
\end{algorithm}
\begin{remark}
    \textbf{Complexity of the \cref{alg:initialization}:} We rely on ACA for the low-rank compression. Also, our theorems guarantee that the rank grows almost linearly for HODLR$d$D algorithm. So, the overall time and memory complexity of the initialization step is also almost linear.
\end{remark}
\begin{remark}
    \textbf{Complexity of the \cref{alg:matvec}:} The complexity of \mvp using the HODLR$d$D representation with $\log \bkt{N}$ levels is $\mathcal{O} \bkt{pN \log \bkt{N}}$. Our theorems guarantee that the rank $p \in \mathcal{O} \bkt{\log \bkt{N} \log^d \bkt{\log \bkt{N}}}$. Therefore, it is an almost linear complexity algorithm for \mvp in $d$ dimensions.
\end{remark}
 \begin{algorithm}[H]
	\caption{HODLR$d$D matrix-vector Product $K\pmb{q} = \pmb{b}$}\label{alg:matvec}
	\begin{algorithmic}[1]
		\Procedure{MatVec}{$\pmb{q}$}
			\For{\texttt{$l=1:\kappa$}} \Comment{Low-rank matrix-vector product}
				\For{\texttt{$i=0:2^{dl}-1$}}
					\State $I\gets \text{Index set of } \mathcal{C}_{i}^{(l)}$
					\For{\texttt{j in $\mathcal{IL}_{\mathcal{C}_{i}^{(l)}}$}}
						\State $J\gets \text{Index set of } \mathcal{C}_{j}^{(l)}$
						\State $\pmb{b}(I) = \pmb{b}(I) + K(I,\tau_{IJ}) U_{IJ}^{-1}L_{IJ}^{-1}  K(\sigma_{IJ},J) \pmb{q}(J)$
					\EndFor
				\EndFor
			\EndFor
   			\For{\texttt{i=0:$2^{d \kappa}-1$}} \Comment{Full-rank/Dense matrix-vector product at leaf level}
				\State $I \gets \text{Index set of } \mathcal{C}_{i}^{(\kappa)}$
				\State Form the dense matrix $K\bkt{I,I}$                   \Comment{Self interaction matrix}
				\State $\pmb{b}(I) = \pmb{b}(I) + K(I,I) \times \pmb{q}(I)$       
				\For{\texttt{j in $\bigcup\limits_{d'=1}^{d-1} \mathcal{HS}_{ \mathcal{C}_i^{\bkt{\kappa}},d'}$}}
					\State $J\gets \text{Index set of } \mathcal{C}_{j}^{(\kappa)}$
					\State Form the dense matrix $K(I,J)$                        \Comment{Other nearby interaction matrices}
				    \State $\pmb{b}(I) = \pmb{b}(I) + K(I,J) \times \pmb{q}(J)$    
				\EndFor
			\EndFor
			\State \textbf{return} $\pmb{b}$
		\EndProcedure
	\end{algorithmic}
\end{algorithm}

The $\texttt{C++}$ implementation with \texttt{OpenMP} parallelization of the HODLR$d$D algorithm is available at \texttt{\url{https://github.com/SAFRAN-LAB/HODLRdD}}. Also, the code works for any value of $d$, i.e., in any dimension $d$ and any user-specified matrix. 
\section{Numerical experiments} \label{num_results}
In this section, we demonstrate the performance of our HODLR$d$D algorithm by comparing it with $\mathcal{H}$ matrix with standard/strong admissibility condition (blocks corresponding to the far-field interactions are compressed) \Cref{fig:H_matrix} and HODLR (where all off-diagonal blocks corresponding to non-self interactions are compressed) \Cref{fig:HODLR_matrix}. In the first experiment, we perform a HODLR$d$D matrix-vector product. In the second experiment, we demonstrate HODLR$d$D accelerated iterative solver (GMRES \cite{saad1986gmres}) to solve a $d$ dimensional integral equation. In the last experiment, we compare the naive kernel SVM with the HODLR$d$D accelerated SVM. We show that one can significantly improve the training period of the SVM algorithm using our HODLR$d$D algorithm. We define some notations in \Cref{tab:app_notations} for \Cref{smvp,ieqs}.
 \begin{table}[H]
     \centering
     \resizebox{\textwidth}{!}{\begin{tabular}{|c|c|}
     \hline
        $N$ &  Total number of particles in the domain or size of the kernel matrix \\
    \hline
        Initialization Time &  \makecell{The time taken (in seconds) by the Initialization routine (\Cref{alg:initialization}) and \\ the time taken to form the dense matrices at leaf level.}  \\
    \hline
        M-V product time  & \makecell{The time taken (in seconds) by the Mat-Vec routine (\Cref{alg:matvec}) \\ excluding the time taken to form the dense matrices at leaf level.}   \\
    \hline
        Total time  & Initialization Time + M-V product time    \\
    \hline
    Maximum rank  &  \makecell{The maximum rank across all the interaction lists, which were compressed using the ACA algorithm \\ while building the hierarchical representation, ACA tolerance $= 10^{-6}$.}  \\
    \hline
    $N_{max}$ & Maximum number of particles at leaf level. We set $N_{max} = 1000$. \\
    \hline
    Memory & The total memory (in GB) needed to store the hierarchical representation. \\
    \hline
    Solution time & \makecell{Total time (in seconds) to solve the system $A \pmb{\sigma} = \pmb{f}$ using GMRES  \\ (\textbf{Stopping criteria:} If the residual is less than $10^{-6}$ GMRES routine will stop) \\ In the GMRES routine, the matrix-vector product is accelerated using different hierarchical matrices}\\
    \hline
     \end{tabular}}
     \caption{Notations for \Cref{smvp} and \Cref{ieqs}}
     \label{tab:app_notations}
 \end{table}

\subsection{HODLR$d$D fast matrix-vector product} \label{smvp} In the first experiment, we demonstrate the scalability of the matrix-vector product through HODLR$d$D $(d=4)$, $\mathcal{H}$ matrix with standard/strong admissibility condition in $4$D and HODLR in $4$D. We consider a uniform distribution of particles inside the domain $C = [-1,1]^4 \subset \Rb^4$. The $(i,j)^{th}$ entry of the kernel matrix is given by              $$K(i,j) = \begin{cases}
            \log \bkt{\magn{\pmb{x}_i - \pmb{x}_j}_2} & \text{ if } i \neq j , \qquad \pmb{x}_i, \pmb{x}_j \in C\\
           \qquad 0 & \text{ otherwise}
            \end{cases}$$
We apply the above matrix to a random column vector and report the initialization time (\Cref{s_init}), matrix-vector product time (\Cref{s_mvp}), total time (\Cref{s_total}), maximum rank (\Cref{s_rank}), relative error (\Cref{s_error}) with $N$ in \Cref{fig:smat_vec}. This experiment was performed on an Intel Xeon Gold 2.5GHz processor without \texttt{OpenMP}.
\begin{figure}[H]
    \centering
    \subfloat[]{\includegraphics[height=3.2cm, width=5cm]{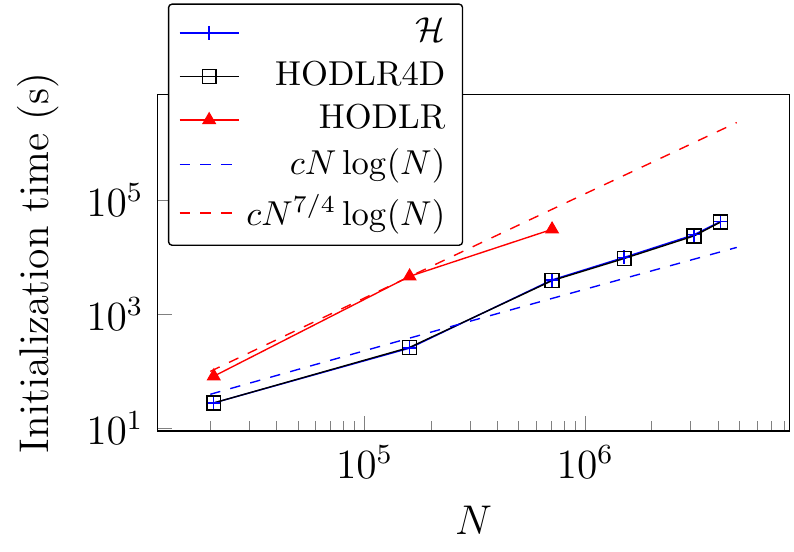}\label{s_init}}\quad%
    \subfloat[]{\includegraphics[height=3.2cm, width=5cm]{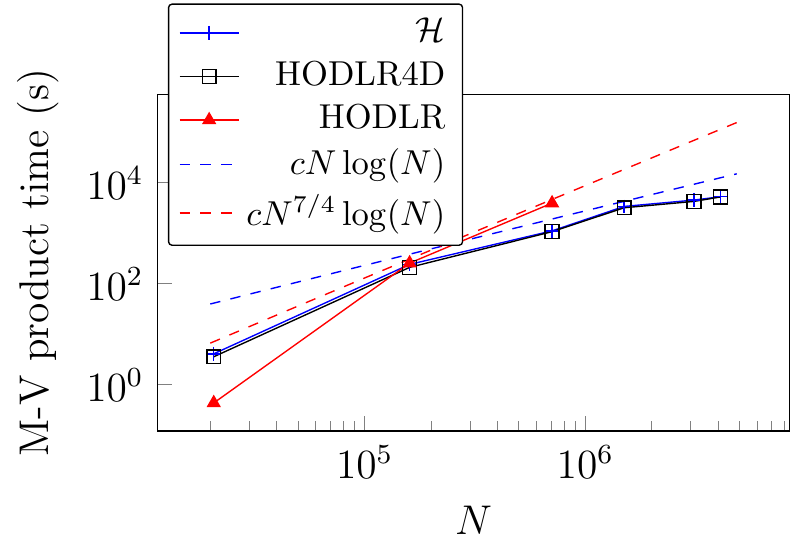}\label{s_mvp}}\quad%
    \subfloat[]{\includegraphics[height=3.2cm, width=5cm]{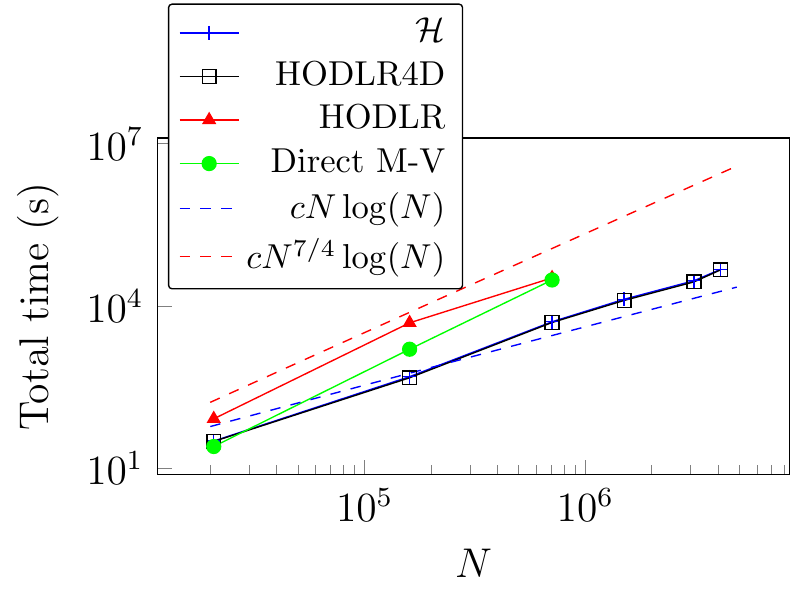}\label{s_total}}\quad%
    \subfloat[]{\includegraphics[height=3.2cm, width=7cm]{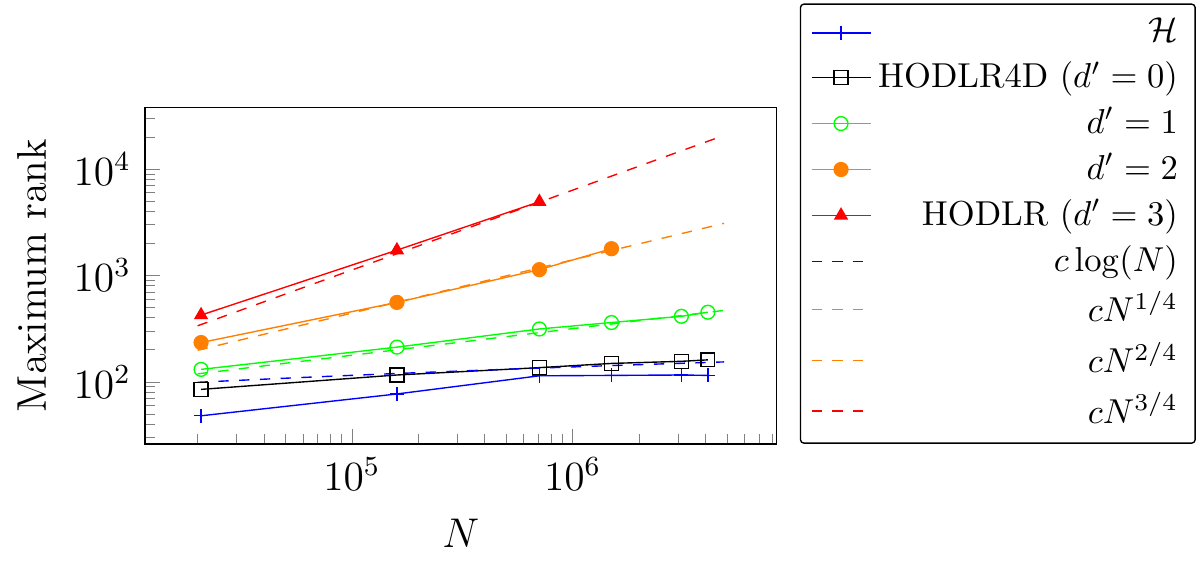}\label{s_rank}}\qquad \qquad%
    \subfloat[]{\includegraphics[height=3.2cm, width=5cm]{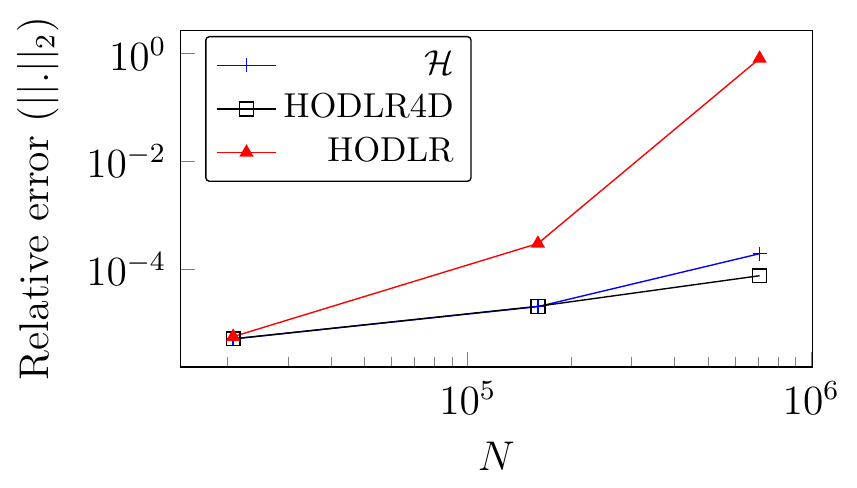}\label{s_error}}\quad%
    \caption{Various benchmarks of the HODLR$d$D $(d=4)$ matrix-vector product in comparison with those of $\mathcal{H}$, HODLR for the kernel $\log \bkt{||\pmb{x}-\pmb{y}||_2}$. In \Cref{s_rank} we compare the \textbf{maximum rank} of all possible hierarchical matrices based on different \textit{weak admissibility} condition, i.e., $d'=0,1,2,3$ and also the $\mathcal{H}$ matrix with \textit{strong admissibility} condition. We set the ACA tolerance $= 10^{-6}$}
    \label{fig:smat_vec}
\end{figure}
\begin{remark}
    For $N>10^6$, the HODLR code fails due to memory issues and also its relative error is very high, whereas HODLR$d$D and $\mathcal{H}$ matrix with \textit{strong admissibility} condition perform better.
\end{remark}
\subsection{HODLR$d$D accelerated iterative solver for integral equations in $d$ dimensions} \label{ieqs} In the second experiment, we consider the Fredholm integral equation of the second kind in $C = [-1,1]^4 \subset \Rb^4$,
\begin{equation} \label{ieq_1}
    \sigma(\pmb{x})+\int_{C} F(\pmb{x},\pmb{y})\sigma(\pmb{y})d\pmb{y} = f(\pmb{x}) \qquad \pmb{x}, \pmb{y} \in C
\end{equation}
with $F(\pmb{x},\pmb{y}) = -\dfrac{1}{4\pi^2 \magn{\pmb{x}-\pmb{y}}_2^{2}}$ ($4$D Green's function for Laplace equation). We follow a piecewise constant collocation method with collocation points on a uniform grid in $4$D to discretize \cref{ieq_1} and obtain a linear system 
\begin{equation} \label{ieq_2}
    K \pmb{\sigma} = \pmb{f}
\end{equation}
We consider a random vector $\pmb{\sigma}$ and calculate the vector $\pmb{f} = K \pmb{\sigma}$ (exact up to roundoff). We take $\pmb{f}$ as the RHS of \cref{ieq_2} and solve the system using an iterative solver, GMRES~\cite{saad1986gmres}, where we accelerate the matrix-vector product in each iteration using the HODLR$d$D, HODLR and $\mathcal{H}$ matrix algorithm. We compute the relative error as $\dfrac{\magn{\pmb{\sigma}_c - \pmb{\sigma}}_2}{\magn{\pmb{\sigma}}_2}$, where $\pmb{\sigma}_c$ be the computed $\pmb{\sigma}$ (we obtain $\pmb{\sigma}_c$ from GMRES solver). We illustrate the comparison between these three algorithms in \Cref{fig:GMRES}. This experiment was performed on an Intel Xeon Gold 2.5GHz processor with 8 \texttt{OpenMP} threads.
\begin{figure}[H]
    \centering
    \subfloat[]{\includegraphics[height=3cm, width=5cm]{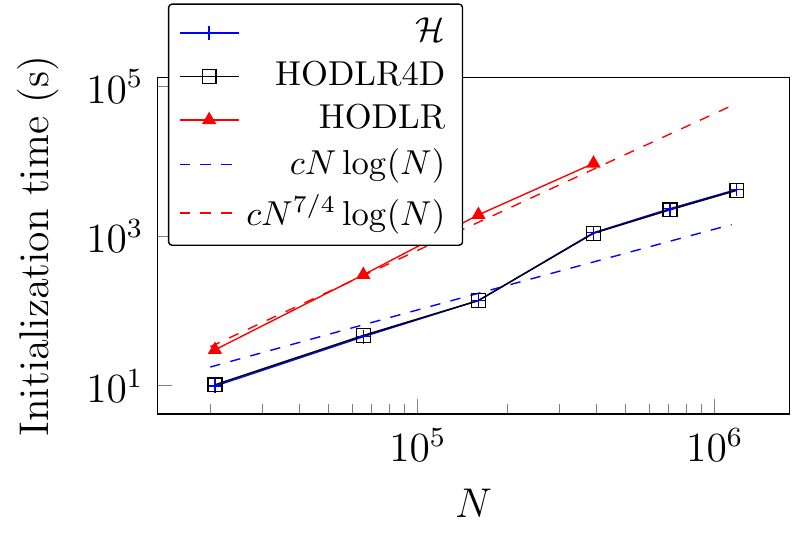}\label{ieq_init}}\quad%
    \subfloat[]{\includegraphics[height=3cm, width=5cm]{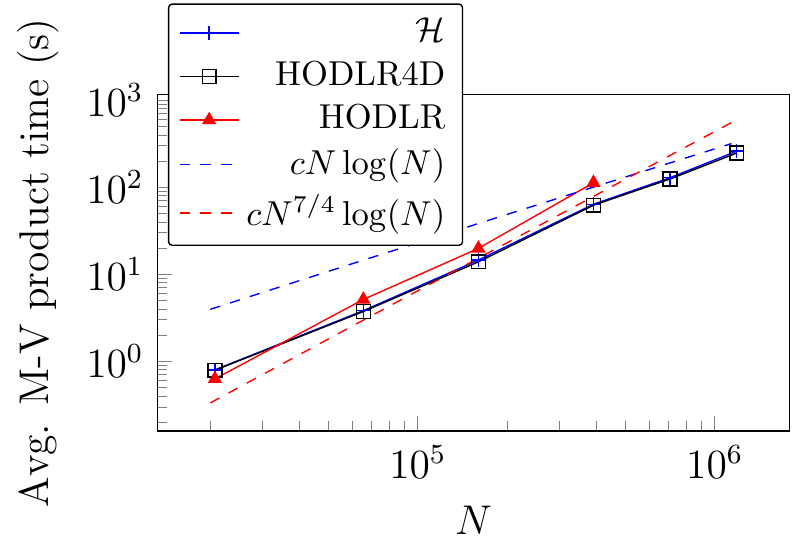}\label{ieq_mvp}}\quad%
    \subfloat[]{\includegraphics[height=3cm, width=4.5cm]{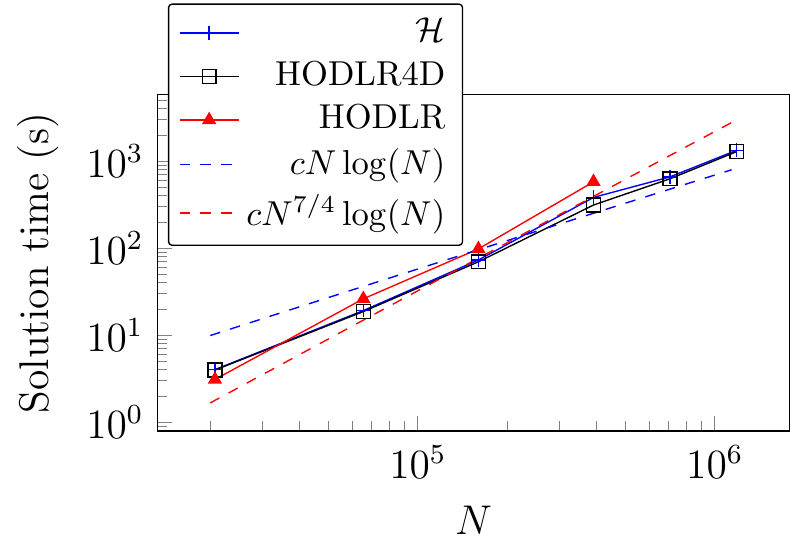}\label{ieq_sol}}\quad%
    \subfloat[]{\includegraphics[height=3cm, width=5cm]{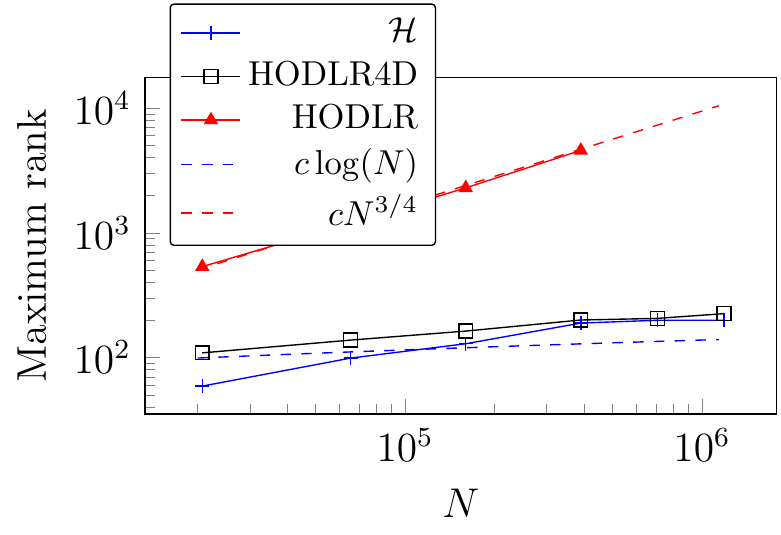}\label{ieq_rank}}\quad%
    \subfloat[]{\includegraphics[height=3cm, width=5cm]{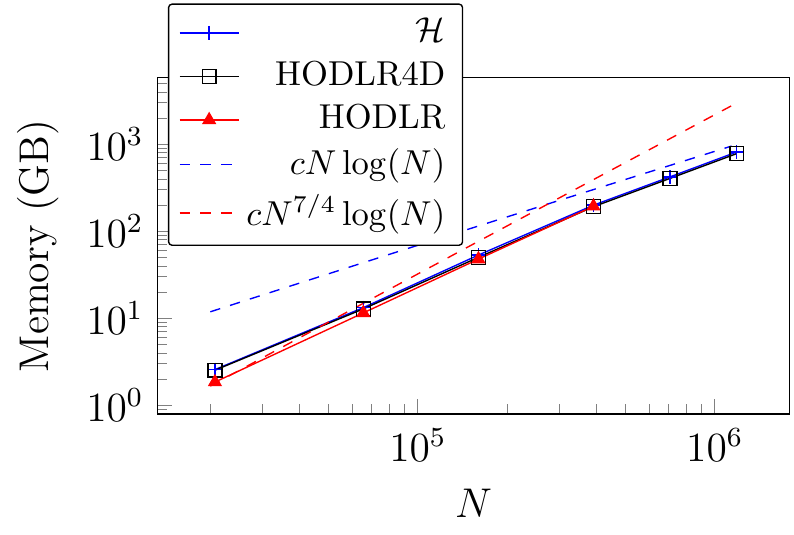}\label{ieq_mem}}\quad%
    \subfloat[]{\includegraphics[height=3cm, width=5cm]{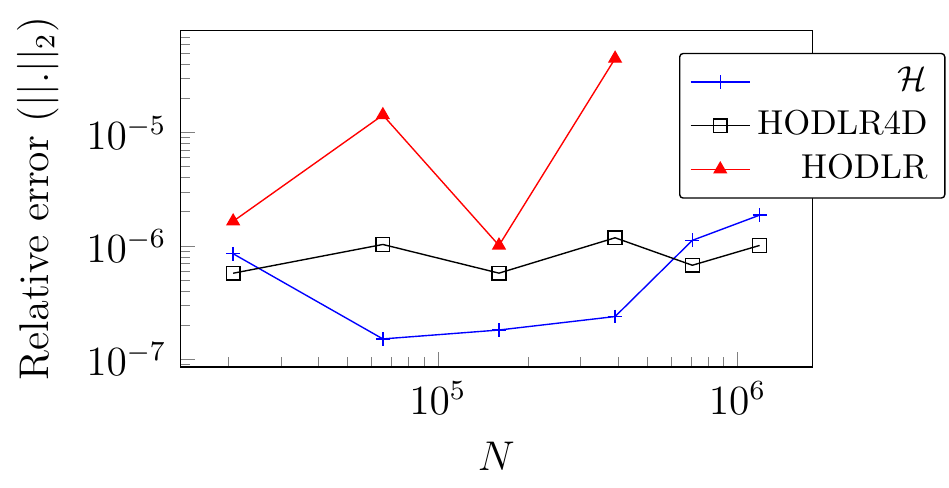}\label{ieq_er}}\quad%
    \caption{Comparison of $\mathcal{H}$, HODLR4D $(d=4)$ and HODLR accelerated GMRES (stopping tolerance $=10^{-6}$) to solve the system \cref{ieq_2}} 
    \label{fig:GMRES}
\end{figure}
\subsection{HODLR$d$D accelerated classical kernel SVM with $d$ features} \label{k_svm} In the last experiment, we chose random data points to ensure that the HODLR$d$D algorithm still applies to the \textbf{non-uniform} distribution of particles. The kernel Support Vector Machine~\cite{boser1992training}  (SVM) is a popular technique for classification problems in the supervised machine learning model. We briefly discuss how to accelerate the classical SVM for binary classification with the data set consisting of data points having $d$ features using our HODLR$d$D algorithm. The main goal of kernel SVM is to solve the objective function
\begin{equation} \label{svm_eq1}
    \underset{\pmb{\alpha}}{\operatorname{argmax}} \dsum_{i=1}^N \alpha_i - \dfrac{1}{2} \dsum_{i=1}^N \dsum_{j=1}^N K \bkt{\pmb{x}_i, \pmb{x}_j} y_i y_j \alpha_i \alpha_j
\end{equation}
subject to $0 \leq \alpha_i \leq \lambda$, $\dsum_{i=1}^N \alpha_i y_i = 0$ and $\pmb{\alpha} = \bkt{\alpha_1, \alpha_2, \dots , \alpha_N}$. 
Where $\pmb{x}_i \in \mathbb{R}^d$ is a data point with $d$ features, $K \bkt{\pmb{x}_i, \pmb{x}_j}$ is the kernel function evaluation at the data points $\pmb{x}_i$ and $\pmb{x}_j$. The label of $\pmb{x}_i$ is defined as 
            \[y_i = \begin{cases}
            1 & \text{ if } \pmb{x}_i \in \text{ Class 1} \\
            -1 & \text{ if } \pmb{x}_i \in \text{ Class 2} 
            \end{cases}\]
$i=1,2,\dots,N$.
The \cref{svm_eq1} can be cast as solving an optimization problem using Lagrange multipliers.
\newline
\begin{equation} \label{svm_eq2}
    L\bkt{\pmb{\alpha}} = \dsum_{i=1}^N \alpha_i - \dfrac{1}{2} \dsum_{i=1}^N \dsum_{j=1}^N K \bkt{\pmb{x}_i, \pmb{x}_j} y_i y_j \alpha_i \alpha_j - \dfrac{1}{2} \beta \bkt{\dsum_{j=1}^N \alpha_j y_j}^2
\end{equation}
We can use the gradient descent method to get the $\pmb{\alpha}$ for which the objective function is maximum.
\newline
\begin{equation} \label{svm_eq3}
    \dfrac{\partial L\bkt{\pmb{\alpha}}}{\partial \alpha_i} = 1 - \dsum_{j=1}^N K \bkt{\pmb{x}_i, \pmb{x}_j} y_i y_j \alpha_j -  \beta \dsum_{j=1}^N \alpha_j y_j y_i
\end{equation} 
\begin{equation} \label{svm_eqgd}
    \alpha_i \leftarrow \alpha_i + \eta \dfrac{\partial L\bkt{\pmb{\alpha}}}{\partial \alpha_i}
\end{equation}
Where $\eta$ is the learning rate and $i=1,2,\dots,N$. In SVM, we often iterate the~\cref{svm_eq3} for training purposes. In matrix-vector format the \cref{svm_eq3} can be written as
\begin{equation} \label{svm_eq4}
    \begin{bmatrix}
        \dfrac{\partial L\bkt{\pmb{\alpha}}}{\partial \alpha_1} \\
        \vdots  \\
        \dfrac{\partial L\bkt{\pmb{\alpha}}}{\partial \alpha_N} 
    \end{bmatrix}
    = 
    \begin{bmatrix}
        1 \\
        \vdots  \\
        1
    \end{bmatrix}
    -  
    \begin{bmatrix}
        y_1 K(x_1,x_1) & \dots  & y_1 K(x_1,x_N) \\
        \vdots & \vdots & \vdots \\
       y_N K(x_N,x_1) & \dots  & y_N K(x_N,x_N) \\
    \end{bmatrix}
\times
    \begin{bmatrix}
        y_1 \alpha_1 \\
        \vdots  \\
        y_N \alpha_N
    \end{bmatrix}
- \beta
    \begin{bmatrix}
        y_1   & \dots  & y_1 \\
        \vdots & \vdots & \vdots \\
       y_N   & \dots  & y_N  \\
    \end{bmatrix}
\times
    \begin{bmatrix}
        y_1 \alpha_1 \\
        \vdots  \\
        y_N \alpha_N
    \end{bmatrix}
\end{equation}
The \cref{svm_eq4} can be rewritten using MATLAB notations as
\begin{equation} \label{svm_eq5}
    \pmb{DL \bkt{\alpha}} = 1 - \pmb{y} \cdot* \bkt{K \pmb{v}} - \beta * \text{sum}(\pmb{v}) * \pmb{y}
\end{equation}
where
$\pmb{DL \bkt{\alpha}} = 
    \begin{bmatrix}
        \dfrac{\partial L\bkt{\pmb{\alpha}}}{\partial \alpha_1} \\
        \vdots  \\
        \dfrac{\partial L\bkt{\pmb{\alpha}}}{\partial \alpha_N} 
    \end{bmatrix},
    \pmb{y} = 
    \begin{bmatrix}
        y_1  \\
        \vdots  \\
        y_N
    \end{bmatrix},
    \pmb{v}=
    \begin{bmatrix}
        y_1 \alpha_1 \\
        \vdots  \\
        y_N \alpha_N
    \end{bmatrix}
$
and
$K = 
    \begin{bmatrix}
        K(x_1,x_1) & \dots  & K(x_1,x_N) \\
        \vdots & \vdots & \vdots \\
       K(x_N,x_1) & \dots  & K(x_N,x_N) \\
    \end{bmatrix}
$
is the kernel matrix.
Therefore, the \cref{svm_eqgd} becomes
\begin{equation} \label{svm_eq6}
\pmb{\alpha} \leftarrow \pmb{\alpha} + \eta \pmb{DL \bkt{\alpha}}
\end{equation}
If we iterate \cref{svm_eq6} $m$ times, then the naive matrix-vector product cost $\mathcal{O}\bkt{m N^2}$, and this makes the training time very expensive. In the Fast SVM algorithm, we accelerate the matrix-vector product $\bkt{K \pmb{v}}$ part of \cref{svm_eq5} using our HODLR$d$D algorithm, which scales almost linearly. So, the training time can be reduced using FSVM. We produce some synthetic data sets to check the performance of the SVM models as follows.
\begin{enumerate}
    \item Create randomly distributed $n$ particles in $[-1,1]$.
    \item To create synthetic data sets of $d$ features, we take the $n^d$ tesnor product grid of randomly distributed particles in $[-1,1]^d$.
    \item We create two different data sets $D_1$, with label $1$ and $D_2$, with label $-1$ by separating the data points by a function such that $ n^d =  \# D_1 + \# D_2$. The data points in $D_1$ and $D_2$ are associated with Class $1$ and Class $2$, respectively.  
    \item We create the $i^{th}$ class training set $C_i$ by taking $85 \%$ of $D_i$. We store the remaining $15\%$ in the set $c_i$ for testing purposes. i.e., $D_i = C_i \cup c_i $ with $C_i \cap c_i = \phi$, for $i=1,2$.
\end{enumerate}

We define some notations in~\cref{tab:SVM_notations}, which will be used in this subsection only.
 \begin{table}[H]
     \centering
     \resizebox{\textwidth}{!}{\begin{tabular}{|c|c|}
     \hline
        \#$C_i$ &  Size of the $i^{th}$ class training set $C_i$. $i=1,2$ \\
    \hline
        \#$c_i$  &  Size of the $i^{th}$ class testing set $c_i$. $i=1,2$  \\
    \hline
        $t_{NSVM}$  &  \makecell{Total time (in seconds) to train the Naive SVM model using the training sets. \\ In NSVM, the matrix-vector products are performed using the \href{https://eigen.tuxfamily.org}{Eigen} library.}   \\
    \hline
        $t_{FSVM}$  &  \makecell{Total time (in seconds) to train the Fast SVM model using the training sets \\ (including the HODLR$d$D initialization time (\Cref{alg:initialization})). \\ In FSVM, the matrix-vector products are performed using the HODLR$d$D algorithm. \\ (ACA tolerance $= 10^{-10}$ and maximum number of particles at leaf level, $N_{max} = 500$}   \\
    \hline
    \textit{iter}  &  The total number of iterations to train the SVM model  \\
    \hline
    ${i}_{NSVM}$  &  \makecell{Time per iteration of NSVM (in seconds) ${i}_{NSVM} = t_{NSVM}/\textit{iter} $ }\\
    \hline
    ${i}_{FSVM}$  &  \makecell{Time per iteration of FSVM (in seconds) ${i}_{FSVM} = t_{FSVM}/\textit{iter} $ }\\
    \hline
    \text{Classification score}  & We check the testing set after training. $f(\pmb{x}) = \dsum_{i=1}^{N} \alpha_i y_i K \bkt{\pmb{x}_i, \pmb{x}} + b$, where $b$ is the bias.   \\
    \hline
    \makecell{Label function \\ to classify the testing set}  &  $\makecell{g(\pmb{x}) = sgn \bkt{f(\pmb{x})} \\ 
            \text{For } \pmb{x} \in c_1 \cup c_2 \text{, if } g(\pmb{x}) = \begin{cases}
            1 & \text{ then } \pmb{x} \in s_1 \\
            -1 & \text{ the } \pmb{x} \in s_2
            \end{cases}} $\\
    \hline
    A$_1$  &  Accuracy of Class 1  $= \bkt{\dfrac{\# s_1}{\# c_1} \times 100} \%$ \\
    \hline
    A$_2$   &  Accuracy of Class 2 $= \bkt{\dfrac{\# s_2}{\# c_2} \times 100} \%$ \\
    \hline
    OA  & Overall accuracy of the model $= \bkt{\dfrac{\# s_1 + \# s_2}{\# c_1 + \# c_2} \times 100} \%$ \\
    \hline
     \end{tabular}}
     \caption{Notations for this subsection (\Cref{k_svm})}
     \label{tab:SVM_notations}
 \end{table}
 Though our HODLR$d$D code applies for any dimension $d$, we choose $d=4$ and $d=5$ due to limited RAM and other computational constraints. We present the performance of the HODLR$d$D accelerated Fast SVM (FSVM) over the Naive SVM (NSVM) in~\cref{tab:SVM_4,tab:SVM_5} for the data sets with four (kernel is Matérn kernel) and five features (kernel is Gaussian kernel), respectively. All the data sets used in~\cref{tab:SVM_4,tab:SVM_5} are different in each case, and we compare the NSVM and FSVM with the same hyperparameters configuration (e.g., learning rate, etc.). We can see from~\cref{tab:SVM_4,tab:SVM_5} that the FSVM reduces the training time, which is the most expensive part of the SVM algorithm. Hence, we can achieve the same accuracy of the SVM model in a faster way using the HODLR$d$D algorithm. Also, the HODLR$d$D algorithm can handle the non-uniform distribution of particles (like FMM) for higher dimensional problems. All the SVM implementations were done in $\texttt{C++}$, and the computations for this experiment were performed on a quad-core, 1.4 GHz Intel Core i5 processor with 8GB RAM.
 \begin{table}[H]
\begin{tabular}{|l|ll|ll|ll|ll|l|l|l|}
\hline
$n$ & \multicolumn{2}{l|}{\#Training set} & \multicolumn{2}{l|}{\#Testing set} & \multicolumn{2}{l|}{Training time (s)}    & \multicolumn{2}{l|}{Time per iter.(s)} & \multirow{2}{*}{A$_1 (\%)$} & \multirow{2}{*}{A$_2 (\%)$} & \multirow{2}{*}{OA (\%)} \\ \cline{1-9}
  & \multicolumn{1}{l|}{\#$C_1$}      & \#$C_2$     & \multicolumn{1}{l|}{\#$c_1$}    & \#$c_2$   & \multicolumn{1}{l|}{$t_{NSVM}$} & $t_{FSVM}$ & \multicolumn{1}{l|}{${i}_{NSVM}$}  & ${i}_{FSVM}$  &                      &                      &                    \\ \hline
 8 & \multicolumn{1}{l|}{1740}         &  1742       & \multicolumn{1}{l|}{307}       &  307     & \multicolumn{1}{l|}{204.07}       & 48.67        & \multicolumn{1}{l|}{0.20}        &  0.04        &    94.13                  &       98.04               &      96.09              \\ \hline
10  & \multicolumn{1}{l|}{4155}         &  4346       & \multicolumn{1}{l|}{733}       &   766    & \multicolumn{1}{l|}{1429.37}       &  134.75       & \multicolumn{1}{l|}{1.51}        &  0.14        &   95.63                   &        99.61              &    97.67                \\ \hline
 11 & \multicolumn{1}{l|}{6387}         &  6058       & \multicolumn{1}{l|}{1127}       &  1069     & \multicolumn{1}{l|}{3891.57}       &  237.65       & \multicolumn{1}{l|}{3.06}        &  0.18        &  97.87                    &     99.15                 &     98.49               \\ \hline
\end{tabular}
     \caption{Performance comparison with the data sets of four features and Matérn kernel $\bkt{e^{-r}}$. In FSVM, the matrix-vector product is accelerated using HODLR4D $(d=4)$}
     \label{tab:SVM_4}
\end{table}
\begin{table}[H]
\begin{tabular}{|l|ll|ll|ll|ll|l|l|l|}
\hline
$n$ & \multicolumn{2}{l|}{\#Training set} & \multicolumn{2}{l|}{\#Testing set} & \multicolumn{2}{l|}{Training time (s)}    & \multicolumn{2}{l|}{Time per iter.(s)} & \multirow{2}{*}{A$_1 (\%)$} & \multirow{2}{*}{A$_2 (\%)$} & \multirow{2}{*}{OA (\%)} \\ \cline{1-9}
  & \multicolumn{1}{l|}{\#$C_1$}      & \#$C_2$     & \multicolumn{1}{l|}{\#$c_1$}    & \#$c_2$   & \multicolumn{1}{l|}{$t_{NSVM}$} & $t_{FSVM}$ & \multicolumn{1}{l|}{${i}_{NSVM}$}  & ${i}_{FSVM}$  &                      &                      &                    \\ \hline
 5 & \multicolumn{1}{l|}{1416}         &  1242       & \multicolumn{1}{l|}{248}       &  219     & \multicolumn{1}{l|}{94.71}       & 9.10        & \multicolumn{1}{l|}{0.11}        &  0.01        &    94.78                  &       95.43               &      95.08              \\ \hline
6  & \multicolumn{1}{l|}{3367}         &  3244       & \multicolumn{1}{l|}{593}       &   572    & \multicolumn{1}{l|}{1731.32}       &  111.68       & \multicolumn{1}{l|}{1.02}        &  0.06        &   92.59                   &        97.20              &    94.85                \\ \hline
\end{tabular}
     \caption{Performance comparison with the data sets of five features and Gaussian kernel $\bkt{e^{-r^2}}$. In FSVM, the matrix-vector product is accelerated using HODLR5D $(d=5)$}
     \label{tab:SVM_5}
\end{table}
\section{Conclusion} We have proved two theorems regarding the rank growth of all the interaction matrices in $d$ dimensions. As a consequence of the theorems, we see that the rank of not just the \emph{far-field} but also the \emph{vertex-sharing} interaction matrices do not scale with any power of $N$. The rank of other hyper-surface sharing interactions scale with some power of $N$. Based on our theorems, we try to extend the notion of \textit{weak admissibility} to higher dimensions and develop an almost linear complexity black-box (kernel independent) algorithm (HODLR$d$D) for $N$-body problems in $d$ dimensions. The $\texttt{C++}$ implementation with \texttt{OpenMP} parallelization of the HODLR$d$D algorithm is available at \texttt{\url{https://github.com/SAFRAN-LAB/HODLRdD}}. The code's speciality is that it is applicable for any value of dimension $d$. We apply our HODLR$d$D algorithm to perform a fast matrix-vector product, solve a $4$D integral equation and accelerate the classical kernel SVM with four and five features. Also, we compare the performance of HODLR$d$D over $\mathcal{H}$ matrix with strong admissibility and HODLR. The numerical results demonstrate that the HODLR$d$D performs way better than the HODLR and is highly competitive to the $\mathcal{H}$ matrix with strong admissibility condition (most of the times HODLR$d$D performs better than $\mathcal{H}$ matrix with strong admissibility). 
We also hope our theorems could be leveraged to build other newer, faster algorithms. A suitable hierarchical representation leveraging the low-rank of nearby interactions could be used to construct direct solvers.
\section{Acknowledgments} We would like to thank Vaishnavi Gujjula for fruitful discussions and also for proof-reading the article. The last author would like to thank the MATRICS grant from SERB (Sanction number: MTR/2019/001241) and YSRA from BRNS, DAE (No.34/20/03/2017-BRNS/34278). We acknowledge the use of the computing resources at HPCE, IIT Madras.

\bibliographystyle{siamplain}
\bibliography{refs}

\begin{thebibliography}{10}

\bibitem{ambikasaran2015generalized}
{\sc S.~Ambikasaran}, {\em Generalized {R}ybicki {P}ress algorithm}, Numerical Linear Algebra with Applications, 22 (2015), pp.~1102--1114.

\bibitem{ambikasaran2013mathcal}
{\sc S.~Ambikasaran and E.~Darve}, {\em An $\mathcal{O}\left({N} \log {N}\right)$ fast direct solver for partial hierarchically semi-separable matrices}, Journal of Scientific Computing, 57 (2013), pp.~477--501.

\bibitem{ifmm}
{\sc S.~Ambikasaran and E.~Darve}, {\em The inverse fast multipole method}, arXiv preprint arXiv:1407.1572,  (2014).

\bibitem{ambikasaran2015fast}
{\sc S.~Ambikasaran, D.~Foreman-Mackey, L.~Greengard, D.~W. Hogg, and M.~O’Neil}, {\em Fast direct methods for gaussian processes}, IEEE transactions on pattern analysis and machine intelligence, 38 (2015), pp.~252--265.

\bibitem{ambikasaran2013large}
{\sc S.~Ambikasaran, J.~Y. Li, P.~K. Kitanidis, and E.~Darve}, {\em Large-scale stochastic linear inversion using hierarchical matrices}, Computational Geosciences, 17 (2013), pp.~913--927.

\bibitem{ambikasaran2014fast}
{\sc S.~Ambikasaran, M.~O'Neil, and K.~R. Singh}, {\em Fast symmetric factorization of hierarchical matrices with applications}, arXiv preprint arXiv:1405.0223,  (2014).

\bibitem{ambikasaran2013fast}
{\sc S.~Ambikasaran, A.~K. Saibaba, E.~F. Darve, and P.~K. Kitanidis}, {\em Fast algorithms for bayesian inversion}, in Computational Challenges in the Geosciences, Springer, 2013, pp.~101--142.

\bibitem{ambikasaran2019hodlrlib}
{\sc S.~Ambikasaran, K.~R. Singh, and S.~S. Sankaran}, {\em {HODLRlib}: {A} library for hierarchical matrices}, Journal of Open Source Software, 4 (2019), p.~1167.

\bibitem{barnes1986hierarchical}
{\sc J.~Barnes and P.~Hut}, {\em A hierarchical $\mathcal{O} ({N} \log {N})$ force-calculation algorithm}, nature, 324 (1986), pp.~446--449.

\bibitem{shortFMM}
{\sc R.~Beatson and L.~Greengard}, {\em A short course on fast multipole methods}, Wavelets, multilevel methods and elliptic PDEs, 1 (1997), pp.~1--37.

\bibitem{bebendorf2009recompression}
{\sc M.~Bebendorf and S.~Kunis}, {\em Recompression techniques for adaptive cross approximation}, The Journal of Integral Equations and Applications,  (2009), pp.~331--357.

\bibitem{bebendorf2003adaptive}
{\sc M.~Bebendorf and S.~Rjasanow}, {\em Adaptive low-rank approximation of collocation matrices}, Computing, 70 (2003), pp.~1--24.

\bibitem{borm2003hierarchical}
{\sc S.~B{\"o}rm, L.~Grasedyck, and W.~Hackbusch}, {\em Hierarchical matrices}, Lecture notes, 21 (2003), p.~2003.

\bibitem{borm2003introduction}
{\sc S.~B{\"o}rm, L.~Grasedyck, and W.~Hackbusch}, {\em Introduction to hierarchical matrices with applications}, Engineering analysis with boundary elements, 27 (2003), pp.~405--422.

\bibitem{boser1992training}
{\sc B.~E. Boser, I.~M. Guyon, and V.~N. Vapnik}, {\em A training algorithm for optimal margin classifiers}, in Proceedings of the fifth annual workshop on Computational learning theory, 1992, pp.~144--152.

\bibitem{cai2018smash}
{\sc D.~Cai, E.~Chow, L.~Erlandson, Y.~Saad, and Y.~Xi}, {\em Smash: Structured matrix approximation by separation and hierarchy}, Numerical Linear Algebra with Applications, 25 (2018), p.~e2204.

\bibitem{carpentieri2004sparse}
{\sc B.~Carpentieri, I.~S. Duff, L.~Giraud, and M.~Magolu~monga Made}, {\em Sparse symmetric preconditioners for dense linear systems in electromagnetism}, Numerical linear algebra with applications, 11 (2004), pp.~753--771.

\bibitem{carr1997surface}
{\sc J.~C. Carr, W.~R. Fright, and R.~K. Beatson}, {\em Surface interpolation with radial basis functions for medical imaging}, IEEE transactions on medical imaging, 16 (1997), pp.~96--107.

\bibitem{carrier1988fast}
{\sc J.~Carrier, L.~Greengard, and V.~Rokhlin}, {\em A fast adaptive multipole algorithm for particle simulations}, SIAM journal on scientific and statistical computing, 9 (1988), pp.~669--686.

\bibitem{chandrasekaran2007fast}
{\sc S.~Chandrasekaran, P.~Dewilde, M.~Gu, W.~Lyons, and T.~Pals}, {\em A fast solver for {HSS} representations via sparse matrices}, SIAM Journal on Matrix Analysis and Applications, 29 (2007), pp.~67--81.

\bibitem{cheng1999fast}
{\sc H.~Cheng, L.~Greengard, and V.~Rokhlin}, {\em A fast adaptive multipole algorithm in three dimensions}, Journal of computational physics, 155 (1999), pp.~468--498.

\bibitem{corona2015n}
{\sc E.~Corona, P.-G. Martinsson, and D.~Zorin}, {\em An o (n) direct solver for integral equations on the plane}, Applied and Computational Harmonic Analysis, 38 (2015), pp.~284--317.

\bibitem{engquist2007fast}
{\sc B.~Engquist and L.~Ying}, {\em Fast directional multilevel algorithms for oscillatory kernels}, SIAM Journal on Scientific Computing, 29 (2007), pp.~1710--1737.

\bibitem{fong2009black}
{\sc W.~Fong and E.~Darve}, {\em The black-box fast multipole method}, Journal of Computational Physics, 228 (2009), pp.~8712--8725.

\bibitem{foreman2017fast}
{\sc D.~Foreman-Mackey, E.~Agol, S.~Ambikasaran, and R.~Angus}, {\em Fast and scalable gaussian process modeling with applications to astronomical time series}, The Astronomical Journal, 154 (2017), p.~220.

\bibitem{Gab19}
{\sc M.~Ga{\ss}, K.~Glau, M.~Mahlstedt, and M.~Mair}, {\em Chebyshev interpolation for parametric option pricing}, Finance and Stochastics, 22 (2018), pp.~701--731.

\bibitem{gillman2012direct}
{\sc A.~Gillman, P.~M. Young, and P.-G. Martinsson}, {\em A direct solver with $\mathcal{O} ({N})$ complexity for integral equations on one-dimensional domains}, Frontiers of Mathematics in China, 7 (2012), pp.~217--247.

\bibitem{glau}
{\sc K.~Glau and M.~Mahlstedt}, {\em Improved error bound for multivariate chebyshev polynomial interpolation}, International Journal of Computer Mathematics, 96 (2019), pp.~2302--2314.

\bibitem{grasedyck2003construction}
{\sc L.~Grasedyck and W.~Hackbusch}, {\em Construction and arithmetics of h-matrices}, Computing, 70 (2003), pp.~295--334.

\bibitem{gray2000}
{\sc A.~Gray and A.~Moore}, {\em N-body'problems in statistical learning}, Advances in neural information processing systems, 13 (2000).

\bibitem{greengard1988rapid}
{\sc L.~Greengard}, {\em The rapid evaluation of potential fields in particle systems}, MIT press, 1988.

\bibitem{greengard1987fast}
{\sc L.~Greengard and V.~Rokhlin}, {\em A fast algorithm for particle simulations}, Journal of computational physics, 73 (1987), pp.~325--348.

\bibitem{greengard1997new}
{\sc L.~Greengard and V.~Rokhlin}, {\em A new version of the fast multipole method for the laplace equation in three dimensions}, Acta numerica, 6 (1997), pp.~229--269.

\bibitem{radial}
{\sc N.~A. Gumerov and R.~Duraiswami}, {\em Fast radial basis function interpolation via preconditioned krylov iteration}, SIAM Journal on Scientific Computing, 29 (2007), pp.~1876--1899.

\bibitem{hackbusch2004hierarchical}
{\sc W.~Hackbusch, B.~N. Khoromskij, and R.~Kriemann}, {\em Hierarchical matrices based on a weak admissibility criterion}, Computing, 73 (2004), pp.~207--243.

\bibitem{ho2012fast}
{\sc K.~L. Ho and L.~Greengard}, {\em A fast direct solver for structured linear systems by recursive skeletonization}, SIAM Journal on Scientific Computing, 34 (2012), pp.~A2507--A2532.

\bibitem{ho2013hierarchical}
{\sc K.~L. Ho and L.~Ying}, {\em Hierarchical interpolative factorization for elliptic operators: integral equations}, arXiv preprint arXiv:1307.2666,  (2013).

\bibitem{kandappan2022hodlr2d}
{\sc V.~A. Kandappan, V.~Gujjula, and S.~Ambikasaran}, {\em {HODLR2D}: A new class of hierarchical matrices}, arXiv preprint arXiv:2204.05536,  (2022).

\bibitem{li2014kalman}
{\sc J.~Y. Li, S.~Ambikasaran, E.~F. Darve, and P.~K. Kitanidis}, {\em A {K}alman filter powered by-matrices for quasi-continuous data assimilation problems}, Water Resources Research, 50 (2014), pp.~3734--3749.

\bibitem{randomized_aca}
{\sc E.~Liberty, F.~Woolfe, P.-G. Martinsson, V.~Rokhlin, and M.~Tygert}, {\em Randomized algorithms for the low-rank approximation of matrices}, Proceedings of the National Academy of Sciences, 104 (2007), pp.~20167--20172.

\bibitem{ying}
{\sc L.~Lin, J.~Lu, and L.~Ying}, {\em Fast construction of hierarchical matrix representation from matrix--vector multiplication}, Journal of Computational Physics, 230 (2011), pp.~4071--4087.

\bibitem{gp_book}
{\sc C.~E. Rasmussen}, {\em Gaussian Processes in Machine Learning}, Springer Berlin Heidelberg, Berlin, Heidelberg, 2004, pp.~63--71, \url{https://doi.org/10.1007/978-3-540-28650-9_4}, \url{https://doi.org/10.1007/978-3-540-28650-9_4}.

\bibitem{saad1986gmres}
{\sc Y.~Saad and M.~H. Schultz}, {\em Gmres: A generalized minimal residual algorithm for solving nonsymmetric linear systems}, SIAM Journal on scientific and statistical computing, 7 (1986), pp.~856--869.

\bibitem{Trefethen}
{\sc L.~N. Trefethen}, {\em Approximation Theory and Approximation Practice, Extended Edition}, SIAM, 2019.

\bibitem{tyrtyshnikov2000incomplete}
{\sc E.~Tyrtyshnikov}, {\em Incomplete cross approximation in the mosaic-skeleton method}, Computing, 64 (2000), pp.~367--380.

\bibitem{wang2018numerical}
{\sc R.~Wang, Y.~Li, and E.~Darve}, {\em On the numerical rank of radial basis function kernels in high dimensions}, SIAM Journal on Matrix Analysis and Applications, 39 (2018), pp.~1810--1835.

\bibitem{xia2021multi}
{\sc J.~Xia}, {\em Multi-layer hierarchical structures}, CSIAM Transactions on Applied Mathematics, 2 (2021).

\bibitem{xia2010fast}
{\sc J.~Xia, S.~Chandrasekaran, M.~Gu, and X.~S. Li}, {\em Fast algorithms for hierarchically semiseparable matrices}, Numerical Linear Algebra with Applications, 17 (2010), pp.~953--976.

\bibitem{ker_den}
{\sc C.~Yang, R.~Duraiswami, N.~A. Gumerov, and L.~Davis}, {\em Improved fast gauss transform and efficient kernel density estimation}, in Computer Vision, IEEE International Conference on, vol.~2, IEEE Computer Society, 2003, pp.~464--464.

\bibitem{ying2004kernel}
{\sc L.~Ying, G.~Biros, and D.~Zorin}, {\em A kernel-independent adaptive fast multipole algorithm in two and three dimensions}, Journal of Computational Physics, 196 (2004), pp.~591--626.

\bibitem{fmm_ref}
{\sc R.~Yokota, H.~Ibeid, and D.~Keyes}, {\em Fast multipole method as a matrix-free hierarchical low-rank approximation}, in International Workshop on Eigenvalue Problems: Algorithms, Software and Applications in Petascale Computing, Springer, 2015, pp.~267--286.

\end{thebibliography}

\newpage
\appendix
\section{Rank growth of different interactions in 1D} \label{1d}
 Using the~\cref{th3}, we discuss the rank of far-field and vertex-sharing interactions as shown in~\cref{1d_interaction}. $I$ and V are the far-field and vertex-sharing domains of $Y$, respectively. 
     \begin{figure}[H]
     	\centering \subfloat[Far field and vertex-sharing domains of $Y$]{\label{1d_interaction}\resizebox{4cm}{!}{
            \begin{tikzpicture}
            \draw[ultra thick] (0,0) -- (3,0);
            \redcircle{0}{0};
            \redcircle{1}{0};
            \redcircle{2}{0};
            \redcircle{3}{0};
            \node at (0.5,0.2) {$I$};
            \node at (1.5,0.2) {V};
            \node at (2.5,0.2) {$Y$};
            \end{tikzpicture}
        }}
        \qquad \qquad \qquad
        \centering \subfloat[Far-field domain]{\label{1d_far}\resizebox{4cm}{!}{
        \begin{tikzpicture}
            \draw[|-|,dashed] (0,0) node[anchor=north] {$-r$} -- (2,0);
            \draw[|-|] (-2,0) node[anchor=north] {$-2r$} -- (0,0);
            \draw[|-|] (2,0) -- (4,0);   
            \node[anchor=north] at (2,0) {$0$};
            \node[anchor=north] at (4,0) {$r$};
            
            \node[anchor=north] at (-1,0.75) {X};
            \node[anchor=north] at (3,0.75) {Y};
        \end{tikzpicture}
        }}
        \caption{Different interactions in $1$D}
    \end{figure}
    \subsection{\textbf{\textit{Rank growth of far-field domains}}} Here, we will assume far-field interaction means the interaction between two domains which are one line segment away. In~\cref{1d_far} the lines $Y = [0,r]$ and $X = [-2r,-r]$ are one line-segment away. We choose $\bkt{p+1}$ Chebyshev nodes in the domain $Y$ to obtain the polynomial interpolation of $F(\xb,\yb)$ along $\yb$. Let $M_i = \displaystyle \sup_{\yb \in \mathcal{B}\bkt{{Y,\rho}}} \abs{F_a\bkt{\xb_i,\yb}}$ ,  $\rho \in (1,\alpha)$, for some $\alpha >1$. Let $\Tilde{K}$ be the approximation of the matrix \(K\). Then by~\cref{th3} the absolute error along the $i^{th}$ row is given by (MATLAB notation)
    \begin{equation} \label{ap_1dfar}
        \Bigl\lvert \bkt{K \bkt{i,:} -\Tilde{K}\bkt{i,:} } \Bigl\lvert \leq \frac{4M_i \rho^{-p}}{\rho -1} , \quad 1 \leq i \leq N
    \end{equation}
         Let $M = \displaystyle \max_{1 \leq i\leq N} \{M_i\}$. Therefore,
         \begin{equation}
             \magn{K  - \Tilde{K}}_{max} \leq 4 M \dfrac{\rho^{-p}}{\rho -1} 
             \implies \dfrac{\magn{K  - \Tilde{K}}_{max}}{\magn{K}_{max}} \leq \dfrac{4M}{\magn{K}_{max}} \dfrac{\rho^{-p}}{\rho -1} = \dfrac{c\rho^{-p}}{\rho -1}
         \end{equation}
    where $c = \dfrac{4M}{\magn{K}_{max}}$. Now, choosing $p$ such that the above relative error to be less than $\delta$ (for some $\delta >0$), we obtain
        $p = \ceil{\frac{ \log \bkt{\frac{c}{\delta (\rho-1)}}}{\log(\rho)}} \implies \dfrac{\magn{K  - \Tilde{K}}_{max}}{\magn{K}_{max}} < \delta$.
Since, the rank of $\Tilde{K} = (p+1)$, i.e., the rank is bounded above by $\bkt{1+\ceil{\frac{\log\bkt{\frac{c}{\delta(\rho-1)}}}{\log(\rho)}}}$.
    Therefore, the rank of $\Tilde{K}$ scales $\mathcal{O}\bkt{ \log\bkt{\frac{1}{\delta}}}$ with
    $\dfrac{\magn{K  - \Tilde{K}}_{max}}{\magn{K}_{max}} < \delta$. The numerical rank plots of the far-field interaction of the eight functions as described in~\Cref{Preliminaries} are shown in~\cref{fig:res_1d_far_log} and tabulated in~\cref{tab:res_1d_far_log}.

\cmt{
    \begin{figure}[H]
    \centering
    \subfloat[][N vs Rank plot\label{fig:res_1d_farc}]{\includegraphics[width=0.5\columnwidth]{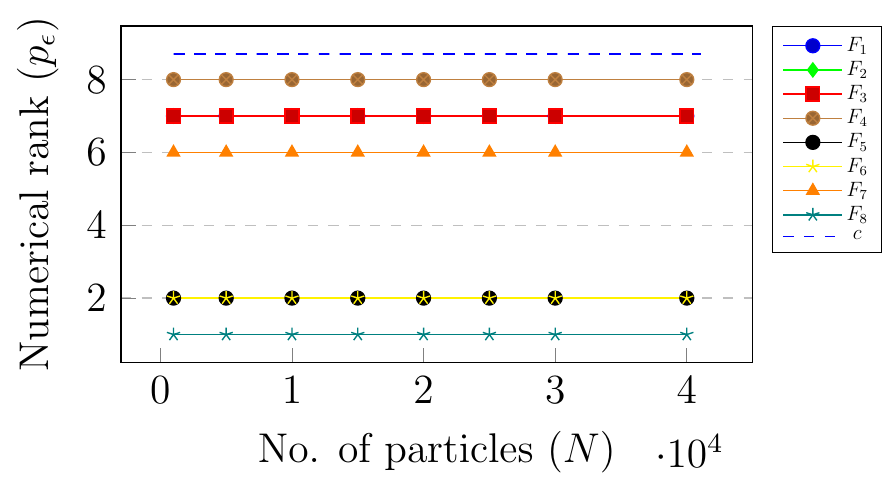}}\qquad
    \subfloat[][Numerical rank of different kernel functions]
    {\adjustbox{width=0.4\columnwidth,valign=B,raise=1.4\baselineskip}{%
      \renewcommand{\arraystretch}{1}%
      \label{tab:res_1d_farc}%
        \begin{tabular}{c|cccc}\hline
            $N$ & $F_1(x,y)$ & $F_2(x,y)$ & $F_3(x,y)$ & $F_4(x,y)$ \\ \hline
            1000 & 7 & 7 & 8 & 8 \\
            5000 & 7 & 7 & 8 & 8 \\
            10000 & 7 & 7 & 8 & 8 \\
            15000 & 7 & 7 & 8 & 8 \\
            20000 & 7 & 7 & 8 & 8 \\
            25000 & 7 & 7 & 8 & 8 \\
            30000 & 7 & 7 & 8 & 8 \\
            40000  & 7 & 7 & 8 & 8 \\ \hline
        \end{tabular}}
    }
    \caption{Numerical rank $\bkt{tol = 10^{-12}}$ of far-field interaction in $1$D}
    \label{result_1d_farc}
    \end{figure}
}
    \subsection{\textbf{\textit{Rank growth of vertex-sharing domains}}}
    Consider a pair of vertex-sharing line-segments (domains) $Y = [0,r]$ and $X = [-r,0]$ of length $r$ as shown in~\cref{1d_ver}. The domain \(Y\) is hierarchically sub-divided using an adaptive binary tree as shown in~\cref{1d_ver_sub_l1,1d_ver_sub}, which gives

    \begin{equation}
        Y = \overbrace{Y_2 \bigcup Y_1}^{\text{At level 1}} = \overbrace{\left[0,\dfrac{r}{2} \right] \bigcup  \left[\dfrac{r}{2},r \right]}^{\text{At level 1}} = \dots = \overbrace{Y_{\kappa +1} \bigcup_{k=1}^{\kappa} Y_{k}}^{\text{At level }\kappa}  = \overbrace{\left[0,\dfrac{r}{2^{\kappa}} \right] \bigcup_{k=1}^{\kappa} \left[\dfrac{r}{2^{k}}, \dfrac{r}{2^{k-1}} \right]}^{\text{At level }\kappa}
    \end{equation} 
    where $\kappa \sim \log_2(N)$ and $ Y_{\kappa +1 } = \left[0,\dfrac{r}{2^{\kappa}} \right]$ having one particle. Let $\Tilde{K_k}$ be the approximation of $K_k$, then the approximation of the matrix \(K\) is given by
    \begin{equation}
        \Tilde{K} = \dsum_{k=1}^{\kappa} \Tilde{K_k} + K_{\kappa+1}
    \end{equation}
    We choose a $\bkt{p_k+1}$ Chebyshev nodes in the line-segment $Y_{k}$ to obtain the polynomial interpolation of $F(\xb,\yb)$ along $\yb$. Let $M^{(i)}_{k} = \displaystyle \sup_{\yb \in \mathcal{B} \bkt{{Y_{k},\rho}}} \abs{F_a\bkt{\xb_i,\yb}}, \rho \in \bkt{1, \alpha}$, for some $\alpha>1$. Then by~\cref{th3} the absolute error along the $i^{th}$ row of the matrix $K_{k}$ is given by (MATLAB notation)
    \begin{equation}
        \Bigl\lvert \bkt{K_k \bkt{i,:}-\Tilde{K}_k \bkt{i,:}} \Bigl\lvert \leq \frac{M^{(i)}_{k} \rho^{-p_k}}{\rho -1}, \quad 1 \leq i \leq N, \quad 1 \leq k \leq \kappa
    \end{equation}
    Let $M = \max \{ M^{(i)}_{k} \} $. Therefore, 
    \begin{equation}
       \magn{K_k - \Tilde{K}_k}_{max} \leq \frac{M  \rho^{-p_k}}{\rho -1} \implies \dfrac{ \magn{K_k - \Tilde{K}_k}_{max}}{\magn{K}_{max}} \leq \frac{M  \rho^{-p_k}}{\magn{K}_{max}\bkt{\rho -1}} = \dfrac{c \rho^{-p_k}}{\rho -1}
    \end{equation}
    where $c = \dfrac{M}{\magn{K}_{max}}$. Now choosing $p_{k}$ such that the above error is less than $\delta_1$ (for some $\delta_1>0$), we obtain \\ $p_k = \bkt{\ceil{\frac{\log\bkt{\frac{c}{\delta_1(\rho-1)}}}{\log(\rho)}}} \implies \dfrac{ \magn{K_k - \Tilde{K}_k}_{max}}{\magn{K}_{max}} < \delta_1$ with rank of $\Tilde{K}_{k} $ is $ (1 + p_k)$.\\
    Let $p_l = \max \{p_{k} : k=1,2,\hdots, \kappa \}$, which corresponds to $\Tilde{K}_l$. Hence, we get
    \begin{align}
    \begin{split}
        \magn{K - \Tilde{K}}_{max} = \magn{ \bkt{K_1+ K_2+\cdots + K_\kappa} - \bkt{\Tilde{K_1}+\Tilde{K_2}+\cdots + \Tilde{K_\kappa}}}_{max} \\ \leq \dsum_{k=1}^\kappa \magn{K_k - \Tilde{K}_k }_{max}  < \kappa \delta_1 \magn{K}_{max}
    \end{split}
    \end{align}
    The rank of $\Tilde{K}$ is bounded above by 
    $\bkt{1+ \kappa \bkt{1+p_l}} = \bkt{1 + \kappa \bkt{1+\ceil{\frac{\log\bkt{\frac{c}{\delta_1(\rho-1)}}}{\log(\rho)}}}}$.
    Note that the rank of $K_{\kappa +1}$ is one. If we choose $\delta_1 = \dfrac{\delta}{\kappa}$, then $ \dfrac{\magn{K - \Tilde{K}}_{max}}{\magn{K}_{max}} < \delta$.

    Therefore, the rank of $\Tilde{K}$ scales $\mathcal{O}\bkt{\log_2(N) \log\bkt{\frac{ \log_2(N)}{\delta}}}$ with
    $ \dfrac{\magn{K - \Tilde{K}}_{max}}{\magn{K}_{max}} < \delta$. The numerical rank plots of the vertex-sharing interaction of the eight functions as described in~\Cref{Preliminaries} are shown in~\cref{fig:res_1d_ver_log} and tabulated in~\cref{tab:res_1d_vert_log}.
    
\cmt{
    \begin{figure}[H]
    \centering
    \subfloat[][N vs Rank plot]{\includegraphics[width=0.5\columnwidth]{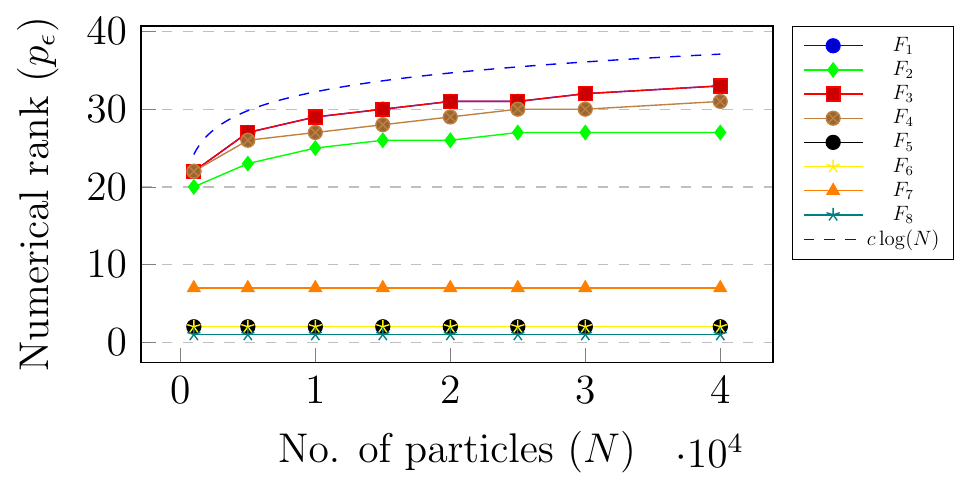}}
    \qquad
    \subfloat[][Numerical rank of different kernel functions]
    {\adjustbox{width=0.4\columnwidth,valign=B,raise=1.4\baselineskip}{%
      \renewcommand{\arraystretch}{1}%
        \begin{tabular}{c|cccc}\hline
            $N$ & $F_1(x,y)$ & $F_2(x,y)$ & $F_3(x,y)$ & $F_4(x,y)$ \\ \hline
            1000 & 22 & 20 & 22 & 22 \\
            5000 & 27 & 23 & 27 & 26 \\
            10000 & 29 & 25 & 29 & 27 \\
            15000 & 30 & 26 & 30 & 28 \\
            20000 & 31 & 26 & 31 & 29 \\
            25000 & 31 & 27 & 31 & 30 \\
            30000 & 32 & 27 & 32 & 30 \\
            40000 & 33 & 27 & 33 & 31 \\ \hline
        \end{tabular}}
    }
    \caption{Numerical rank $\bkt{tol = 10^{-12}}$ of vertex-sharing interaction in $1$D}
    \label{result_1d_verc}
    \end{figure}
}


\section{Rank growth of different interactions in 2D} \label{2d}  We discuss the rank of far-field, vertex-sharing and edge sharing interactions as shown in~\cref{2d_inter}. $I$, V and $E$ are far-field, vertex-sharing and edge sharing domains of $Y$ respectively. 

\begin{figure}[H]
    \centering   \subfloat[Far field, vertex-sharing and Edge sharing of $Y$]{\label{2d_inter}\resizebox{5cm}{!}{
        \begin{tikzpicture}[scale=1]
            \redsquare{0}{0};
            \redsquare{1}{0};
            \redsquare{2}{0};
            \redsquare{1}{-1};
            \node at (0.5,0.5) {$I$};
            \node at (1.5,-0.5) {V};
            \node at (1.5,0.5) {$E$};
            \node at (2.5,0.5) {$Y$};
        \end{tikzpicture}
    \label{2d_interactions}
        }}
        \qquad \qquad \qquad
    \centering   \subfloat[Far-field square boxes]{\label{2d_far_pic}\resizebox{5cm}{!}{
		\begin{tikzpicture}[scale=0.75]
			\draw (-1,-1) rectangle (1,1);
			\draw (3,-1) rectangle (5,1);
			\node at (0,0) {$X$};
			\node at (4,0) {$Y$};
			\draw [<->] (-1,-1.25) -- (1,-1.25);
			\draw [<->] (1,-1.25) -- (3,-1.25);
			\draw [<->] (3,-1.25) -- (5,-1.25);
			\draw [<->] (-1.25,-1) -- (-1.25,1);
			\node at (2,-1.5) {$r$};
			\node at (0,-1.5) {$r$};
			\node at (4,-1.5) {$r$};
			\node at (-1.5,0) {$r$};
		\end{tikzpicture}
		}}
    \caption{Different interactions in $2$D}
    \label{2d_far}
\end{figure}

    \subsection{\textbf{\textit{Rank growth of far-field domains}}} Let $X$ and $Y$ be square boxes of length $r$ and the distance between them is also $r$ as shown in ~\cref{2d_far_pic}. We choose a $\bkt{p+1}\times\bkt{p+1}$ tensor product grid on Chebyshev nodes in the square $Y$ to obtain the polynomial interpolation of $F(\xb,\yb)$ along $\yb$. Let $M_i = \displaystyle \sup_{\yb \in \mathcal{B}\bkt{{Y,\rho}}} \abs{F_a\bkt{\xb_i,\yb}}$ ,  $\rho \in (1,\alpha)^2$, for some $\alpha >1$. Let $\Tilde{K}$ be the approximation of the matrix \(K\). Then by~\cref{eq1} the absolute error along the $i^{th}$ row is given by (MATLAB notation)
        \begin{equation} \label{ap_2dfar}
            \Bigl\lvert \bkt{K \bkt{i,:}-\Tilde{K} \bkt{i,:}} \Bigl\lvert \leq 4 M_i V_2 \frac{\rho^{-p}}{\rho -1}, \quad 1 \leq i \leq N
        \end{equation}
         Let $M = \displaystyle \max_{1 \leq i\leq N} \{M_i\}$. Therefore,
         \begin{equation}
             \magn{K  - \Tilde{K}}_{max} \leq 4 M V_2 \dfrac{\rho^{-p}}{\rho -1} 
             \implies \dfrac{\magn{K  - \Tilde{K}}_{max}}{\magn{K}_{max}} \leq \dfrac{4M V_2}{\magn{K}_{max}} \dfrac{\rho^{-p}}{\rho -1} = \dfrac{c\rho^{-p}}{\rho -1}
         \end{equation}
    where $c = \dfrac{4M V_2}{\magn{K}_{max}}$. Now, choosing $p$ such that the above relative error to be less than $\delta$ (for some $\delta >0$), we obtain
        $p = \ceil{\frac{ \log \bkt{\frac{c}{\delta (\rho-1)}}}{\log(\rho)}} \implies \dfrac{\magn{K  - \Tilde{K}}_{max}}{\magn{K}_{max}} < \delta$.
Since, the rank of $\Tilde{K} = (p+1)^2$, i.e., the rank is bounded above by $\bkt{1+\ceil{\frac{\log\bkt{\frac{c}{\delta(\rho-1)}}}{\log(\rho)}}}^2$.
        Therefore, the rank of $\Tilde{K}$ scales
        $\mathcal{O}\bkt{ \log^2 \bkt{\frac{1}{\delta}}}$ with $ \dfrac{\magn{K  - \Tilde{K}}_{max}}{\magn{K}_{max}} < \delta$. The numerical rank plots of the far-field interaction of the eight functions as described in~\Cref{Preliminaries} are shown in~\cref{fig:res_2d_far_log} and tabulated in~\cref{tab:res_2d_far_log}
\cmt{
    \begin{figure}[H]
    \centering
    \subfloat[][N vs Rank plot]{\includegraphics[width=0.5\columnwidth]{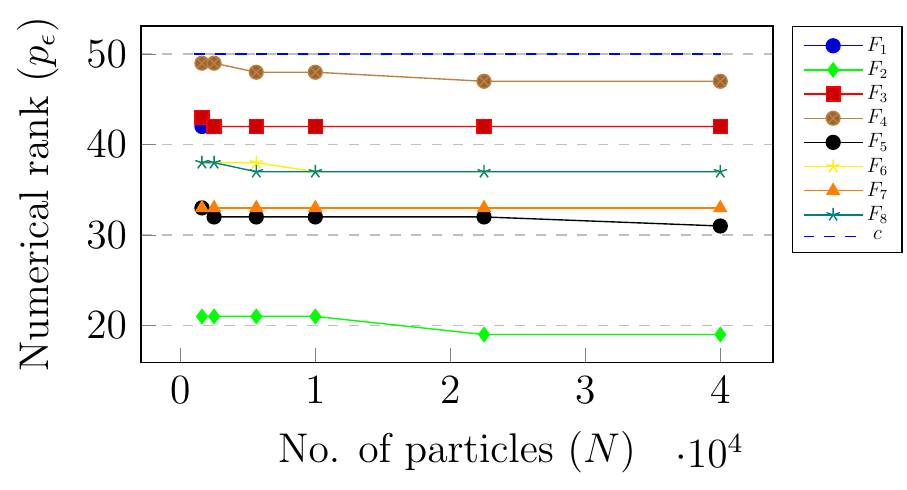}}\qquad
    \subfloat[][Numerical rank of different kernel functions]
    {\adjustbox{width=0.4\columnwidth,valign=B,raise=1.4\baselineskip}{%
      \renewcommand{\arraystretch}{1}%
        \begin{tabular}{c|cccc}\hline
            $N$ & $F_1(x,y)$ & $F_2(x,y)$ & $F_3(x,y)$ & $F_4(x,y)$ \\ \hline
            1600 & 42 & 21 & 48	& 49 \\
            2500 & 42 & 21 & 48	& 49 \\
            5625 & 42 & 21 & 47	& 48 \\
            10000 & 42 & 21	& 47 & 48 \\
            22500 & 42 & 19	& 47 & 47 \\
            40000  & 42	& 19 & 46 & 47 \\ \hline
        \end{tabular}}
    }
    \caption{Numerical rank $\bkt{tol = 10^{-12}}$ of far-field interaction in $2$D}
    \label{result_2d_far}
\end{figure}
}
    \subsection{\textbf{\textit{Rank growth of vertex-sharing domains}}}
        Consider two vertex-sharing square domains of length \(r\), \(X\) and \(Y\) as shown in~\cref{2d_vertex}. The box $Y$ is hierarchically sub-divided using an adaptive quad tree as shown in~\cref{2d_vertex_divide_level1,2d_vertex_divide}, which gives

        \begin{equation}
            Y = \overbrace{Y_2 \bigcup_{j=1}^{2^2-1} Y_{1,j}}^{\text{At level 1}} = \dots = \overbrace{Y_{\kappa + 1} \bigcup_{k=1}^{\kappa} \bigcup_{j=1}^{3} Y_{k,j}}^{\text{At level } \kappa}
        \end{equation}
        where $\kappa \sim \log_4(N)$ and $ Y_{\kappa +1 }$ having one particle. Let $\Tilde{K}$ be the approximation of the kernel matrix \(K\).
        Let $\Tilde{K}_{k,j}$ be approximation of the matrix ${K_{k,j}}$, then approximation of matrix $K$ is given by
        \begin{equation}
            \Tilde{K} = \sum_{k=1}^{\kappa} \sum_{j=1}^{3} \Tilde{K}_{k,j} + K_{\kappa+1}
        \end{equation}
        We choose a $\bkt{p_{k,j}+1}\times\bkt{p_{k,j}+1}$ tensor product grid on Chebyshev nodes in the square $Y_{k,j}$ to obtain the polynomial interpolation of $F(\xb,\yb)$ along $\yb$. Let $M^{(i)}_{k,j} = \displaystyle \sup_{\yb \in \mathcal{B} \bkt{{Y_{k,j},\rho}}} \abs{F_a\bkt{\xb_i,\yb}}$,  $\rho \in (1,\alpha)^2$, for some $\alpha > 1$. Then using~\cref{eq1} the absolute error along the $i^{th}$ row of the matrix $K_{k,j}$ is given by (MATLAB notation)
        \begin{equation} \label{ver_2}
            \Bigl\lvert \bkt{K_{k,j} \bkt{i,:}-\Tilde{K}_{k,j} \bkt{i,:} } \Bigl\lvert \leq 4 M^{(i)}_{k,j} V_2 \frac{\rho^{-p_{k,j}}}{\rho -1}, \quad 1 \leq i \leq N, \quad 1 \leq k \leq \kappa, \quad 1 \leq j \leq 3
        \end{equation}
        $M = \max \{  M^{(i)}_{k,j} \}$. Therefore,
        \begin{equation}
            \magn{K_{k,j} - \Tilde{K}_{k,j} }_{max} \leq 4 M V_2 \frac{\rho^{-p_{k,j}}}{\rho -1} \implies \dfrac{\magn{K_{k,j} - \Tilde{K}_{k,j} }_{max}}{\magn{K}_{max}} \leq \dfrac{4M V_2}{\magn{K}_{max}} \frac{\rho^{-p_{k,j}}}{\rho -1} = c \frac{\rho^{-p_{k,j}}}{\rho -1} 
        \end{equation}
        where $c = \dfrac{4M V_2}{\magn{K}_{max}}$. Now choosing $p_{k,j}$ such that the above error is less than $\delta_1$ (for some $\delta_1>0$), we obtain \\ 
        $p_{k,j} =  \ceil{\frac{ \log \bkt{\frac{c}{\delta_1 (\rho-1)}}}{\log(\rho)}} \implies  \dfrac{\magn{K_{k,j} - \Tilde{K}_{k,j} }_{max}}{\magn{K}_{max}} < \delta_1$ with rank of $\Tilde{K}_{k,j} $ is $ (1 + p_{k,j})^2$. \\
    Let $p_{l,m} = \max \{p_{k,j} : k=1,2,\hdots, \kappa \text{ and } j = 1,2,3 \}$, which corresponds to $\Tilde{K}_{l,m}$.
        So, at level $k$ the absolute error of matrix $K_k$ is
        \begin{align}
            \magn{K_k - \Tilde{K}_k}_{max} = \magn{\sum_{j=1}^{3}  \bkt{K_{k,j} -\Tilde{K}_{k,j} }}_{max} \leq \sum_{j=1}^{3} \magn{\bkt{K_{k,j} -\Tilde{K}_{k,j}}}_{max} < 3 \delta_1 \magn{K}_{max}
        \end{align}
    with rank of $\Tilde{K}_k$ is bounded above by $3 (1 + p_{l,m})^2$.
        Hence, we get 
        \begin{gather}
            \magn{K - \Tilde{K}}_{max} = \magn{\dsum_{k=1}^\kappa \bkt{K_k - \Tilde{K}_k} }_{max} \leq \dsum_{k=1}^\kappa \magn{K_k - \Tilde{K}_k}_{max} < 3 \kappa \delta_1 \magn{K}_{max}
        \end{gather}
        and the rank of $\Tilde{K} $ is bounded above by $ 1+3 \kappa (1 + p_{l,m})^2 = 1+3 \kappa \bkt{1 + \ceil{\frac{ \log \bkt{\frac{c}{\delta_1 (\rho-1)}}}{\log(\rho)}}}^2$. Note that, the rank of $K_{\kappa +1}$ is one.
        If we choose $\delta_1 = \dfrac{\delta}{3 \kappa}$, then $\dfrac{\magn{K - \Tilde{K}}_{max}}{\magn{K}_{max}} < \delta$ with rank of $\Tilde{K} $ bounded above by $ 1 + 3 \kappa \bkt{1+\ceil{\frac{ \log \bkt{\frac{3 \kappa c}{\delta (\rho-1)}}}{\log(\rho)}}}^2 = 1 + 3 \log_4(N) \bkt{1+\ceil{\frac{ \log \bkt{\frac{3 \log_4(N) c}{\delta (\rho-1)}}}{\log(\rho)}}}^2$.
        
        Therefore, the rank of $\Tilde{K}$ scales $\mathcal{O}\bkt{\log_4(N)\log^2 \bkt{\frac{ \log_4(N)}{\delta}}}$ with
        $\dfrac{\magn{K - \Tilde{K}}_{max}}{\magn{K}_{max}} < \delta$. The numerical rank plots of the vertex-sharing interaction of the eight functions as described in~\Cref{Preliminaries} are shown in~\cref{fig:res_2d_ver_log} and tabulated in~\cref{tab:res_2d_ver_log}.
         
 \cmt{
    \begin{figure}[H]
    \centering
    \subfloat[][N vs Rank plot]{\includegraphics[width=0.5\columnwidth]{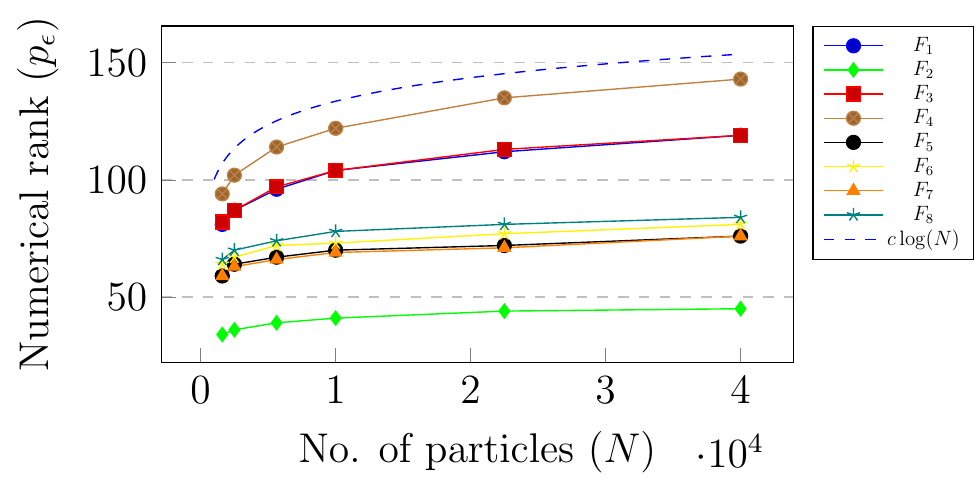}}\qquad
    \subfloat[][Numerical rank of different kernel functions]
    {\adjustbox{width=0.4\columnwidth,valign=B,raise=1.4\baselineskip}{%
      \renewcommand{\arraystretch}{1}%
        \begin{tabular}{c|cccc}\hline
            $N$ & $F_1(x,y)$ & $F_2(x,y)$ & $F_3(x,y)$ & $F_4(x,y)$ \\ \hline
            1600 & 81 & 34 & 87	& 94 \\
            2500 & 87 & 36 & 91 & 102 \\
            5625 & 96 & 39 & 101 & 114 \\
            10000 & 104	& 41 & 108 & 122 \\
            22500 & 112	& 44 & 118 & 135 \\
            40000  & 119 & 45 & 124 & 143 \\ \hline
        \end{tabular}}
    }
    \caption{Numerical rank $\bkt{tol = 10^{-12}}$ of vertex-sharing interaction in $2$D}
    \label{result_2d_ver}
    \end{figure}         
 }       
    
    \subsection{\textbf{\textit{Rank growth of edge sharing domains}}}

    Consider two edge sharing square boxes \(X\) and \(Y\) as shown in~\cref{2d_edge}. The box \(Y\) is hierarchically sub-divided using an adaptive quad tree as shown in~\cref{2d_edge_divide_l1,2d_edge_divide}, which gives
    \begin{equation}
        Y = \overbrace{Y_2 \bigcup_{j=1}^{2} Y_{1,j}}^{\text{At level 1}} = \dots= \overbrace{Y_{\kappa +1} \bigcup_{k=1}^{\kappa} \bigcup_{j=1}^{2^k} Y_{k,j}}^{\text{At level }\kappa}
    \end{equation}
    where $\kappa \sim \log_4(N)$ and $Y_{\kappa +1 }$ having $2^{\kappa} = \sqrt{N}$ particles. Let $\Tilde{K}_{k,j}$ be approximation of ${K_{k,j}}$, then approximation of matrix $K$ is given by
    \begin{equation}
        \Tilde{K} = \sum_{k=1}^{\kappa} \sum_{j=1}^{2^k} \Tilde{K}_{k,j} + K_{\kappa+1}
    \end{equation}
    We choose a $\bkt{p_{k,j}+1}$ tensor product grid on Chebyshev nodes in the square $Y_{k,j}$ to obtain the polynomial interpolation of $F(\xb,\yb)$ along $\yb$ by interpolating the function only along the direction of \emph{non-smoothness}. Let $M^{(i)}_{k,j} = \displaystyle \sup_{\yb \in \mathcal{B} \bkt{{Y_{k,j},\rho}}} \abs{F_a\bkt{\xb_i,\yb}}$,  $\rho \in (1,\alpha)^2$, for some $\alpha > 1$. Then by~\cref{eq1} the absolute error along the $i^{th}$ row of the matrix $K_{k,j}$ is given by
    \begin{equation} \label{app_ed2}
        \Bigl\lvert \bkt{K_{k,j} \bkt{i,:}-\Tilde{K}_{k,j} \bkt{i,:}} \Bigl\lvert \leq 4 M^{(i)}_{k,j}  V_2 \frac{\rho^{-p_{k,j}}}{\rho -1}, \quad 1 \leq i \leq N, \quad 1 \leq k \leq \kappa, \quad 1 \leq j \leq 2^k 
    \end{equation}
    Let $M = \max \{ M^{(i)}_{k,j} \}$. Therefore, 
    \begin{equation}
        \magn{K_{k,j} - \Tilde{K}_{k,j}}_{max} \leq 4 M  V_2 \frac{\rho^{-p_{k,j}}}{\rho -1} \implies \dfrac{\magn{K_{k,j} - \Tilde{K}_{k,j}}_{max}}{\magn{K}_{max}} \leq \dfrac{4 M  V_2}{\magn{K}_{max}} \frac{\rho^{-p_{k,j}}}{\rho -1} = c \frac{\rho^{-p_{k,j}}}{\rho -1}
    \end{equation}
    where $c = \dfrac{4 M  V_2}{\magn{K}_{max}} $. Now choosing $p_{k,j}$ such that the above error is less than $\delta_1$ (for some $\delta_1>0$), we obtain \\ $p_{k,j}  = \ceil{\frac{ \log \bkt{\frac{c}{\delta_1 (\rho-1)}}}{\log(\rho)}} \implies  \dfrac{\magn{K_{k,j} - \Tilde{K}_{k,j}}_{max}}{\magn{K}_{max}} < \delta_1$ with rank of $\Tilde{K}_{k,j} $ is $ (1 + p_{k,j})$. Let $p_{l,m} = \max \{p_{k,j} : k=1,2,\hdots, \kappa \text{ and } j = 1,2, \hdots, 2^k \}$, which corresponds to $\Tilde{K}_{l,m}$.
    Note that the rank of $K_{\kappa +1}$ is $\sqrt{N}$. So, the rank of $\Tilde{K}$ is bounded above by
    \begin{align} \label{eq2}
        \begin{split}
            \sqrt{N} + \sum_{k=1}^{\kappa} 2^k \bkt{1 + \ceil{\frac{ \log \bkt{\frac{c}{\delta_1 (\rho-1)}}}{\log(\rho)}}} = \sqrt{N} + \bkt{2^{\kappa+1}-2}\bkt{1 + \ceil{\frac{ \log \bkt{\frac{c}{\delta_1 (\rho-1)}}}{\log(\rho)}}} \\ = \sqrt{N} + \bkt{2\sqrt{N}-2}\bkt{1 + \ceil{\frac{ \log \bkt{\frac{c}{\delta_1 (\rho-1)}}}{\log(\rho)}}}
        \end{split}
    \end{align}
     Therefore, the rank of $\Tilde{K} \in \mathcal{O}\bkt{\sqrt{N}\log \bkt{\frac{1}{\delta_1}}}$
    and the absolute error is given by
    \begin{equation}
         \magn{K - \Tilde{K}}_{max} <  \sum_{k=1}^{\kappa} 2^k \delta_1 \magn{K}_{max} = 2\bkt{2^{\kappa}-1} \delta_1 \magn{K}_{max} =\bkt{2^{\kappa +1}-2} \delta_1 \magn{K}_{max} < 2 \sqrt{N} \delta_1 \magn{K}_{max}, 
    \end{equation}
    as $\kappa \sim \log_4(N)$. If we choose $\delta_1 = \frac{\delta}{2 \sqrt{N}}$ then
    $\dfrac{\magn{K - \Tilde{K}}_{max}}{\magn{K}_{max}} < \delta $ and the rank of $\Tilde{K} \in$ $\mathcal{O}\bkt{\sqrt{N}\bkt{\log\bkt{\frac{2\sqrt{N}}{\delta}}}}$
    
    Hence, the rank of $\Tilde{K}$ scales $\mathcal{O}\bkt{\sqrt{N}\log\bkt{\frac{N}{\delta}}}$ with
    $\dfrac{\magn{K - \Tilde{K}}_{max}}{\magn{K}_{max}} < \delta$. The numerical rank plots of the edge-sharing interaction of the eight functions as described in~\Cref{Preliminaries} are shown in~\cref{fig:res_2d_edge_log} and tabulated in~\cref{tab:res_2d_edge_log}.

\cmt{
    \begin{figure}[H]
    \centering
    \subfloat[][N vs Rank plot]{\includegraphics[width=0.5\columnwidth]{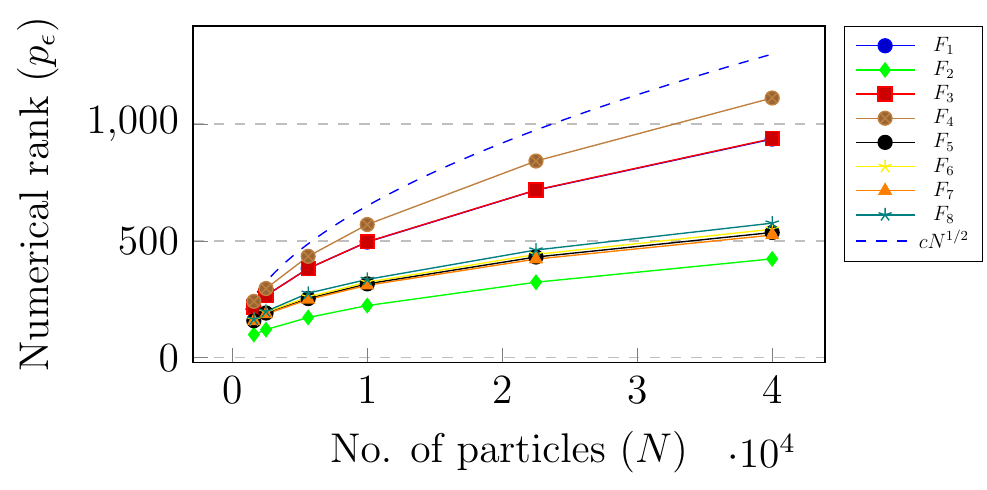}}\qquad
    \subfloat[][Numerical rank of different kernel functions]
    {\adjustbox{width=0.4\columnwidth,valign=B,raise=1.4\baselineskip}{%
      \renewcommand{\arraystretch}{1}%
        \begin{tabular}{c|cccc}\hline
            $N$ & $F_1(x,y)$ & $F_2(x,y)$ & $F_3(x,y)$ & $F_4(x,y)$ \\ \hline
            1600 & 216 & 99	& 220 & 241 \\
            2500 & 266 & 120 & 269 & 296 \\
            5625 & 382 & 172 & 388 & 434 \\
            10000 & 495	& 223 & 502 & 570 \\
            22500 & 717	& 323 & 727 & 842 \\
            40000  & 936 & 423 & 949 & 1112 \\ \hline
        \end{tabular}}
    }
    \caption{Numerical rank $\bkt{tol = 10^{-12}}$ of edge sharing interaction in $2$D}
    \label{result_2d_edge}
    \end{figure}
}

\section{Rank growth of different of interactions in 3D} \label{3d}
    We discuss the rank of far-field, vertex-sharing, edge sharing and face-sharing interactions as shown in~\cref{3d_inter}. $I$, V, $E$ and $F$ be the far-field, vertex-sharing, edge sharing and face-sharing domains of the cube $Y$ respectively.
    \begin{figure}[H]
    \centering  \subfloat[Different interactions of $Y$]{\label{3d_inter}\resizebox{5cm}{!}{
     	\begin{tikzpicture} 
        \cube{0}{0}
        \cube{-1}{0}
        \cube{-2}{0}
        \cube{1+\iso}{\iso}
        \cube{1+\iso}{1+\iso}
        \draw [ultra thick] (-2,1) -- (-2+\iso,1+\iso) -- (1+\iso,1+\iso);
        \draw [ultra thick] (1,0) -- (1+\iso,\iso);
        \draw [ultra thick] (2+\iso,\iso) -- (2+2*\iso,2*\iso) -- (2+2*\iso,2+2*\iso) -- (1+2*\iso,2+2*\iso) -- (1+\iso,2+\iso);
    
        \node at (0.7,0.7) {$Y$};
        \node at (-0.3,0.7) {$F$};
        \node at (-1.3,0.7) {$I$};
        \node at (1.7+\iso,0.7+\iso) {$E$};
        \node at (1.7+\iso,1.7+\iso) {V};
        \end{tikzpicture}
        }}
        \qquad \qquad \qquad
        \centering   \subfloat[Far-field cubes]{\label{3d_far}\resizebox{5cm}{!}{
                    	\begin{tikzpicture}[scale=0.5]
         				\draw (2,2,0)--(0,2,0)--(0,2,2)--(2,2,2)--(2,2,0)--(2,0,0)--(2,0,2)--(0,0,2)--(0,2,2);
         		     	\draw
         		     	(2,2,2)--(2,0,2);
         		     	\draw
         		     	(2,0,0)--(0,0,0)--(0,2,0);
         		     	\draw
         		     	(0,0,0)--(0,0,2);
         			    \node at (1,1,1) {$X$};
         		
                        \draw
                        (4,0,2) rectangle (6,2,2);
                        \draw
                        (4,0,0) rectangle (6,2,0);
                        \draw
                        (6,2,0)--(6,2,2);
                        \draw
                        (6,0,0)--(6,0,2);
                        \draw (4,0,0) -- (4,0,2);
                        \draw (4,2,0) -- (4,2,2);
                        \node at (5,1,1) {$Y$};
         			\end{tikzpicture}
        }}
        \caption{Different interactions in 3D}
    \end{figure}

    \subsection{\textbf{\textit{Rank growth of far-field domains}}} Let $X$ and $Y$ be cubes of size \(r\) separated by a distance \(r\) as shown in ~\cref{3d_far}. We choose a $\bkt{p+1}\times\bkt{p+1}\times\bkt{p+1}$ tensor product grid of Chebyshev nodes in the box $Y$ to obtain the polynomial interpolation of $F(\xb,\yb)$ along $\yb$. Let $M_i = \displaystyle \sup_{\yb \in \mathcal{B}\bkt{{Y,\rho}}} \abs{F_a\bkt{\xb_i,\yb}}$ ,  $\rho \in (1,\alpha)^3$, for some $\alpha >1$. Let $\Tilde{K}$ be the approximation of the matrix \(K\). Then by~\cref{eq1} the absolute error along the $i^{th}$ row is given by (MATLAB notation)
    \begin{equation} \label{ap_3dfar}
        \Bigl\lvert \bkt{K \bkt{i,:}-\Tilde{K} \bkt{i,:}} \Bigl\lvert \leq 4 M_i V_3 \frac{\rho^{-p}}{\rho -1}, \quad 1 \leq i \leq N
    \end{equation}
        Let $M = \displaystyle \max_{1 \leq i\leq N} \{M_i\}$. Therefore,
         \begin{equation}
             \magn{K  - \Tilde{K}}_{max} \leq 4 M V_3 \dfrac{\rho^{-p}}{\rho -1} 
             \implies \dfrac{\magn{K  - \Tilde{K}}_{max}}{\magn{K}_{max}} \leq \dfrac{4M V_3}{\magn{K}_{max}} \dfrac{\rho^{-p}}{\rho -1} = \dfrac{c\rho^{-p}}{\rho -1}
         \end{equation}
    where $c = \dfrac{4M V_3}{\magn{K}_{max}}$. Now, choosing $p$ such that the above relative error to be less than $\delta$ (for some $\delta >0$), we obtain
        $p = \ceil{\frac{ \log \bkt{\frac{c}{\delta (\rho-1)}}}{\log(\rho)}} \implies \dfrac{\magn{K  - \Tilde{K}}_{max}}{\magn{K}_{max}} < \delta$.
Since, the rank of $\Tilde{K} = (p+1)^3$, i.e., the rank is bounded above by $\bkt{1+\ceil{\frac{\log\bkt{\frac{c}{\delta(\rho-1)}}}{\log(\rho)}}}^3$.
        Therefore, the rank of $\Tilde{K}$ scales
        $\mathcal{O}\bkt{ \log^3 \bkt{\frac{1}{\delta}}}$ with $ \dfrac{\magn{K  - \Tilde{K}}_{max}}{\magn{K}_{max}} < \delta$. 
The numerical rank plots of the far-field interaction of the eight functions as described in~\Cref{Preliminaries} are shown in~\cref{fig:res_3d_far_log} and tabulated in~\cref{tab:res_3d_far_log}.

\cmt{
\begin{figure}[H]
    \centering
    \subfloat[][N vs Rank plot]{\includegraphics[width=0.5\columnwidth]{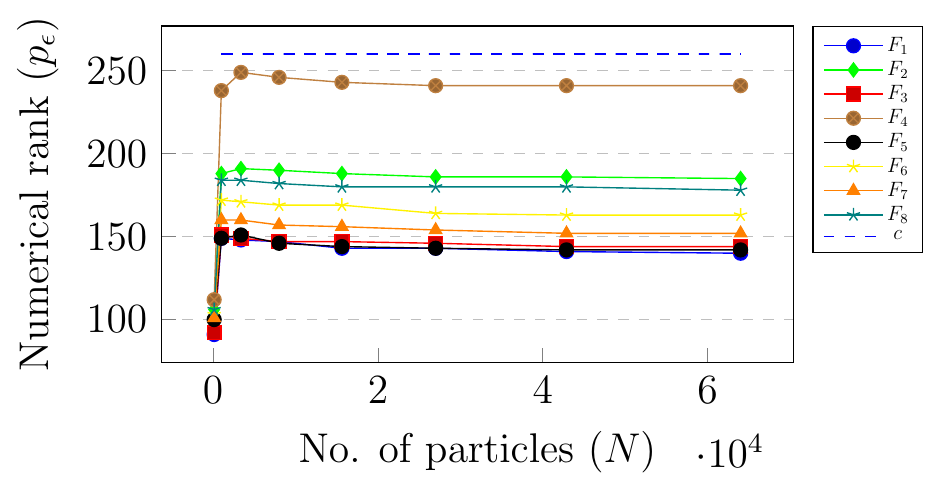}}\qquad
    \subfloat[][Numerical rank of different kernel functions]
    {\adjustbox{width=0.4\columnwidth,valign=B,raise=1.4\baselineskip}{%
      \renewcommand{\arraystretch}{1}%
        \begin{tabular}{c|cccc}\hline
            $N$ & $F_1(x,y)$ & $F_2(x,y)$ & $F_3(x,y)$ & $F_4(x,y)$ \\ \hline
            1000 & 149 & 188 & 163 & 238 \\
            3375 & 148 & 191 & 160 & 249 \\
            8000 & 147 & 190 & 158 & 246 \\
            15625 & 143	& 188 & 156 & 243 \\
            27000 & 143	& 186 & 156 & 241 \\
            42875  & 141 & 186 & 154 & 241 \\
            64000  & 140 & 185 & 154 & 241 \\
            \hline
        \end{tabular}}
    }
    \caption{Numerical rank $\bkt{tol = 10^{-12}}$ of far-field interaction in $3$D}
    \label{result_3d_far}
\end{figure}
}

    \subsection{\textbf{\textit{Rank growth of vertex-sharing domains}}}
        Consider two vertex-sharing cubes \(X\) (big red cube) and \(Y\) (big black cube) as shown in~\cref{3d_ver}. The black cube $Y$ is hierarchically sub-divided using an adaptive oct tree as shown in~\cref{3d_ver}. \href{https://sites.google.com/view/dom3d/vertex-sharing-domains}{\textbf{The link here provides a better $3$D view}}. So, we can write $Y$ as

        \begin{equation}
            Y = \overbrace{Y_2 \bigcup_{j=1}^{2^3 -1}Y_{1,j}}^{\text{At level 1}} =  \dots =  \overbrace{Y_{\kappa +1} \bigcup_{k=1}^{\kappa} \bigcup_{j=1}^{7} Y_{k,j}}^{\text{At level } \kappa}
        \end{equation}
        where $\kappa \sim \log_8(N)$ and $Y_{\kappa +1}$ having one particle. Let $\Tilde{K}$ be the approximation of the kernel matrix \(K\).
        Let $\Tilde{K}_{k,j}$ be approximation of ${K_{k,j}}$, then approximation of the kernel matrix $K$ is given by
        \begin{equation}
            \Tilde{K} = \sum_{k=1}^{\kappa} \sum_{j=1}^{7} \Tilde{K}_{k,j} + K_{\kappa +1}
        \end{equation}
        We choose a $\bkt{p_{k,j}+1}\times\bkt{p_{k,j}+1}\times\bkt{p_{k,j}+1}$ tensor product grid on Chebyshev nodes in the cube $Y_{k,j}$ to obtain the polynomial interpolation of $F(\xb,\yb)$ along $\yb$. Let $M^{(i)}_{k,j} = \displaystyle \sup_{\yb \in \mathcal{B} \bkt{{Y_{k,j},\rho}}} \abs{F_a\bkt{\xb_i,\yb}}$,  $\rho \in (1,\alpha)^3$, for some $\alpha > 1$. Then using~\cref{eq1} the absolute error along the $i^{th}$ row of the matrix $K_{k,j}$ is given by (MATLAB notation)
        \begin{equation} 
            \Bigl\lvert \bkt{K_{k,j} \bkt{i,:}-\Tilde{K}_{k,j} \bkt{i,:} } \Bigl\lvert \leq 4 M^{(i)}_{k,j} V_3 \frac{\rho^{-p_{k,j}}}{\rho -1}, \quad 1 \leq i \leq N, \quad 1 \leq k \leq \kappa, \quad 1 \leq j \leq 7
        \end{equation}
        $M = \max \{  M^{(i)}_{k,j} \}$. Therefore,
        \begin{equation}
            \magn{K_{k,j} - \Tilde{K}_{k,j} }_{max} \leq 4 M V_3 \frac{\rho^{-p_{k,j}}}{\rho -1} \implies \dfrac{\magn{K_{k,j} - \Tilde{K}_{k,j} }_{max}}{\magn{K}_{max}} \leq \dfrac{4M V_3}{\magn{K}_{max}} \frac{\rho^{-p_{k,j}}}{\rho -1} = c \frac{\rho^{-p_{k,j}}}{\rho -1} 
        \end{equation}
        where $c = \dfrac{4M V_3}{\magn{K}_{max}}$. Now choosing $p_{k,j}$ such that the above error is less than $\delta_1$ (for some $\delta_1>0$), we obtain \\ 
        $p_{k,j} =  \ceil{\frac{ \log \bkt{\frac{c}{\delta_1 (\rho-1)}}}{\log(\rho)}} \implies  \dfrac{\magn{K_{k,j} - \Tilde{K}_{k,j} }_{max}}{\magn{K}_{max}} < \delta_1$ with rank of $\Tilde{K}_{k,j} $ is $ (1 + p_{k,j})^3$. \\
    Let $p_{l,m} = \max \{p_{k,j} : k=1,2,\hdots, \kappa \text{ and } j = 1,2,\hdots,7 \}$, which corresponds to $\Tilde{K}_{l,m}$.
        So, at level $k$ the absolute error of matrix $K_k$ is
        \begin{align}
            \magn{K_k - \Tilde{K}_k}_{max} = \magn{\sum_{j=1}^{7}  \bkt{K_{k,j} -\Tilde{K}_{k,j} }}_{max} \leq \sum_{j=1}^{7} \magn{\bkt{K_{k,j} -\Tilde{K}_{k,j}}}_{max} < 7 \delta_1 \magn{K}_{max}
        \end{align}
    with rank of $\Tilde{K}_k$ is bounded above by $7 (1 + p_{l,m})^3$.
        Hence, we get 
        \begin{gather}
            \magn{K - \Tilde{K}}_{max} = \magn{\dsum_{k=1}^\kappa \bkt{K_k - \Tilde{K}_k} }_{max} \leq \dsum_{k=1}^\kappa \magn{K_k - \Tilde{K}_k}_{max} < 7 \kappa \delta_1 \magn{K}_{max}
        \end{gather}
        and the rank of $\Tilde{K} $ is bounded above by $ 1+7 \kappa (1 + p_{l,m})^3 = 1+7 \kappa \bkt{1 + \ceil{\frac{ \log \bkt{\frac{c}{\delta_1 (\rho-1)}}}{\log(\rho)}}}^3$. Note that, the rank of $K_{\kappa +1}$ is one.
        If we choose $\delta_1 = \dfrac{\delta}{7 \kappa}$, then $\dfrac{\magn{K - \Tilde{K}}_{max}}{\magn{K}_{max}} < \delta$ with rank of $\Tilde{K} $ bounded above by $ 1 + 7 \kappa \bkt{1+\ceil{\frac{ \log \bkt{\frac{7 \kappa c}{\delta (\rho-1)}}}{\log(\rho)}}}^3 = 1 + 7 \log_8(N) \bkt{1+\ceil{\frac{ \log \bkt{\frac{7 \log_8(N) c}{\delta (\rho-1)}}}{\log(\rho)}}}^3$.
        
        Therefore, the rank of $\Tilde{K}$ scales $\mathcal{O}\bkt{\log_8(N)\log^3 \bkt{\frac{ \log_8(N)}{\delta}}}$ with
        $\dfrac{\magn{K - \Tilde{K}}_{max}}{\magn{K}_{max}} < \delta$. The numerical rank plots of the vertex-sharing interaction of the eight functions as described in~\Cref{Preliminaries} are shown in~\cref{fig:res_3d_ver_log} and tabulated in~\cref{tab:res_3d_ver_log}.
\cmt{    
\begin{figure}[H]
    \centering
    \subfloat[][N vs Rank plot]{\includegraphics[width=0.5\columnwidth]{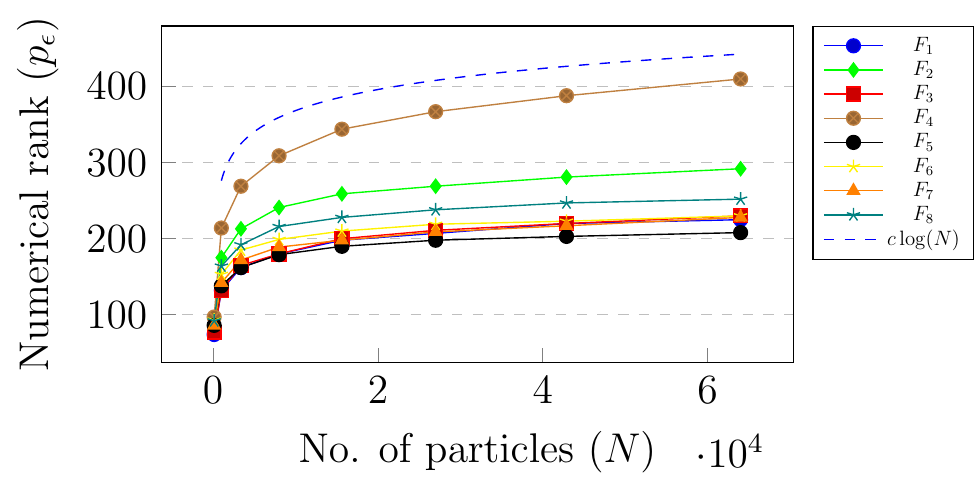}}\qquad
    \subfloat[][Numerical rank of different kernel functions]
    {\adjustbox{width=0.4\columnwidth,valign=B,raise=1.4\baselineskip}{%
      \renewcommand{\arraystretch}{1}%
        \begin{tabular}{c|cccc}\hline
            $N$ & $F_1(x,y)$ & $F_2(x,y)$ & $F_3(x,y)$ & $F_4(x,y)$ \\ \hline
            1000 & 132 & 175 & 150 & 214 \\
            3375 & 162 & 213 & 180 & 269 \\
            8000 & 180 & 241 & 205 & 309 \\
            15625 & 198	& 259 & 219 & 344 \\
            27000 & 207	& 269 & 231 & 367 \\
            42875  & 220 & 281 & 246 & 388 \\
            64000  & 225 & 292 & 253 & 410 \\
            \hline
        \end{tabular}}
    }
    \caption{Numerical rank $\bkt{tol = 10^{-12}}$ of vertex-sharing interaction in $3$D}
    \label{result_3d_ver}
\end{figure} }       
    \subsection{\textbf{\textit{Rank growth of edge sharing domains}}}
        Consider edge sharing cubes \(X\) (big red cube) and \(Y\) (big black cube) as shown in~\cref{3d_edg}. The black cube $Y$ is hierarchically sub-divided using an adaptive oct tree as shown in~\cref{3d_edg}. \href{https://sites.google.com/view/dom3d/edge-sharing-domains}{\textbf{The link here provides a better $3$D view}}. So, we can write $Y$ as
        
        \begin{equation}
            Y =  \overbrace{ Y_2 \bigcup_{j=1}^{3 \cdot 2^1} Y_{1,j}}^{\text{At level 1}} =\dots = \overbrace{Y_{\kappa+1} \bigcup_{k=1}^{\kappa} \bigcup_{j=1}^{3 \cdot 2^k} Y_{k,j}}^{\text{At level }\kappa}
        \end{equation}
    where $\kappa \sim \log_8(N)$ and $Y_{\kappa +1}$ having $2^{\kappa} = N^{1/3}$ particles. Let $\Tilde{K}_{k,j}$ be approximation of ${K_{k,j}}$, then approximation of matrix $K$ is given by
    \begin{equation}
       \Tilde{K} = \dsum_{k=1}^{\kappa} \dsum_{j=1}^{3 \cdot 2^k} \Tilde{K}_{k,j} + K_{\kappa +1} 
    \end{equation}
   We choose a $\bkt{p_{k,j}+1}\times\bkt{p_{k,j}+1}$ tensor product grid on Chebyshev nodes in the cube $Y_{k,j}$ to obtain the polynomial interpolation of $F(\xb,\yb)$ along $\yb$ by interpolating the function only along the direction of \emph{non-smoothness}. Let $M^{(i)}_{k,j} = \displaystyle \sup_{\yb \in \mathcal{B} \bkt{{Y_{k,j},\rho}}} \abs{F_a\bkt{\xb_i,\yb}}$,  $\rho \in (1,\alpha)^3$, for some $\alpha > 1$. Then by~\cref{eq1} the absolute error along the $i^{th}$ row of the matrix $K_{k,j}$ is given by
    \begin{equation} 
        \Bigl\lvert \bkt{K_{k,j} \bkt{i,:}-\Tilde{K}_{k,j} \bkt{i,:}} \Bigl\lvert \leq 4 M^{(i)}_{k,j}  V_3 \frac{\rho^{-p_{k,j}}}{\rho -1}, \quad 1 \leq i \leq N, \quad 1 \leq k \leq \kappa, \quad 1 \leq j \leq 3 \cdot 2^k 
    \end{equation}
    Let $M = \max \{ M^{(i)}_{k,j} \}$. Therefore, 
        \begin{equation}
            \magn{K_{k,j} - \Tilde{K}_{k,j} }_{max} \leq 4 M V_3 \frac{\rho^{-p_{k,j}}}{\rho -1} \implies \dfrac{\magn{K_{k,j} - \Tilde{K}_{k,j} }_{max}}{\magn{K}_{max}} \leq \dfrac{4M V_3}{\magn{K}_{max}} \frac{\rho^{-p_{k,j}}}{\rho -1} = c \frac{\rho^{-p_{k,j}}}{\rho -1} 
        \end{equation}
        where $c = \dfrac{4M V_3}{\magn{K}_{max}}$. Now choosing $p_{k,j}$ such that the above error is less than $\delta_1$ (for some $\delta_1>0$), we obtain \\ $p_{k,j}  = \ceil{\frac{ \log \bkt{\frac{c}{\delta_1 (\rho-1)}}}{\log(\rho)}} \implies \dfrac{\magn{K_{k,j} - \Tilde{K}_{k,j} }_{max}}{\magn{K}_{max}} < \delta_1$ with rank of $\Tilde{K}_{k,j} $ is $ (1 + p_{k,j})^2$. Let $p_{l,m} = \max \{p_{k,j} : k=1,2,\hdots, \kappa \text{ and } j = 1,2, \hdots, 3 \cdot 2^k \}$, which corresponds to $\Tilde{K}_{l,m}$.
    Note that the rank of $K_{\kappa +1}$ is $N^{\frac{1}{3}}$. So, the rank of $\Tilde{K}$ is bounded above by
    \begin{align} \label{eq3}
        \begin{split}
            N^{\frac{1}{3}} + \sum_{k=1}^{\kappa} 3 \cdot 2^k \bkt{1 + \ceil{\frac{ \log \bkt{\frac{c}{\delta_1 (\rho-1)}}}{\log(\rho)}}}^2  = N^{\frac{1}{3}} +  3 \bkt{2^{\kappa+1}-2}\bkt{1 + \ceil{\frac{ \log \bkt{\frac{c}{\delta_1 (\rho-1)}}}{\log(\rho)}}}^2 \\ = N^{\frac{1}{3}} + 3 \bkt{2 N^{\frac{1}{3}}-2}\bkt{1 + \ceil{\frac{ \log \bkt{\frac{c}{\delta_1 (\rho-1)}}}{\log(\rho)}}}^2 
        \end{split}
    \end{align}
    Therefore, the rank $\Tilde{K} \in \mathcal{O}\bkt{N^{\frac{1}{3}}\bkt{\log^2 \bkt{\frac{1}{\delta_1}}}}$ and the absolute error is given by
    \begin{equation}
         \magn{K- \Tilde{K}}_{max} <  \sum_{k=1}^{\kappa} 3 \cdot 2^k \delta_1 \magn{K}_{max} = 3 \bkt{2^{\kappa +1}-2} \delta_1 \magn{K}_{max} < 6 N^{\frac{1}{3}} \delta_1 \magn{K}_{max}  ,\text{ as }  \kappa \sim \log_8(N)
    \end{equation}
    If we choose $\delta_1 = \frac{\delta}{6 N^{\frac{1}{3}}}$ then
        $ \dfrac{\magn{K- \Tilde{K}}_{max}}{\magn{K}_{max}} < \delta $ and the rank of $\Tilde{K} \in$ $\mathcal{O}\bkt{N^{1/3} \bkt{\log^2 \bkt{\frac{6N^{1/3}}{\delta}}}}$
        
    Hence, the rank of $\Tilde{K}$ scales $\mathcal{O}\bkt{N^{1/3} \log^2\bkt{ \frac{N}{\delta}}}$ with 
    $\dfrac{\magn{K- \Tilde{K}}_{max}}{\magn{K}_{max}} < \delta$. The numerical rank plots of the edge sharing interaction of the eight functions as described in~\Cref{Preliminaries} are shown in~\cref{fig:res_3d_edge_log} and tabulated in~\cref{tab:res_3d_edge_log}.

\cmt{
\begin{figure}[H]
    \centering
    \subfloat[][N vs Rank plot]{\includegraphics[width=0.5\columnwidth]{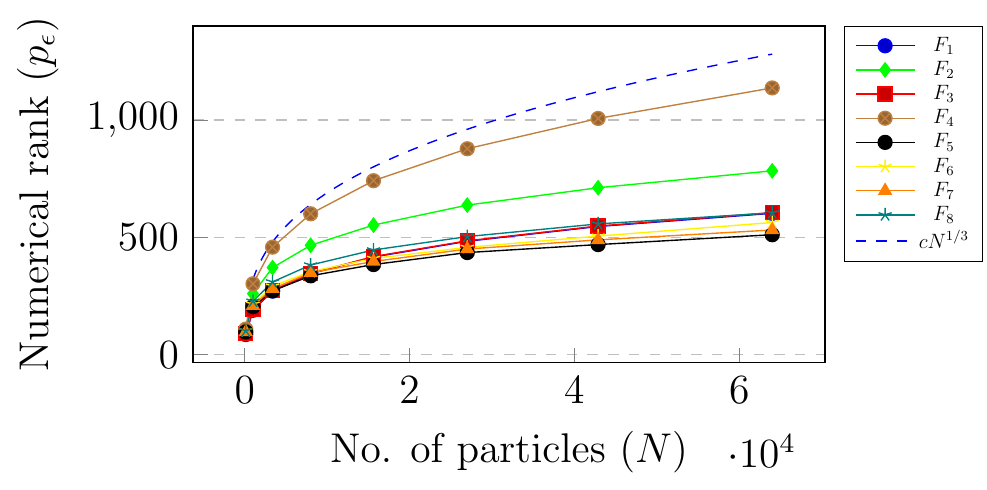}}\qquad
    \subfloat[][Numerical rank of different kernel functions]
    {\adjustbox{width=0.4\columnwidth,valign=B,raise=1.4\baselineskip}{%
      \renewcommand{\arraystretch}{1}%
        \begin{tabular}{c|cccc}\hline
            $N$ & $F_1(x,y)$ & $F_2(x,y)$ & $F_3(x,y)$ & $F_4(x,y)$ \\ \hline
            1000 & 189 & 260 & 202 & 302 \\
            3375 & 271 & 371 & 289 & 458 \\
            8000 & 345 & 466 & 366 & 600 \\
            15625 & 416	& 552 & 442 & 741 \\
            27000 & 483	& 637 & 512 & 877 \\
            42875  & 546 & 711 & 580 & 1006 \\
            64000  & 602 & 783 & 649 & 1136 \\
            \hline
        \end{tabular}}
    }
    \caption{Numerical rank $\bkt{tol = 10^{-12}}$ of edge-sharing interaction in $3$D}
    \label{result_3d_edge}
\end{figure}  }      
    \subsection{\textbf{\textit{Rank growth of face-sharing domains}}}
        Consider two face-sharing cubes \(X\) (big red cube) and \(Y\) (big black cube) as shown in~\cref{3d_fac}. The black cube $Y$ is hierarchically subdivided using an adaptive oct tree as shown~\cref{3d_fac}. \href{https://sites.google.com/view/dom3d/face-sharing-domains}{\textbf{The link here provides a better $3$D view}}. So, we can write $Y$ as




    
        
        \begin{equation}
            Y =  \overbrace{Y_2  \bigcup_{j=1}^{4} Y_{1,j}}^{\text{At level 1}} = \dots = \overbrace{Y_{\kappa +1} \bigcup_{k=1}^{\kappa} \bigcup_{j=1}^{4^k} Y_{k,j}}^{\text{At level } \kappa}
        \end{equation}
       where $\kappa \sim \log_8(N)$ and $Y_{\kappa +1}$ having $4^{\kappa} = N^{2/3}$ particles. Let $\Tilde{K}_{k,j}$ be approximation of ${K_{k,j}}$, then approximation of matrix $K$ is
       \begin{equation}
           \Tilde{K} = \dsum_{k=1}^{\kappa} \dsum_{j=1}^{4^k} \Tilde{K}_{k,j} + K_{\kappa +1}
       \end{equation}
         We choose a $\bkt{p_{k,j}+1}$ tensor product grid on Chebyshev nodes in the cube $Y_{k,j}$ to obtain the polynomial interpolation of $F(\xb,\yb)$ along $\yb$ by interpolating the function only along the direction of \emph{non-smoothness}. Let $M^{(i)}_{k,j} = \displaystyle \sup_{\yb \in \mathcal{B} \bkt{{Y_{k,j},\rho}}} \abs{F_a\bkt{\xb_i,\yb}}$,  $\rho \in (1,\alpha)^3$, for some $\alpha > 1$. Then by~\cref{eq1} the absolute error along the $i^{th}$ row of the matrix $K_{k,j}$ is given by
    \begin{equation} 
        \Bigl\lvert \bkt{K_{k,j} \bkt{i,:}-\Tilde{K}_{k,j} \bkt{i,:}} \Bigl\lvert \leq 4 M^{(i)}_{k,j}  V_3 \frac{\rho^{-p_{k,j}}}{\rho -1}, \quad 1 \leq i \leq N, \quad 1 \leq k \leq \kappa, \quad 1 \leq j \leq 4^k 
    \end{equation}
    Let $M = \max \{ M^{(i)}_{k,j} \}$. Therefore, 
        \begin{equation}
            \magn{K_{k,j} - \Tilde{K}_{k,j} }_{max} \leq 4 M V_3 \frac{\rho^{-p_{k,j}}}{\rho -1} \implies \dfrac{\magn{K_{k,j} - \Tilde{K}_{k,j} }_{max}}{\magn{K}_{max}} \leq \dfrac{4M V_3}{\magn{K}_{max}} \frac{\rho^{-p_{k,j}}}{\rho -1} = c \frac{\rho^{-p_{k,j}}}{\rho -1} 
        \end{equation}
        where $c = \dfrac{4M V_3}{\magn{K}_{max}}$. Now choosing $p_{k,j}$ such that the above error is less than $\delta_1$ (for some $\delta_1>0$), we obtain \\ $p_{k,j}  = \ceil{\frac{ \log \bkt{\frac{c}{\delta_1 (\rho-1)}}}{\log(\rho)}} \implies \dfrac{\magn{K_{k,j} - \Tilde{K}_{k,j} }_{max}}{\magn{K}_{max}} < \delta_1$ with rank of $\Tilde{K}_{k,j} $ is $ (1 + p_{k,j})$. Let $p_{l,m} = \max \{p_{k,j} : k=1,2,\hdots, \kappa \text{ and } j = 1,2, \hdots, 4^k \}$, which corresponds to $\Tilde{K}_{l,m}$.
    Note that the rank of $K_{\kappa +1}$ is $N^{\frac{2}{3}}$. So, the rank of $\Tilde{K}$ is bounded above by
    \begin{align} \label{eq4}
        \begin{split}
            N^{\frac{2}{3}} + \sum_{k=1}^{\kappa} 4^k \bkt{1 + \ceil{\frac{ \log \bkt{\frac{c}{\delta_1 (\rho-1)}}}{\log(\rho)}}} = N^{\frac{2}{3}} +  \frac{4}{3}\bkt{4^\kappa-1} \bkt{1 + \ceil{\frac{ \log \bkt{\frac{c}{\delta_1 (\rho-1)}}}{\log(\rho)}}} \\ = N^{\frac{2}{3}} + \frac{4}{3}\bkt{N^\frac{2}{3}-1} \bkt{1 + \ceil{\frac{ \log \bkt{\frac{c}{\delta_1 (\rho-1)}}}{\log(\rho)}}} 
        \end{split}
    \end{align}
  Therefore, the rank of $\Tilde{K} \in \mathcal{O}\bkt{N^{\frac{2}{3}}\log\bkt{\frac{1}{\delta_1}}}$  and the error in \mvp
    \begin{equation}
         \magn{K - \Tilde{K}}_{max} <  \sum_{k=1}^{\kappa} 4^k \delta_1 \magn{K}_{max} = \frac{4}{3}\bkt{4^{\kappa}-1} \delta_1 \magn{K}_{max}=\frac{4}{3}\bkt{N^\frac{2}{3}-1} \delta_1 \magn{K}_{max} < \dfrac{4}{3} N^{\frac{2}{3}} \delta_1  \magn{K}_{max},
    \end{equation}
    as $\kappa \sim \log_8(N)$.
    If we choose $\delta_1 =\frac{3 \delta}{4 N^{\frac{2}{3}}}$ then
        $ \dfrac{\magn{K - \Tilde{K}}_{max}}{\magn{K}_{max}} < \delta $ and the rank of $\Tilde{K} \in$ $\mathcal{O}\bkt{N^{2/3} \bkt{\log\bkt{\frac{4N^{2/3}}{3 \delta}}}}$
        
    Hence, the rank of $\Tilde{K}$ scales $\mathcal{O}\bkt{N^{2/3}\log\bkt{\frac{N}{\delta}}}$ with 
    $\dfrac{\magn{K - \Tilde{K}}_{max}}{\magn{K}_{max}} < \delta$. The numerical rank plots of the face-sharing interaction of the eight functions as described in~\Cref{Preliminaries} are shown in~\cref{fig:res_3d_face_log,tab:res_3d_face_log}.
\cmt{
\begin{figure}[H]
    \centering
    \subfloat[][N vs Rank plot]{
        \includegraphics[width=0.4\columnwidth]{Images/3d_face.pdf}
    }\qquad
    \subfloat[][Numerical rank of different kernel functions]
    {\adjustbox{width=0.4\columnwidth,valign=B,raise=1.4\baselineskip}{%
      \renewcommand{\arraystretch}{1}%
        \begin{tabular}{c|cccc}\hline
            $N$ & $F_1(x,y)$ & $F_2(x,y)$ & $F_3(x,y)$ & $F_4(x,y)$ \\ \hline
            1000 & 312 & 422 & 315 & 465 \\
            3375 & 610 & 815 & 614 & 964 \\
            8000 & 1003 & 1314 & 1012 & 1623 \\
            15625 & 1491 & 1931 & 1503 & 2443 \\
            27000 & 2077 & 2641 & 2092 & 3419 \\
            42875  & 2764 & 3465 & 2781 & 4556 \\
            64000  & 3547 & 4373 & 3564 & 5751 \\
            \hline
        \end{tabular}}
    }
    \caption{Numerical rank $\bkt{tol = 10^{-12}}$ of face-sharing interaction in $3$D}
    \label{result_3d_face}
\end{figure}
}
\begin{table}[H]
    \centering \setlength\tabcolsep{0.1pt}
    \subfloat[$4$D far-field]{        \resizebox{0.4\textwidth}{!}{%
        \begin{tabular}{p{1cm}|cccccccc}\hline
            $N$ & $F_1(x,y)$ & $F_2(x,y)$ & $F_3(x,y)$ & $F_4(x,y)$ & $F_5(x,y)$ & $F_6(x,y)$ & $F_7(x,y)$ & $F_8(x,y)$\\ \hline
            1296  & 641 & 599 & 656  & 630  & 529 & 598 & 533 & 632 \\
            2401  & 791 & 735 & 814  & 791  & 605 & 709 & 628 & 772 \\
            4096  & 902 & 808 & 931  & 928  & 636 & 750 & 661 & 855 \\
            10000 & 990 & 836 & 1008 & 1034 & 650 & 776 & 685 & 894 \\
            20736 & 994 & 842 & 1012 & 1037 & 645 & 775 & 680 & 898 \\
            38416 & 985 & 836 & 1012 & 1035 & 635 & 765 & 670 & 899 \\
            50625 & 982 & 831 & 1009 & 1031 & 633 & 759 & 666 & 897\\ \hline
        \end{tabular}}} \label{tab:res_4d_far_log} \quad%
    \subfloat[$3$D far-field]{        \resizebox{0.4\textwidth}{!}{%
        \begin{tabular}{p{1cm}|cccccccc}\hline
            $N$ & $F_1(x,y)$ & $F_2(x,y)$ & $F_3(x,y)$ & $F_4(x,y)$ & $F_5(x,y)$ & $F_6(x,y)$ & $F_7(x,y)$ & $F_8(x,y)$\\ \hline
            125   & 91  & 106 & 92  & 112 & 100 & 103 & 101 & 106 \\
            1000  & 149 & 188 & 151 & 238 & 149 & 172 & 160 & 184 \\
            3375  & 148 & 191 & 149 & 249 & 151 & 171 & 160 & 184 \\
            8000  & 147 & 190 & 147 & 246 & 146 & 169 & 157 & 182 \\
            15625 & 143 & 188 & 147 & 243 & 144 & 169 & 156 & 180 \\
            27000 & 143 & 186 & 146 & 241 & 143 & 164 & 154 & 180 \\
            42875 & 141 & 186 & 144 & 241 & 142 & 163 & 152 & 180 \\
            64000 & 140 & 185 & 144 & 241 & 142 & 163 & 152 & 178 \\ \hline
        \end{tabular}}} \label{tab:res_3d_far_log} \quad
    \subfloat[$2$D far-field]{        \resizebox{0.4\textwidth}{!}{%
        \begin{tabular}{p{1cm}|cccccccc}\hline
            $N$ & $F_1(x,y)$ & $F_2(x,y)$ & $F_3(x,y)$ & $F_4(x,y)$ & $F_5(x,y)$ & $F_6(x,y)$ & $F_7(x,y)$ & $F_8(x,y)$\\ \hline
            1600  & 42 & 21 & 43 & 49 & 33 & 38 & 33 & 38 \\
            2500  & 42 & 21 & 42 & 49 & 32 & 38 & 33 & 38 \\
            5625  & 42 & 21 & 42 & 48 & 32 & 38 & 33 & 37 \\
            10000 & 42 & 21 & 42 & 48 & 32 & 37 & 33 & 37 \\
            22500 & 42 & 19 & 42 & 47 & 32 & 37 & 33 & 37 \\
            40000 & 42 & 19 & 42 & 47 & 31 & 37 & 33 & 37 \\ \hline
        \end{tabular}}} \label{tab:res_2d_far_log} \quad
    \subfloat[$1$D far-field]{        \resizebox{0.4\textwidth}{!}{%
        \begin{tabular}{p{1cm}|cccccccc}\hline
                $N$ & $F_1(x,y)$ & $F_2(x,y)$ & $F_3(x,y)$ & $F_4(x,y)$ & $F_5(x,y)$ & $F_6(x,y)$ & $F_7(x,y)$ & $F_8(x,y)$\\ \hline
                1000  & 7 & 7 & 7 & 8 & 2 & 2 & 6 & 1                         \\
                5000  & 7 & 7 & 7 & 8 & 2 & 2 & 6 & 1                         \\
                10000 & 7 & 7 & 7 & 8 & 2 & 2 & 6 & 1                         \\
                15000 & 7 & 7 & 7 & 8 & 2 & 2 & 6 & 1                         \\
                20000 & 7 & 7 & 7 & 8 & 2 & 2 & 6 & 1                         \\
                25000 & 7 & 7 & 7 & 8 & 2 & 2 & 6 & 1                         \\
                30000 & 7 & 7 & 7 & 8 & 2 & 2 & 6 & 1                         \\
                40000 & 7 & 7 & 7 & 8 & 2 & 2 & 6 & 1\\ \hline
        \end{tabular}}} \label{tab:res_1d_far_log} \quad
        \caption{Numerical rank $(p_{\epsilon})$ with $\bkt{\epsilon = 10^{-12}}$ of different kernel functions of far-field interaction in $4$D, $3$D, $2$D, $1$D}
\end{table}

\begin{table}[H]
       \centering         \setlength\tabcolsep{0.1pt}
        \subfloat[$4$D vertex-sharing]{\resizebox{0.4\textwidth}{!}{%
        \begin{tabular}{p{1cm}|cccccccc}\hline
            $N$ & $F_1(x,y)$ & $F_2(x,y)$ & $F_3(x,y)$ & $F_4(x,y)$ & $F_5(x,y)$ & $F_6(x,y)$ & $F_7(x,y)$ & $F_8(x,y)$\\ \hline
            1296  & 369 & 345 & 382 & 392                         & 288 & 339 & 295 & 380 \\
            2401  & 446 & 401 & 459 & 462                         & 326 & 383 & 337 & 427 \\
            4096  & 506 & 444 & 517 & 511                         & 360 & 425 & 366 & 479 \\
            10000 & 582 & 505 & 592 & 605                         & 404 & 481 & 420 & 537 \\
            20736 & 643 & 551 & 657 & 671                         & 437 & 516 & 458 & 600 \\
            38416 & 690 & 597 & 705 & 738                         & 472 & 547 & 485 & 642 \\
            50625 & 714 & 614 & 735 & 758                         & 486 & 566 & 496 & 662\\\hline
        \end{tabular}}} \label{tab:res_4d_ver_log} \quad
        \subfloat[$3$D vertex-sharing]{\resizebox{0.4\textwidth}{!}{
        \begin{tabular}{p{1cm}|cccccccc}\hline
            $N$ & $F_1(x,y)$ & $F_2(x,y)$ & $F_3(x,y)$ & $F_4(x,y)$ & $F_5(x,y)$ & $F_6(x,y)$ & $F_7(x,y)$ & $F_8(x,y)$\\ \hline
            125   & 74  & 93  & 77  & 97  & 86  & 93  & 86  & 92  \\
            1000  & 132 & 175 & 132 & 214 & 138 & 153 & 142 & 164 \\
            3375  & 162 & 213 & 165 & 269 & 162 & 185 & 172 & 192 \\
            8000  & 180 & 241 & 180 & 309 & 179 & 199 & 189 & 216 \\
            15625 & 198 & 259 & 200 & 344 & 190 & 210 & 198 & 228 \\
            27000 & 207 & 269 & 211 & 367 & 198 & 219 & 209 & 238 \\
            42875 & 220 & 281 & 220 & 388 & 203 & 223 & 217 & 247 \\
            64000 & 225 & 292 & 230 & 410 & 208 & 230 & 228 & 252 \\ \hline
        \end{tabular}}} \label{tab:res_3d_ver_log} \quad
        \subfloat[$2$D vertex-sharing]{\resizebox{0.4\textwidth}{!}{
        \begin{tabular}{p{1cm}|cccccccc}\hline
            $N$ & $F_1(x,y)$ & $F_2(x,y)$ & $F_3(x,y)$ & $F_4(x,y)$ & $F_5(x,y)$ & $F_6(x,y)$ & $F_7(x,y)$ & $F_8(x,y)$\\ \hline
            1600  & 81  & 34 & 82  & 94  & 59 & 64 & 59 & 66 \\
            2500  & 87  & 36 & 87  & 102 & 64 & 67 & 63 & 70 \\
            5625  & 96  & 39 & 97  & 114 & 67 & 72 & 66 & 74 \\
            10000 & 104 & 41 & 104 & 122 & 70 & 73 & 69 & 78 \\
            22500 & 112 & 44 & 113 & 135 & 72 & 77 & 71 & 81 \\
            40000 & 119 & 45 & 119 & 143 & 76 & 81 & 76 & 84\\  \hline
        \end{tabular}}} \label{tab:res_2d_ver_log} \quad
        \subfloat[$1$D vertex-sharing]{\resizebox{0.4\textwidth}{!}{
        \begin{tabular}{p{1cm}|cccccccc}\hline
            $N$ & $F_1(x,y)$ & $F_2(x,y)$ & $F_3(x,y)$ & $F_4(x,y)$ & $F_5(x,y)$ & $F_6(x,y)$ & $F_7(x,y)$ & $F_8(x,y)$\\ \hline
            1000  & 22 & 20 & 22 & 22 & 2 & 2 & 7 & 1 \\ 
            5000  & 27 & 23 & 27 & 26 & 2 & 2 & 7 & 1 \\
            10000 & 29 & 25 & 29 & 27 & 2 & 2 & 7 & 1 \\
            15000 & 30 & 26 & 30 & 28 & 2 & 2 & 7 & 1 \\
            20000 & 31 & 26 & 31 & 29 & 2 & 2 & 7 & 1 \\
            25000 & 31 & 27 & 31 & 30 & 2 & 2 & 7 & 1 \\
            30000 & 32 & 27 & 32 & 30 & 2 & 2 & 7 & 1 \\
            40000 & 33 & 27 & 33 & 31 & 2 & 2 & 7 & 1 \\ \hline
        \end{tabular}}} \label{tab:res_1d_vert_log}
        \caption{Numerical rank $(p_{\epsilon})$ with $\bkt{\epsilon = 10^{-12}}$ of different kernel functions of vertex-sharing interaction in $4$D, $3$D, $2$D, $1$D}
    \end{table}
\begin{table}[H]
       \centering         \setlength\tabcolsep{0.1pt}
        \subfloat[$4$D 1-hyper-surface sharing]{        \resizebox{0.4\textwidth}{!}{
        \begin{tabular}{p{1cm}|cccccccc}\hline
            $N$ & $F_1(x,y)$ & $F_2(x,y)$ & $F_3(x,y)$ & $F_4(x,y)$ & $F_5(x,y)$ & $F_6(x,y)$ & $F_7(x,y)$ & $F_8(x,y)$\\ \hline
            1296  & 482  & 450  & 496  & 492  & 378 & 426 & 376 & 460  \\
            2401  & 585  & 555  & 601  & 606  & 447 & 491 & 449 & 554  \\
            4096  & 693  & 629  & 705  & 706  & 495 & 560 & 502 & 633  \\
            10000 & 862  & 773  & 879  & 891  & 595 & 671 & 608 & 772  \\
            20736 & 1029 & 898  & 1044 & 1065 & 682 & 761 & 696 & 891  \\
            38416 & 1179 & 1008 & 1197 & 1226 & 743 & 840 & 767 & 994  \\
            50625 & 1251 & 1058 & 1269 & 1298 & 774 & 880 & 798 & 1031 \\ \hline
        \end{tabular}}} \label{tab:res_4d_h1_log} \quad
        \subfloat[$3$D 1-hyper-surface (edge) sharing]{        \resizebox{0.4\textwidth}{!}{
        \begin{tabular}{p{1cm}|cccccccc}\hline
            $N$ & $F_1(x,y)$ & $F_2(x,y)$ & $F_3(x,y)$ & $F_4(x,y)$ & $F_5(x,y)$ & $F_6(x,y)$ & $F_7(x,y)$ & $F_8(x,y)$\\ \hline
            125   & 86  & 104 & 89                          & 109  & 97  & 102 & 97  & 102 \\
            1000  & 189 & 260 & 191                         & 302  & 205 & 215 & 209 & 228 \\
            3375  & 271 & 371 & 273                         & 458  & 272 & 292 & 281 & 310 \\
            8000  & 345 & 466 & 346                         & 600  & 336 & 354 & 349 & 382 \\
            15625 & 416 & 552 & 418                         & 741  & 384 & 410 & 397 & 446 \\
            27000 & 483 & 637 & 485                         & 877  & 435 & 457 & 450 & 503 \\
            42875 & 546 & 711 & 548                         & 1006 & 469 & 505 & 489 & 557 \\
            64000 & 602 & 783 & 605                         & 1136 & 511 & 563 & 531 & 604 \\\hline
        \end{tabular}}} \label{tab:res_3d_edge_log} \quad
        \subfloat[$2$D 1-hyper-surface (edge) sharing]{        \resizebox{0.4\textwidth}{!}{
        \begin{tabular}{p{1cm}|cccccccc}\hline
            $N$ & $F_1(x,y)$ & $F_2(x,y)$ & $F_3(x,y)$ & $F_4(x,y)$ & $F_5(x,y)$ & $F_6(x,y)$ & $F_7(x,y)$ & $F_8(x,y)$\\ \hline
            1600  & 216 & 99  & 217 & 241  & 158 & 162 & 157 & 165 \\
            2500  & 266 & 120 & 266 & 296  & 191 & 195 & 187 & 198 \\
            5625  & 382 & 172 & 382 & 434  & 253 & 260 & 249 & 276 \\
            10000 & 495 & 223 & 496 & 570  & 316 & 323 & 310 & 335 \\
            22500 & 717 & 323 & 718 & 842  & 431 & 442 & 422 & 461 \\
            40000 & 936 & 423 & 938 & 1112 & 536 & 550 & 525 & 576\\ \hline
        \end{tabular}}} \label{tab:res_2d_edge_log} \quad
        \caption{Numerical rank $(p_{\epsilon})$ with $\bkt{\epsilon = 10^{-12}}$ of different kernel functions of 1-hyper-surface-sharing interaction in $4$D, $3$D, $2$D}%
    \end{table}
\begin{table}[H]
       \centering         \setlength\tabcolsep{0.1pt}
        \subfloat[$4$D 2-hyper-surface sharing]{        \resizebox{0.4\textwidth}{!}{
        \begin{tabular}{p{1cm}|cccccccc}\hline
            $N$ & $F_1(x,y)$ & $F_2(x,y)$ & $F_3(x,y)$ & $F_4(x,y)$ & $F_5(x,y)$ & $F_6(x,y)$ & $F_7(x,y)$ & $F_8(x,y)$\\ \hline
            1296  & 639  & 624  & 651  & 649  & 526  & 565  & 522  & 599  \\
            2401  & 856  & 825  & 873  & 862  & 673  & 725  & 669  & 784  \\
            4096  & 1082 & 1018 & 1096 & 1080 & 810  & 875  & 807  & 965  \\
            10000 & 1555 & 1432 & 1568 & 1562 & 1096 & 1186 & 1107 & 1332 \\
            20736 & 2071 & 1873 & 2091 & 2096 & 1394 & 1504 & 1408 & 1717 \\
            38416 & 2640 & 2341 & 2653 & 2683 & 1693 & 1830 & 1716 & 2102 \\
            50625 & 2949 & 2585 & 2964 & 2996 & 1845 & 1991 & 1873 & 2297 \\ \hline
        \end{tabular}}} \label{tab:res_4d_h2_log} \quad
        \subfloat[$3$D 2-hyper-surface (face) sharing]{        \resizebox{0.4\textwidth}{!}{
        \begin{tabular}{p{1cm}|cccccccc}\hline
            $N$ & $F_1(x,y)$ & $F_2(x,y)$ & $F_3(x,y)$ & $F_4(x,y)$ & $F_5(x,y)$ & $F_6(x,y)$ & $F_7(x,y)$ & $F_8(x,y)$\\ \hline
            125   & 98   & 118  & 99                           & 119  & 111  & 115  & 108  & 113  \\
            1000  & 312  & 422  & 312                          & 465  & 335  & 348  & 332  & 360  \\
            3375  & 610  & 815  & 609                          & 964  & 627  & 642  & 619  & 668  \\
            8000  & 1003 & 1314 & 1003                         & 1623 & 983  & 1006 & 978  & 1046 \\
            15625 & 1491 & 1931 & 1494                         & 2443 & 1392 & 1430 & 1401 & 1497 \\
            27000 & 2077 & 2641 & 2082                         & 3419 & 1874 & 1914 & 1887 & 2012 \\
            42875 & 2764 & 3465 & 2766                         & 4556 & 2393 & 2460 & 2397 & 2587 \\
            64000 & 3547 & 4373 & 3551                         & 5751 & 2960 & 3030 & 2965 & 3224\\ \hline
        \end{tabular}}} \label{tab:res_3d_face_log} \quad
        \caption{Numerical rank $(p_{\epsilon})$ with $\bkt{\epsilon = 10^{-12}}$ of different kernel functions of 2-hyper-surface-sharing interaction in $4$D, $3$D}%
    \end{table}
\begin{table}[H]
       \centering         \setlength\tabcolsep{0.1pt}
        \subfloat[$4$D 3-hyper-surface sharing]{        \resizebox{0.4\textwidth}{!}{
        \begin{tabular}{p{1cm}|cccccccc}\hline
            $N$ & $F_1(x,y)$ & $F_2(x,y)$ & $F_3(x,y)$ & $F_4(x,y)$ & $F_5(x,y)$ & $F_6(x,y)$ & $F_7(x,y)$ & $F_8(x,y)$\\ \hline
            1296  & 842  & 857  & 846  & 841  & 752  & 782  & 733  & 802  \\
            2401  & 1251 & 1267 & 1266 & 1245 & 1093 & 1135 & 1075 & 1177 \\
            4096  & 1755 & 1763 & 1767 & 1741 & 1486 & 1542 & 1448 & 1638 \\
            10000 & 3114 & 3059 & 3127 & 3056 & 2509 & 2584 & 2475 & 2759 \\
            20736 & 4963 & 4807 & 4973 & 4901 & 3917 & 4006 & 3853 & 4303 \\
            38416 & 7461 & 7122 & 7473 & 7337 & 5736 & 5894 & 5677 & 6254 \\
            50625 & 8925 & 8496 & 8940 & 8814 & 6813 & 7010 & 6755 & 7440 \\\hline
        \end{tabular}}}
        \caption{Numerical rank $(p_{\epsilon})$ with $\bkt{\epsilon = 10^{-12}}$ of different kernel functions of 3-hyper-surface-sharing interaction in $4$D}
        \label{tab:res_4d_h3_log}%
    \end{table}
\end{document}


\section{Rank growth of different interactions in 1D} \label{1d}
 Using the~\cref{th3}, we discuss the rank of far-field and vertex sharing interactions as shown in~\cref{1d_interaction}. $I$ and V are the far-field and vertex sharing domains of $Y$ respectively. 
     \begin{figure}[H]
     	\centering \subfloat[Far field and vertex sharing domains of $Y$]{\label{1d_interaction}\resizebox{4cm}{!}{
            \begin{tikzpicture}
            \draw[ultra thick] (0,0) -- (3,0);
            \redcircle{0}{0};
            \redcircle{1}{0};
            \redcircle{2}{0};
            \redcircle{3}{0};
            \node at (0.5,0.2) {$I$};
            \node at (1.5,0.2) {V};
            \node at (2.5,0.2) {$Y$};
            \end{tikzpicture}
        }}
        \qquad \qquad \qquad
        \centering \subfloat[Far-field domain]{\label{1d_far}\resizebox{4cm}{!}{
        \begin{tikzpicture}
            \draw[|-|,dashed] (0,0) node[anchor=north] {$-r$} -- (2,0);
            \draw[|-|] (-2,0) node[anchor=north] {$-2r$} -- (0,0);
            \draw[|-|] (2,0) -- (4,0);   
            \node[anchor=north] at (2,0) {$0$};
            \node[anchor=north] at (4,0) {$r$};
            
            \node[anchor=north] at (-1,0.75) {X};
            \node[anchor=north] at (3,0.75) {Y};
        \end{tikzpicture}
        }}
        \caption{Different interactions in $1D$}
    \end{figure}
    \subsection{\textbf{\textit{Rank growth of far-field domains}}} Here we will assume far-field interaction means the interaction between two domains which are one line-segment away. In~\cref{1d_far} the lines $Y = [0,r]$ and $X = [-2r,-r]$ are one line-segment away. We choose $\bkt{p+1}$ Chebyshev nodes in the domain $Y$ to obtain the polynomial interpolation of $F(\xb,\yb)$ along $\yb$. Let $\Tilde{K}$ be the approximation of the matrix \(K\). Then by~\cref{th3} the error in \mvp at $i^{th}$ component
    \begin{equation}
        \Bigl\lvert \bkt{K \pmb{q}-\Tilde{K} \pmb{q}}_i \Bigl\lvert \leq \frac{4M Q \rho^{-p}}{\rho -1}
    \end{equation}
    Hence, setting the error to be less than $\delta$ (for some $\delta >0$), we have
    $p= \ceil{\frac{\log\bkt{\frac{4MQ}{\delta(\rho-1)}}}{\log(\rho)}} \implies \Bigl\lvert \bkt{K \pmb{q}-\Tilde{K} \pmb{q}}_i \Bigl\lvert < \delta$. Since, the rank of $\Tilde{K} = (p+1)$ i.e., the rank of $\Tilde{K}$ is bounded above by $\bkt{1+\ceil{\frac{\log\bkt{\frac{4MQ}{\delta(\rho-1)}}}{\log(\rho)}}}$.
    
    Therefore, the rank of $\Tilde{K}$ scales $\mathcal{O}\bkt{ \log\bkt{\frac{MQ}{\delta}}}$ with
    $\magn{K \pmb{q} - \Tilde{K} \pmb{q}}_{\infty} < \delta$. The numerical rank plots of the far-field interaction of the eight functions as described in~\Cref{Preliminaries} are shown in~\cref{fig:res_1d_far} and tabulated in~\cref{tab:res_1d_far}.

\cmt{
    \begin{figure}[H]
    \centering
    \subfloat[][N vs Rank plot\label{fig:res_1d_farc}]{\includegraphics[width=0.5\columnwidth]{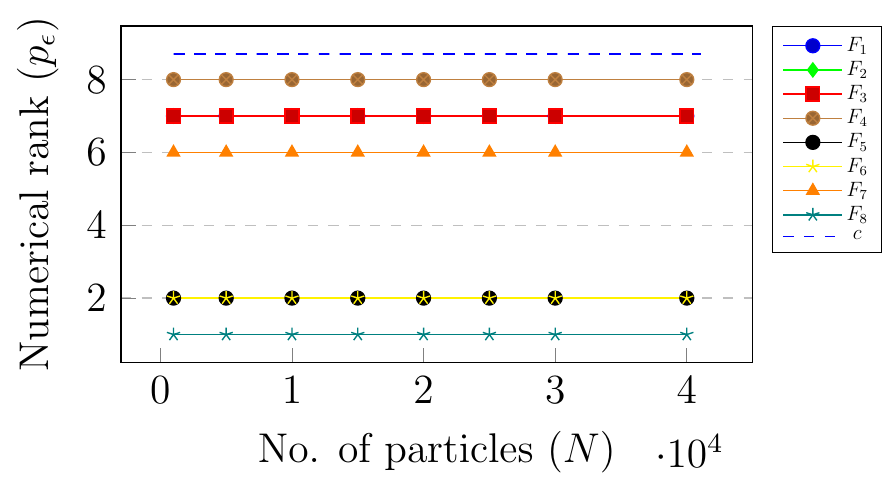}}\qquad
    \subfloat[][Numerical rank of different kernel functions]
    {\adjustbox{width=0.4\columnwidth,valign=B,raise=1.4\baselineskip}{%
      \renewcommand{\arraystretch}{1}%
      \label{tab:res_1d_farc}%
        \begin{tabular}{c|cccc}\hline
            $N$ & $F_1(x,y)$ & $F_2(x,y)$ & $F_3(x,y)$ & $F_4(x,y)$ \\ \hline
            1000 & 7 & 7 & 8 & 8 \\
            5000 & 7 & 7 & 8 & 8 \\
            10000 & 7 & 7 & 8 & 8 \\
            15000 & 7 & 7 & 8 & 8 \\
            20000 & 7 & 7 & 8 & 8 \\
            25000 & 7 & 7 & 8 & 8 \\
            30000 & 7 & 7 & 8 & 8 \\
            40000  & 7 & 7 & 8 & 8 \\ \hline
        \end{tabular}}
    }
    \caption{Numerical rank $\bkt{tol = 10^{-12}}$ of far-field interaction in $1D$}
    \label{result_1d_farc}
    \end{figure}
}
    \subsection{\textbf{\textit{Rank growth of vertex sharing domains}}}
    Consider a pair of vertex sharing line-segments (domains) $Y = [0,r]$ and $X = [-r,0]$ of length $r$ as shown in~\cref{1d_ver}. The domain \(Y\) is hierarchically sub-divided using an adaptive binary tree as shown in~\cref{1d_ver_sub_l1,1d_ver_sub}, which gives

    \begin{equation}
        Y = \overbrace{Y_2 \bigcup Y_1}^{\text{At level 1}} = \overbrace{\left[0,\dfrac{r}{2} \right] \bigcup  \left[\dfrac{r}{2},r \right]}^{\text{At level 1}} = \dots = \overbrace{Y_{\kappa +1} \bigcup_{k=1}^{\kappa} Y_{k}}^{\text{At level }\kappa}  = \overbrace{\left[0,\dfrac{r}{2^{\kappa}} \right] \bigcup_{k=1}^{\kappa} \left[\dfrac{r}{2^{k}}, \dfrac{r}{2^{k-1}} \right]}^{\text{At level }\kappa}
    \end{equation} 
    where $\kappa \sim \log_2(N)$ and $ Y_{\kappa +1 } = \left[0,\dfrac{r}{2^{\kappa}} \right]$ having one particle. Let $\Tilde{K_k}$ be the approximation of $K_k$, then the approximation of the matrix \(K\) is given by
    \begin{equation}
        \Tilde{K} = \dsum_{k=1}^{\kappa} \Tilde{K_k} + K_{\kappa+1}
    \end{equation}
    We choose a $\bkt{p_k+1}$ Chebyshev nodes in the line-segment $Y_{k}$ to obtain the polynomial interpolation of $F(\xb,\yb)$ along $\yb$. Then by~\cref{th3} the $i^{th}$ component error in \mvp at level \(k\) is
    \begin{equation}
        \Bigl\lvert \bkt{K_k \pmb{q}-\Tilde{K}_k \pmb{q}}_i \Bigl\lvert \leq \frac{4M_k Q_k \rho^{-p_k}}{\rho -1}
    \end{equation}
    Now choosing $p_{k}$ such that the above error is less than $\delta_1$ (for some $\delta_1>0$), we obtain \\ $p_k = \bkt{\ceil{\frac{\log\bkt{\frac{4 M_k Q_k}{\delta_1(\rho-1)}}}{\log(\rho)}}} \implies \abs{\bkt{K_k \pmb{q}-\Tilde{K}_k \pmb{q}}_i} < \delta_1$ with rank of $\Tilde{K}_{k} $ is $ (1 + p_k)$.\\
    Let $p_l = \max \{p_{k} : k=1,2,\hdots, \kappa \}$, which corresponds to $\Tilde{K}_l$. Hence, we get
    \begin{align}
    \begin{split}
        \Bigl\lvert \bkt{K \pmb{q} - \Tilde{K} \pmb{q}}_i \Bigl\lvert =\Bigl\lvert \bkt{\bkt{K_1+ K_2+\cdots + K_\kappa}q - \bkt{\Tilde{K_1}+\Tilde{K_2}+\cdots + \Tilde{K_\kappa}} \pmb{q}}_i \Bigl\lvert \\
        \leq \Bigl\lvert \bkt{K_1 \pmb{q}-\Tilde{K}_1 \pmb{q}}_i \Bigl\lvert + \Bigl\lvert \bkt{K_2 \pmb{q}-\Tilde{K}_2 \pmb{q}}_i \Bigl\lvert + \cdots + \Bigl\lvert \bkt{K_\kappa \pmb{q}-\Tilde{K}_\kappa \pmb{q}}_i \Bigl\lvert  < \kappa \delta_1
    \end{split}
    \end{align}
    The rank of $\Tilde{K}$ is bounded above by 
    $\bkt{1+ \kappa \bkt{1+p_l}} = \bkt{1 + \kappa \bkt{1+\ceil{\frac{\log\bkt{\frac{4 M_l Q_l}{\delta_1(\rho-1)}}}{\log(\rho)}}}}$.
    Note that the rank of $K_{\kappa +1}$ is one. If we choose $\delta_1 = \dfrac{\delta}{\kappa}$, then $\Bigl\lvert \bkt{K \pmb{q} - \Tilde{K} \pmb{q}}_i \Bigl\lvert < \delta$ with rank of $\Tilde{K} $ bounded above by $ 1 +  \kappa \bkt{1+\ceil{\frac{ \log \bkt{\frac{4 \kappa M_{l} Q_{l}}{\delta (\rho-1)}}}{\log(\rho)}}} = 1 + \log_2(N) \bkt{1+\ceil{\frac{ \log \bkt{\frac{4 \log_2(N) M_{l} Q_{l}}{\delta (\rho-1)}}}{\log(\rho)}}}$.

    Therefore, the rank of $\Tilde{K}$ scales $\mathcal{O}\bkt{\log_2(N) \log\bkt{\frac{M_l Q_l \log_2(N)}{\delta}}}$ with
    $\magn{K \pmb{q} - \Tilde{K} \pmb{q}}_{\infty} < \delta$. The numerical rank plots of the vertex sharing interaction of the eight functions as described in~\Cref{Preliminaries} are shown in~\cref{fig:res_1d_vert} and tabulated in~\cref{tab:res_1d_vert}.
    
\cmt{
    \begin{figure}[H]
    \centering
    \subfloat[][N vs Rank plot]{\includegraphics[width=0.5\columnwidth]{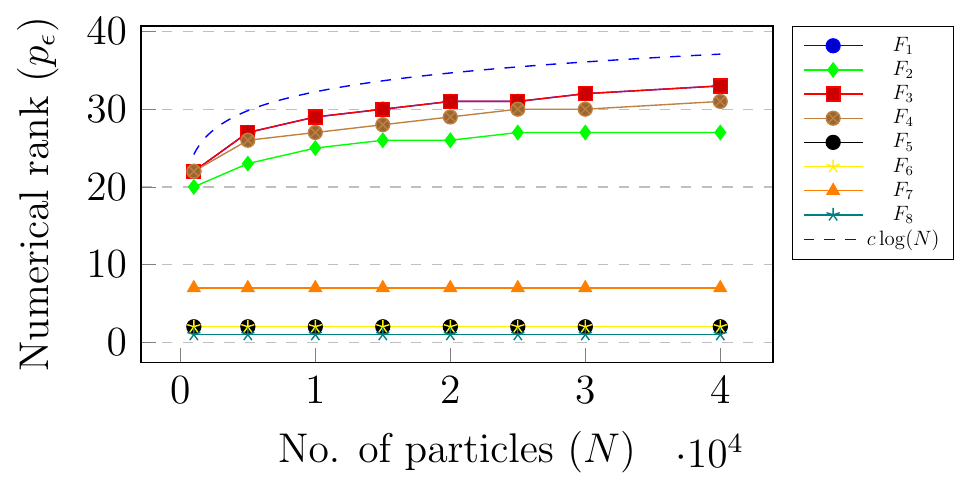}}
    \qquad
    \subfloat[][Numerical rank of different kernel functions]
    {\adjustbox{width=0.4\columnwidth,valign=B,raise=1.4\baselineskip}{%
      \renewcommand{\arraystretch}{1}%
        \begin{tabular}{c|cccc}\hline
            $N$ & $F_1(x,y)$ & $F_2(x,y)$ & $F_3(x,y)$ & $F_4(x,y)$ \\ \hline
            1000 & 22 & 20 & 22 & 22 \\
            5000 & 27 & 23 & 27 & 26 \\
            10000 & 29 & 25 & 29 & 27 \\
            15000 & 30 & 26 & 30 & 28 \\
            20000 & 31 & 26 & 31 & 29 \\
            25000 & 31 & 27 & 31 & 30 \\
            30000 & 32 & 27 & 32 & 30 \\
            40000 & 33 & 27 & 33 & 31 \\ \hline
        \end{tabular}}
    }
    \caption{Numerical rank $\bkt{tol = 10^{-12}}$ of vertex sharing interaction in $1D$}
    \label{result_1d_verc}
    \end{figure}
}


\section{Rank growth of different interactions in 2D} \label{2d}  We discuss the rank of far-field, vertex sharing and edge sharing interactions as shown in~\cref{2d_inter}. $I$, V and $E$ are far-field, vertex sharing and edge sharing domains of $Y$ respectively. 

\begin{figure}[H]
    \centering   \subfloat[Far field, Vertex sharing and Edge sharing of $Y$]{\label{2d_inter}\resizebox{5cm}{!}{
        \begin{tikzpicture}[scale=1]
            \redsquare{0}{0};
            \redsquare{1}{0};
            \redsquare{2}{0};
            \redsquare{1}{-1};
            \node at (0.5,0.5) {$I$};
            \node at (1.5,-0.5) {V};
            \node at (1.5,0.5) {$E$};
            \node at (2.5,0.5) {$Y$};
        \end{tikzpicture}
    \label{2d_interactions}
        }}
        \qquad \qquad \qquad
    \centering   \subfloat[Far-field square boxes]{\label{2d_far_pic}\resizebox{5cm}{!}{
		\begin{tikzpicture}[scale=0.75]
			\draw (-1,-1) rectangle (1,1);
			\draw (3,-1) rectangle (5,1);
			\node at (0,0) {$X$};
			\node at (4,0) {$Y$};
			\draw [<->] (-1,-1.25) -- (1,-1.25);
			\draw [<->] (1,-1.25) -- (3,-1.25);
			\draw [<->] (3,-1.25) -- (5,-1.25);
			\draw [<->] (-1.25,-1) -- (-1.25,1);
			\node at (2,-1.5) {$r$};
			\node at (0,-1.5) {$r$};
			\node at (4,-1.5) {$r$};
			\node at (-1.5,0) {$r$};
		\end{tikzpicture}
		}}
    \caption{Different interactions in $2D$}
    \label{2d_far}
\end{figure}

    \subsection{\textbf{\textit{Rank growth of far-field domains}}} Let $X$ and $Y$ be square boxes of length $r$ and the distance between them is also $r$ as shown in ~\cref{2d_far_pic}. We choose a $\bkt{p+1}\times\bkt{p+1}$ tensor product grid on Chebyshev nodes in the square $Y$ to obtain the polynomial interpolation of $F(\xb,\yb)$ along $\yb$. Let $\Tilde{K}$ be the approximation of the matrix \(K\). Then by~\cref{eq1} the $i^{th}$ component error in \mvp is
        \begin{equation}
            \Bigl\lvert \bkt{K \pmb{q}-\Tilde{K} \pmb{q}}_i \Bigl\lvert \leq 4 M Q V_2 \frac{\rho^{-p}}{\rho -1}
        \end{equation}
        Hence, setting the error to be less than $\delta$ (for a given $\delta >0$),
            $ p =  \ceil{\frac{ \log \bkt{\frac{4M Q V_2}{\delta (\rho-1)}}}{\log(\rho)}} \implies \Bigl\lvert \bkt{K \pmb{q}-\Tilde{K} \pmb{q}}_i \Bigl\lvert < \delta$.
        Since, the rank of $\Tilde{K} =(p+1)^2$ i.e., the rank of $\Tilde{K}$ is bounded above by $ \bkt{1 + \ceil{\frac{ \log \bkt{\frac{4 M Q V_2}{\delta (\rho-1)}}}{\log(\rho)}}}^2.$ 
        
        Therefore, the rank of $\Tilde{K}$ scales
        $\mathcal{O}\bkt{\bkt{ \log \bkt{\frac{MQ}{\delta}}}^2}$ with $\magn{K \pmb{q} - \Tilde{K} \pmb{q}}_{\infty} < \delta$. The numerical rank plots of the far-field interaction of the eight functions as described in~\Cref{Preliminaries} are shown in~\cref{fig:res_2d_far} and tabulated in~\cref{tab:res_2d_far}

\cmt{
    \begin{figure}[H]
    \centering
    \subfloat[][N vs Rank plot]{\includegraphics[width=0.5\columnwidth]{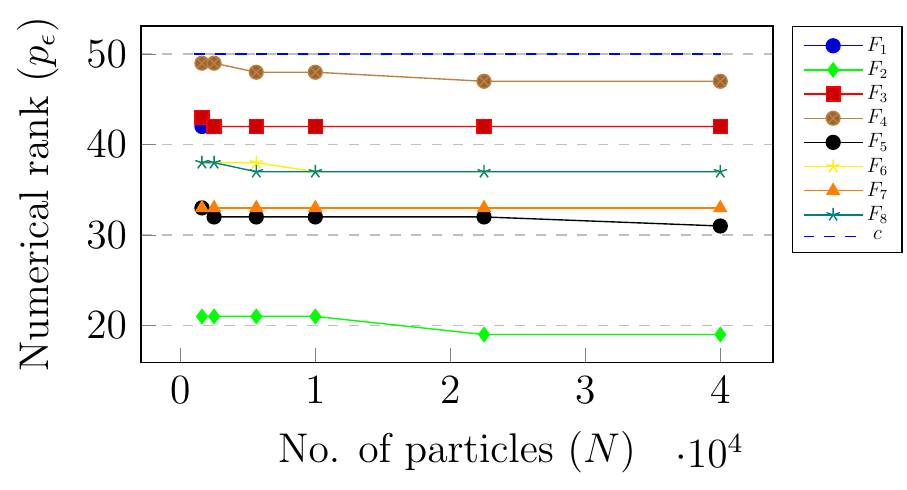}}\qquad
    \subfloat[][Numerical rank of different kernel functions]
    {\adjustbox{width=0.4\columnwidth,valign=B,raise=1.4\baselineskip}{%
      \renewcommand{\arraystretch}{1}%
        \begin{tabular}{c|cccc}\hline
            $N$ & $F_1(x,y)$ & $F_2(x,y)$ & $F_3(x,y)$ & $F_4(x,y)$ \\ \hline
            1600 & 42 & 21 & 48	& 49 \\
            2500 & 42 & 21 & 48	& 49 \\
            5625 & 42 & 21 & 47	& 48 \\
            10000 & 42 & 21	& 47 & 48 \\
            22500 & 42 & 19	& 47 & 47 \\
            40000  & 42	& 19 & 46 & 47 \\ \hline
        \end{tabular}}
    }
    \caption{Numerical rank $\bkt{tol = 10^{-12}}$ of far-field interaction in $2D$}
    \label{result_2d_far}
\end{figure}
}
    \subsection{\textbf{\textit{Rank growth of vertex sharing domains}}}
        Consider two vertex sharing square domains of length \(r\), \(X\) and \(Y\) as shown in~\cref{2d_vertex}. The box $Y$ is hierarchically sub-divided using an adaptive quad tree as shown in~\cref{2d_vertex_divide_level1,2d_vertex_divide}, which gives

        \begin{equation}
            Y = \overbrace{Y_2 \bigcup_{j=1}^{2^2-1} Y_{1,j}}^{\text{At level 1}} = \dots = \overbrace{Y_{\kappa + 1} \bigcup_{k=1}^{\kappa} \bigcup_{j=1}^{3} Y_{k,j}}^{\text{At level } \kappa}
        \end{equation}
        where $\kappa \sim \log_4(N)$ and $ Y_{\kappa +1 }$ having one particle. Let $\Tilde{K}$ be the approximation of the kernel matrix \(K\).
        Let $\Tilde{K}_{k,j}$ be approximation of the matrix ${K_{k,j}}$, then approximation of matrix $K$ is given by
        \begin{equation}
            \Tilde{K} = \sum_{k=1}^{\kappa} \sum_{j=1}^{3} \Tilde{K}_{k,j} + K_{\kappa+1}
        \end{equation}
        We choose a $\bkt{p_{k,j}+1}\times\bkt{p_{k,j}+1}$ tensor product grid on Chebyshev nodes in the square $Y_{k,j}$ to obtain the polynomial interpolation of $F(\xb,\yb)$ along $\yb$. Then by~\cref{eq1} the $i^{th}$ component error in \mvp at $j^{th}$ subdivision of the level $k$ is given by (pick $p_{k,j}$ such that the error is less than $\delta_1$)
        \begin{equation}
            \Bigl\lvert \bkt{K_{k,j} \pmb{q}-\Tilde{K}_{k,j} \pmb{q}}_i \Bigl\lvert \leq 4M_{k,j} Q_{k,j} V_2 \frac{\rho^{-p_{k,j}}}{\rho -1}
        \end{equation}
        Now choosing $p_{k,j}$ such that the above error is less than $\delta_1$ (for some $\delta_1>0$), we obtain \\ 
        $p_{k,j} =  \ceil{\frac{ \log \bkt{\frac{4M_{k,j} Q_{k,j} V_2}{\delta_1 (\rho-1)}}}{\log(\rho)}} \implies  \Bigl\lvert \bkt{K_{k,j} \pmb{q}-\Tilde{K}_{k,j} \pmb{q}}_i \Bigl\lvert < \delta_1$ with rank of $\Tilde{K}_{k,j} $ is $ (1 + p_{k,j})^2$. \\
    Let $p_{l,m} = \max \{p_{k,j} : k=1,2,\hdots, \kappa \text{ and } j = 1,2,3 \}$, which corresponds to $\Tilde{K}_{l,m}$.
        So, at level $k$ the error in \mvp at $i^{th}$ component is
        \begin{align}
            \Bigl\lvert \bkt{K_{k} \pmb{q}-\Tilde{K}_{k} \pmb{q}}_i \Bigl\lvert = \abs{ \sum_{j=1}^{3}  \bkt{K_{k,j} \pmb{q}-\Tilde{K}_{k,j} \pmb{q}}_i} \leq  \sum_{j=1}^{3} \Bigl\lvert \bkt{K_{k,j} \pmb{q}-\Tilde{K}_{k,j} \pmb{q}}_i \Bigl\lvert < 3 \delta_1
        \end{align}
    with rank of $\Tilde{K}_k$ is bounded above by $3 (1 + p_{l,m})^2$.
        Hence, we get 
        \begin{gather}
            \Bigl\lvert \bkt{K \pmb{q} - \Tilde{K} \pmb{q}}_i \Bigl\lvert =\Bigl\lvert \bkt{\bkt{K_1+ K_2+\cdots + K_\kappa} \pmb{q} - \bkt{\Tilde{K_1}+\Tilde{K_2}+\cdots + \Tilde{K_\kappa}} \pmb{q}}_i \Bigl\lvert
            \leq \dsum_{k=1}^{\kappa} \Bigl\lvert \bkt{K_k \pmb{q}-\Tilde{K}_k \pmb{q}}_i \Bigl\lvert  < 3 \kappa \delta_1
        \end{gather}
        and the rank of $\Tilde{K} $ is bounded above by $ 1+3 \kappa (1 + p_{l,m})^2 = 1+3 \kappa \bkt{1 + \ceil{\frac{ \log \bkt{\frac{4 M_{l,m} Q_{l,m} V_2}{\delta_1 (\rho-1)}}}{\log(\rho)}}}^2$. Note that the rank of $K_{\kappa +1}$ is one.
        If we choose $\delta_1 = \dfrac{\delta}{3 \kappa}$, then $\Bigl\lvert \bkt{K \pmb{q} - \Tilde{K} \pmb{q}}_i \Bigl\lvert < \delta$ with rank of $\Tilde{K} $ bounded above by $ 1 + 3 \kappa \bkt{1+\ceil{\frac{ \log \bkt{\frac{12 \kappa M_{l,m} Q_{l,m} V_2}{\delta (\rho-1)}}}{\log(\rho)}}}^2 = 1 + 3 \log_4(N) \bkt{1+\ceil{\frac{ \log \bkt{\frac{12 \log_4(N) M_{l,m} Q_{l,m} V_2}{\delta (\rho-1)}}}{\log(\rho)}}}^2$.
        
        Therefore, the rank of $\Tilde{K}$ scales $\mathcal{O}\bkt{\log_4(N) \bkt{ \log \bkt{\frac{M_{l,m} Q_{l,m} \log_4(N)}{\delta}}}^2}$ with
         $\magn{K \pmb{q} - \Tilde{K} \pmb{q}}_{\infty} < \delta$. The numerical rank plots of the vertex sharing interaction of the eight functions as described in~\Cref{Preliminaries} are shown in~\cref{fig:res_2d_ver} and tabulated in~\cref{tab:res_2d_ver}.
         
 \cmt{
    \begin{figure}[H]
    \centering
    \subfloat[][N vs Rank plot]{\includegraphics[width=0.5\columnwidth]{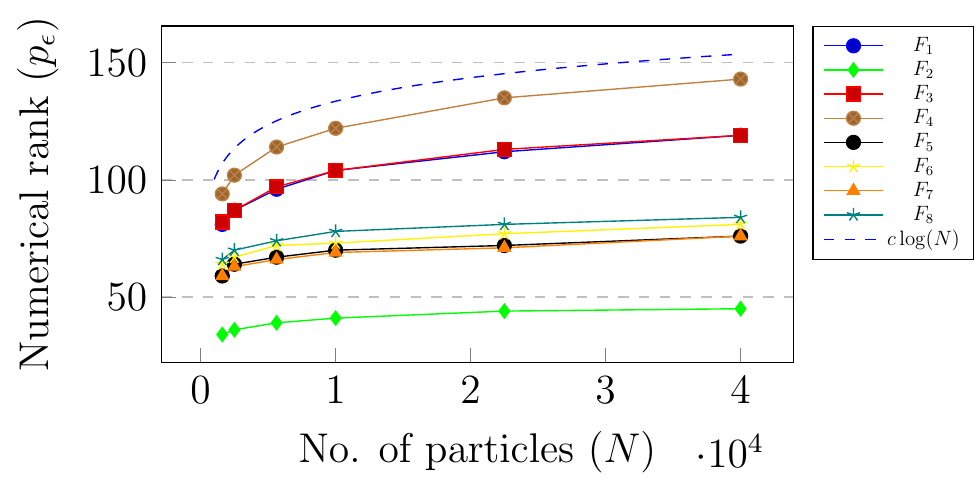}}\qquad
    \subfloat[][Numerical rank of different kernel functions]
    {\adjustbox{width=0.4\columnwidth,valign=B,raise=1.4\baselineskip}{%
      \renewcommand{\arraystretch}{1}%
        \begin{tabular}{c|cccc}\hline
            $N$ & $F_1(x,y)$ & $F_2(x,y)$ & $F_3(x,y)$ & $F_4(x,y)$ \\ \hline
            1600 & 81 & 34 & 87	& 94 \\
            2500 & 87 & 36 & 91 & 102 \\
            5625 & 96 & 39 & 101 & 114 \\
            10000 & 104	& 41 & 108 & 122 \\
            22500 & 112	& 44 & 118 & 135 \\
            40000  & 119 & 45 & 124 & 143 \\ \hline
        \end{tabular}}
    }
    \caption{Numerical rank $\bkt{tol = 10^{-12}}$ of vertex sharing interaction in $2D$}
    \label{result_2d_ver}
    \end{figure}         
 }       
    
    \subsection{\textbf{\textit{Rank growth of edge sharing domains}}}

    Consider two edge sharing square boxes \(X\) and \(Y\) as shown in~\cref{2d_edge}. The box \(Y\) is hierarchically sub-divided using an adaptive quad tree as shown in~\cref{2d_edge_divide_l1,2d_edge_divide}, which gives
    
         \begin{figure}[tbhp]
    \centering
    \subfloat[Edge sharing square boxes]{\label{2d_edge}\resizebox{4.2cm}{!}{
                    \begin{tikzpicture}[scale=1.2]
         				\draw (-1,-1) rectangle (1,1);
         				\draw (1,-1) rectangle (3,1);
         				\node at (0,0) {$X$};
         				\node at (2,0) {$Y$};
         				\draw [<->] (-1,-1.25) -- (1,-1.25);
         				\node at (0,-1.5) {$r$};
         				\draw [<->] (1,-1.25) -- (3,-1.25);
         				\node at (2,-1.5) {$r$};
         				\draw [<->] (-1.25,-1) -- (-1.25,1);
         				\node[anchor=east] at (-1.25,0) {$r$};
         			\end{tikzpicture}
        }}
    \qquad
    \subfloat[subdivision at level 1]
    {\resizebox{4.2cm}{!}{
	        \label{2d_edge_divide_l1} 
	        \begin{tikzpicture}[scale=1.2]
         				\draw (-1,-1) rectangle (1,1);
         				\draw (1,-1) rectangle (3,1);
         				
         				\draw (1,-1) grid (3,1);
         				\node at (2.5,-0.5) {$Y_{1,1}$};
         				\node at (2.5,0.5) {$Y_{1,2}$};
         				\node at (1.5,0) {$Y_{2}$};
         				
         				   
         				


 
         				\node at (0,0) {\small $X$};
         				\draw [<->] (-1,-1.25) -- (1,-1.25);
         				\node at (0,-1.5) {$r$};
         				\draw [<->] (1,-1.25) -- (3,-1.25);
         				\node at (2,-1.5) {$r$};
         				\draw [<->] (-1.25,-1) -- (-1.25,1);
         				\node[anchor=east] at (-1.25,0) {$r$};
            \end{tikzpicture}
        }
    }
        \qquad 
    \subfloat[Hierarchical subdivision]
    {\resizebox{4.2cm}{!}{
	        \label{2d_edge_divide} 
	        \begin{tikzpicture}[scale=1.2]
         				\draw (-1,-1) rectangle (1,1);
         				\draw (1,-1) rectangle (3,1);
         				
         				\draw (1,-1) grid (3,1);
         				\node at (2.5,-0.5) {$Y_{1,1}$};
         				\node at (2.5,0.5) {$Y_{1,2}$};
         				
         				\foreach \i in {0,1,2,3}
         				    \draw (1.5,-1+0.5*\i) rectangle (2,-0.5+0.5*\i);
         				   
         				\node at (1.75,-0.75) {\small $Y_{2,1}$};
         				\node at (1.75,-0.25) {\small $Y_{2,2}$};
         				\node at (1.75,0.25) {\small $Y_{2,3}$};
         				\node at (1.75,0.75) {\small $Y_{2,4}$};
         				
         				\foreach \i in {0,1,2,...,7}
         				    \draw (1.25,-1+0.25*\i) rectangle (1.5,-0.75+0.25*\i);

                        \foreach \i in {0,1,2,...,15}
         				    \draw (1.125,-1+0.125*\i) rectangle (1.25,-0.875+0.125*\i);

                        \foreach \i in {0,1,2,...,31}
         				    \draw (1.0625,-1+0.0625*\i) rectangle (1.125,-0.9375+0.0625*\i);
 
         				\node at (0,0) {\small $X$};
         				\draw [<->] (-1,-1.25) -- (1,-1.25);
         				\node at (0,-1.5) {$r$};
         				\draw [<->] (1,-1.25) -- (3,-1.25);
         				\node at (2,-1.5) {$r$};
         				\draw [<->] (-1.25,-1) -- (-1.25,1);
         				\node[anchor=east] at (-1.25,0) {$r$};
            \end{tikzpicture}
        }
    }
    \caption{Edge sharing boxes and its subdivision in $2D$}
\end{figure}
    
    \begin{equation}
        Y = \overbrace{Y_2 \bigcup_{j=1}^{2} Y_{1,j}}^{\text{At level 1}} = \dots= \overbrace{Y_{\kappa +1} \bigcup_{k=1}^{\kappa} \bigcup_{j=1}^{2^k} Y_{k,j}}^{\text{At level }\kappa}
    \end{equation}
    where $\kappa \sim \log_4(N)$ and $Y_{\kappa +1 }$ having $2^{\kappa} = \sqrt{N}$ particles. Let $\Tilde{K}_{k,j}$ be approximation of ${K_{k,j}}$, then approximation of matrix $K$ is given by
    \begin{equation}
        \Tilde{K} = \sum_{k=1}^{\kappa} \sum_{j=1}^{2^k} \Tilde{K}_{k,j} + K_{\kappa+1}
    \end{equation}
    We choose a $\bkt{p_{k,j}+1}\times\bkt{p_{k,j}+1}$ tensor product grid on Chebyshev nodes in the square $Y_{k,j}$ to obtain the polynomial interpolation of $F(\xb,\yb)$ along $\yb$. Then by~\cref{eq1} the $i^{th}$ component error in \mvp at $j^{th}$ subdivision of the level $k$ is given by (pick $p_{k,j}$ such that the error is less than $\delta_1$)
    \begin{equation}
        \Bigl\lvert \bkt{K_{k,j} \pmb{q}-\Tilde{K}_{k,j} \pmb{q}}_i \Bigl\lvert \leq 4M_{k,j} Q_{k,j} V_2 \frac{\rho^{-p_{k,j}}}{\rho -1} 
    \end{equation}
    Now choosing $p_{k,j}$ such that the above error is less than $\delta_1$ (for some $\delta_1>0$), we obtain \\ $p_{k,j}  = \ceil{\frac{ \log \bkt{\frac{4 M_{k,j} Q_{k,j} V_2}{\delta_1 (\rho-1)}}}{\log(\rho)}} \implies  \Bigl\lvert \bkt{K_{k,j} \pmb{q}-\Tilde{K}_{k,j} \pmb{q}}_i \Bigl\lvert < \delta_1$. Let $p_{l,m} = \max \{p_{k,j} : k=1,2,\hdots, \kappa \text{ and } j = 1,2, \hdots, 2^k \}$, which corresponds to $\Tilde{K}_{l,m}$.
    Note that the rank of $K_{\kappa +1}$ is $\sqrt{N}$. So, the rank of $\Tilde{K}$ is bounded above by
    \begin{align} \label{eq2}
        \begin{split}
            \sqrt{N} + \sum_{k=1}^{\kappa} 2^k \bkt{1 + \ceil{\frac{ \log \bkt{\frac{4 M_{l,m} Q_{l,m} V_2}{\delta_1 (\rho-1)}}}{\log(\rho)}}}^2 = \sqrt{N} + \bkt{2^{\kappa+1}-2}\bkt{1 + \ceil{\frac{ \log \bkt{\frac{4 M_{l,m} Q_{l,m} V_2}{\delta_1 (\rho-1)}}}{\log(\rho)}}}^2 \\ = \sqrt{N} + \bkt{2\sqrt{N}-2}\bkt{1 + \ceil{\frac{ \log \bkt{\frac{4 M_{l,m} Q_{l,m} V_2}{\delta_1 (\rho-1)}}}{\log(\rho)}}}^2
        \end{split}
    \end{align}
     Therefore, the rank of $\Tilde{K} \in \mathcal{O}\bkt{\sqrt{N}\bkt{\log\bkt{\frac{M_{l,m} Q_{l,m}}{\delta_1}}}^2}$
    and the error in \mvp
    \begin{equation}
         \Bigl\lvert \bkt{K \pmb{q} - \Tilde{K} \pmb{q}}_i \Bigl\lvert <  \sum_{k=1}^{\kappa} 2^k \delta_1 = 2\bkt{2^{\kappa}-1} \delta_1=\bkt{2^{\kappa +1}-2} \delta_1 < 2 \sqrt{N} \delta_1  ,\text{ as }  \kappa \sim \log_4(N)
    \end{equation}
    If we choose $\delta_1 = \frac{\delta}{2 \sqrt{N}}$ then
    $\Bigl\lvert \bkt{K \pmb{q} - \Tilde{K} \pmb{q}}_i \Bigl\lvert < \delta $ and the rank of $\Tilde{K} \in$ $\mathcal{O}\bkt{\sqrt{N}\bkt{\log\bkt{\frac{2\sqrt{N} M_{l,m} Q_{l,m}}{\delta}}}^2}$
    
    Hence, the rank of $\Tilde{K}$ scales $\mathcal{O}\bkt{\sqrt{N}\bkt{\log\bkt{\frac{\sqrt{N} M_{l,m} Q_{l,m}}{\delta}}}^2}$ with
    $\magn{K \pmb{q} - \Tilde{K} \pmb{q}}_{\infty} < \delta$. The numerical rank plots of the edge sharing interaction of the eight functions as described in~\Cref{Preliminaries} are shown in~\cref{fig:res_2d_edge} and tabulated in~\cref{tab:res_2d_edge}.

\cmt{
    \begin{figure}[H]
    \centering
    \subfloat[][N vs Rank plot]{\includegraphics[width=0.5\columnwidth]{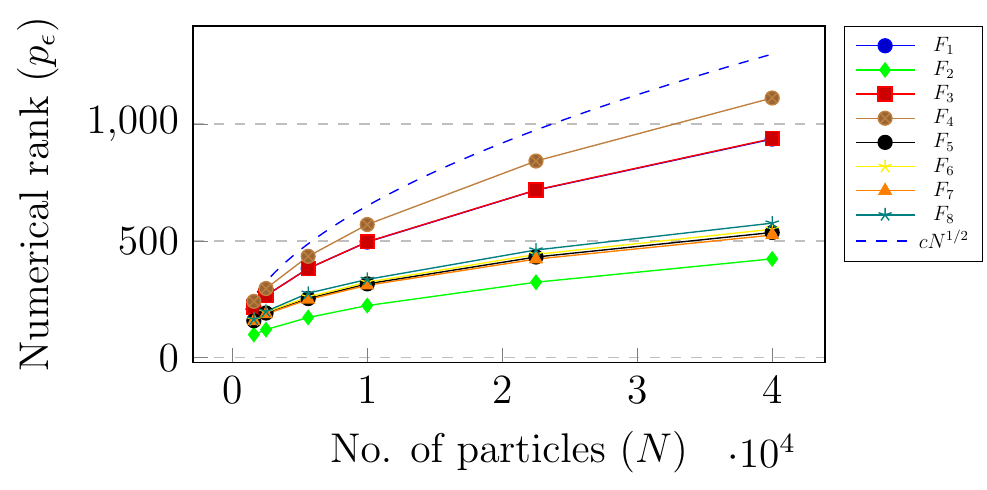}}\qquad
    \subfloat[][Numerical rank of different kernel functions]
    {\adjustbox{width=0.4\columnwidth,valign=B,raise=1.4\baselineskip}{%
      \renewcommand{\arraystretch}{1}%
        \begin{tabular}{c|cccc}\hline
            $N$ & $F_1(x,y)$ & $F_2(x,y)$ & $F_3(x,y)$ & $F_4(x,y)$ \\ \hline
            1600 & 216 & 99	& 220 & 241 \\
            2500 & 266 & 120 & 269 & 296 \\
            5625 & 382 & 172 & 388 & 434 \\
            10000 & 495	& 223 & 502 & 570 \\
            22500 & 717	& 323 & 727 & 842 \\
            40000  & 936 & 423 & 949 & 1112 \\ \hline
        \end{tabular}}
    }
    \caption{Numerical rank $\bkt{tol = 10^{-12}}$ of edge sharing interaction in $2D$}
    \label{result_2d_edge}
    \end{figure}
}

\section{Rank growth of different of interactions in 3D} \label{3d}
    We discuss the rank of far-field, vertex sharing, edge sharing and face sharing interactions as shown in~\cref{3d_inter}. $I$, V, $E$ and $F$ be the far-field, vertex sharing, edge sharing and face sharing domains of the cube $Y$ respectively.
    \begin{figure}[H]
    \centering  \subfloat[Different interactions of $Y$]{\label{3d_inter}\resizebox{5cm}{!}{
     	\begin{tikzpicture} 
        \cube{0}{0}
        \cube{-1}{0}
        \cube{-2}{0}
        \cube{1+\iso}{\iso}
        \cube{1+\iso}{1+\iso}
        \draw [ultra thick] (-2,1) -- (-2+\iso,1+\iso) -- (1+\iso,1+\iso);
        \draw [ultra thick] (1,0) -- (1+\iso,\iso);
        \draw [ultra thick] (2+\iso,\iso) -- (2+2*\iso,2*\iso) -- (2+2*\iso,2+2*\iso) -- (1+2*\iso,2+2*\iso) -- (1+\iso,2+\iso);
    
        \node at (0.7,0.7) {$Y$};
        \node at (-0.3,0.7) {$F$};
        \node at (-1.3,0.7) {$I$};
        \node at (1.7+\iso,0.7+\iso) {$E$};
        \node at (1.7+\iso,1.7+\iso) {V};
        \end{tikzpicture}
        }}
        \qquad \qquad \qquad
        \centering   \subfloat[Far-field cubes]{\label{3d_far}\resizebox{5cm}{!}{
                    	\begin{tikzpicture}[scale=0.5]
         				\draw (2,2,0)--(0,2,0)--(0,2,2)--(2,2,2)--(2,2,0)--(2,0,0)--(2,0,2)--(0,0,2)--(0,2,2);
         		     	\draw
         		     	(2,2,2)--(2,0,2);
         		     	\draw
         		     	(2,0,0)--(0,0,0)--(0,2,0);
         		     	\draw
         		     	(0,0,0)--(0,0,2);
         			    \node at (1,1,1) {$X$};
         		
                        \draw
                        (4,0,2) rectangle (6,2,2);
                        \draw
                        (4,0,0) rectangle (6,2,0);
                        \draw
                        (6,2,0)--(6,2,2);
                        \draw
                        (6,0,0)--(6,0,2);
                        \draw (4,0,0) -- (4,0,2);
                        \draw (4,2,0) -- (4,2,2);
                        \node at (5,1,1) {$Y$};
         			\end{tikzpicture}
        }}
        \caption{Different interactions in 3D}
    \end{figure}

    \subsection{\textbf{\textit{Rank growth of far-field domains}}} Let $X$ and $Y$ be cubes of size \(r\) separated by a distance \(r\) as shown in ~\cref{3d_far}. We choose a $\bkt{p+1}\times\bkt{p+1}\times\bkt{p+1}$ tensor product grid of Chebyshev nodes in the box $Y$ to obtain the polynomial interpolation of $F(\xb,\yb)$ along $\yb$. Let $\Tilde{K}$ be the approximation of the matrix \(K\). Then by~\cref{eq1} the $i^{th}$ component error in \mvp is given by
    \begin{equation}
        \Bigl\lvert \bkt{K \pmb{q}-\Tilde{K} \pmb{q}}_i \Bigl\lvert \leq 4 M Q V_3 \frac{\rho^{-p}}{\rho -1}
    \end{equation}
    Now, setting the above error to be less than $\delta >0$, we have
        $ p= \ceil{\frac{ \log \bkt{\frac{4 M Q V_3}{\delta (\rho-1)}}}{\log(\rho)}} \implies \Bigl\lvert \bkt{K \pmb{q}-\Tilde{K} \pmb{q}}_i \Bigl\lvert < \delta$.
        Since, the rank of $ \Tilde{K} = (p + 1)^3$ i.e., the rank of $\Tilde{K}$ is bounded above by $\bkt{1 + \ceil{\frac{ \log \bkt{\frac{4 M Q V_3}{\delta (\rho-1)}}}{\log(\rho)}}}^3.$
        
        Hence, the rank of $\Tilde{K}$ scales $\mathcal{O}\bkt{\bkt{ \log\bkt{\frac{MQ}{\delta}}}^3}$ and
       $\magn{K \pmb{q} - \Tilde{K} \pmb{q}}_{\infty} < \delta$. The numerical rank plots of the far-field interaction of the eight functions as described in~\Cref{Preliminaries} are shown in~\cref{fig:res_3d_far} and tabulated in~\cref{tab:res_3d_far}.

\cmt{
\begin{figure}[H]
    \centering
    \subfloat[][N vs Rank plot]{\includegraphics[width=0.5\columnwidth]{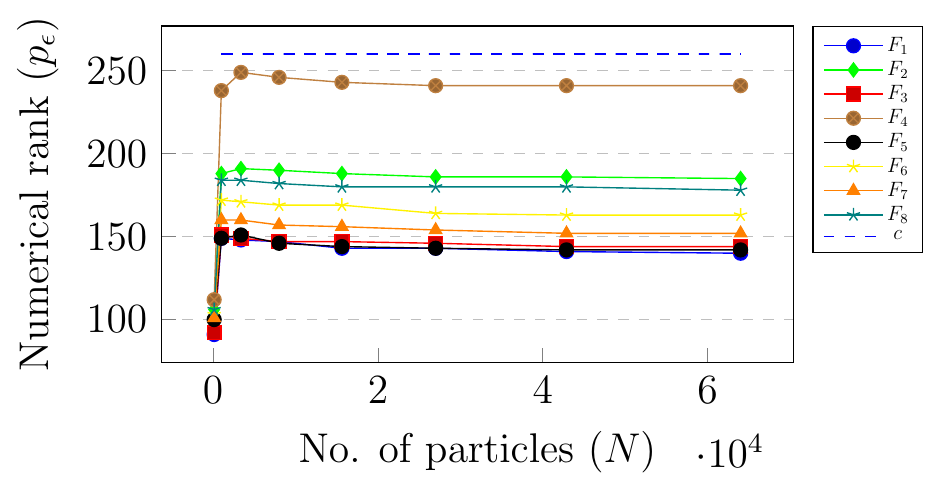}}\qquad
    \subfloat[][Numerical rank of different kernel functions]
    {\adjustbox{width=0.4\columnwidth,valign=B,raise=1.4\baselineskip}{%
      \renewcommand{\arraystretch}{1}%
        \begin{tabular}{c|cccc}\hline
            $N$ & $F_1(x,y)$ & $F_2(x,y)$ & $F_3(x,y)$ & $F_4(x,y)$ \\ \hline
            1000 & 149 & 188 & 163 & 238 \\
            3375 & 148 & 191 & 160 & 249 \\
            8000 & 147 & 190 & 158 & 246 \\
            15625 & 143	& 188 & 156 & 243 \\
            27000 & 143	& 186 & 156 & 241 \\
            42875  & 141 & 186 & 154 & 241 \\
            64000  & 140 & 185 & 154 & 241 \\
            \hline
        \end{tabular}}
    }
    \caption{Numerical rank $\bkt{tol = 10^{-12}}$ of far-field interaction in $3D$}
    \label{result_3d_far}
\end{figure}
}

    \subsection{\textbf{\textit{Rank growth of vertex sharing domains}}}
        Consider two vertex sharing cubes \(X\) (big red cube) and \(Y\) (big black cube) as shown in~\cref{3d_ver}. The black cube $Y$ is hierarchically sub-divided using an adaptive oct tree as shown in~\cref{3d_ver}. \href{https://sites.google.com/view/dom3d/vertex-sharing-domains}{\textbf{The link here provides a better $3D$ view}}. So, we can write $Y$ as

        \begin{equation}
            Y = \overbrace{Y_2 \bigcup_{j=1}^{2^3 -1}Y_{1,j}}^{\text{At level 1}} =  \dots =  \overbrace{Y_{\kappa +1} \bigcup_{k=1}^{\kappa} \bigcup_{j=1}^{7} Y_{k,j}}^{\text{At level } \kappa}
        \end{equation}
        where $\kappa \sim \log_8(N)$ and $Y_{\kappa +1}$ having one particle. Let $\Tilde{K}$ be the approximation of the kernel matrix \(K\).
        Let $\Tilde{K}_{k,j}$ be approximation of ${K_{k,j}}$, then approximation of the kernel matrix $K$ is given by
        \begin{equation}
            \Tilde{K} = \sum_{k=1}^{\kappa} \sum_{j=1}^{7} \Tilde{K}_{k,j} + K_{\kappa +1}
        \end{equation}
        We choose a $\bkt{p_{k,j}+1}\times\bkt{p_{k,j}+1}\times\bkt{p_{k,j}+1}$ tensor product grid on Chebyshev nodes in the cube $Y_{k,j}$ to obtain the polynomial interpolation of $F(\xb,\yb)$ along $\yb$. Then by~\cref{eq1} the $i^{th}$ component error in \mvp at $j^{th}$ subdivision of the level $k$ is given by
        \begin{equation}
            \Bigl\lvert \bkt{K_{k,j} \pmb{q}-\Tilde{K}_{k,j} \pmb{q}}_i \Bigl\lvert \leq 4 M_{k,j} Q_{k,j} V_3 \frac{\rho^{-p_{k,j}}}{\rho -1} 
        \end{equation}
        Now choosing $p_{k,j}$ such that the above error is less than $\delta_1$ (for some $\delta_1>0$), we obtain \\ $p_{k,j} = \ceil{\frac{ \log \bkt{\frac{4 M_{k,j} Q_{k,j} V_3}{\delta_1 (\rho-1)}}}{\log(\rho)}} \implies  \abs{ \bkt{K_{k,j} \pmb{q}-\Tilde{K}_{k,j} \pmb{q}}_i} < \delta_1$ 
        with rank of $\Tilde{K}_{k,j} $ is $ (1 + p_{k,j})^3$. \\ Let $p_{l,m} = \max \{p_{k,j} : k=1,2,\hdots, \kappa \text{ and } j = 1,2,\hdots,7 \}$, which corresponds to $\Tilde{K}_{l,m}$. So, at level $k$ the error in \mvp at $i^{th}$ component is
        \begin{align}
            \abs{\bkt{K_{k} \pmb{q}-\Tilde{K}_{k} \pmb{q}}_i} =\abs{ \sum_{j=1}^{7}  \bkt{K_{k,j} \pmb{q}-\Tilde{K}_{k,j} \pmb{q}}_i} \leq  \sum_{j=1}^{7} \abs{ \bkt{K_{k,j} \pmb{q}-\Tilde{K}_{k,j} \pmb{q}}_i} < 7 \delta_1
        \end{align}
        with the rank of $\Tilde{K}_{k} $ is bounded above by $ 7(1 + p_{l,m})^3$.
        Hence, we get 
        \begin{align}
            \begin{split}
                \Bigl\lvert \bkt{K \pmb{q} - \Tilde{K} \pmb{q}}_i \Bigl\lvert &=\Bigl\lvert \bkt{\bkt{K_1+ K_2+\cdots + K_\kappa} \pmb{q} - \bkt{\Tilde{K_1}+\Tilde{K_2}+\cdots + \Tilde{K_\kappa}} \pmb{q}}_i \Bigl\lvert \\
            &\leq \dsum_{k=1}^{\kappa} \Bigl\lvert \bkt{K_k \pmb{q}-\Tilde{K}_k \pmb{q}}_i \Bigl\lvert < 7 \kappa \delta_1
            \end{split}
        \end{align}
        and the rank of $\Tilde{K} $ is bounded by $ 1+7 \kappa (1 + p_{l,m})^3 = 1+7 \kappa \bkt{1 + \ceil{\frac{ \log \bkt{\frac{4 M_{l,m} Q_{l,m} V_3}{\delta_1 (\rho-1)}}}{\log(\rho)}}}^3$. Note that the rank of $K_{\kappa +1}$ is one.
        If we choose $\delta_1 = \dfrac{\delta}{7 \kappa}$, then $\Bigl\lvert \bkt{K \pmb{q} - \Tilde{K} \pmb{q}}_i \Bigl\lvert < \delta$ with the rank of $\Tilde{K} $ bounded above by $ 1 + 7 \kappa \bkt{1+\ceil{\frac{ \log \bkt{\frac{28 \kappa M_{l,m} Q_{l,m} V_3}{\delta (\rho-1)}}}{\log(\rho)}}}^3 = 1 + 7 \log_8(N) \bkt{1+\ceil{\frac{ \log \bkt{\frac{28 \log_8(N) M_{l,m} Q_{l,m} V_3}{\delta (\rho-1)}}}{\log(\rho)}}}^3$
    
        Therefore, the rank of $\Tilde{K}$ scales $\mathcal{O}\bkt{\log_8(N) \bkt{ \log\bkt{\frac{M_{l,m} Q_{l,m}\log_8(N)}{\delta}}}^3}$ with
       $\magn{K \pmb{q} - \Tilde{K} \pmb{q}}_{\infty} < \delta$. The numerical rank plots of the vertex sharing interaction of the eight functions as described in~\Cref{Preliminaries} are shown in~\cref{fig:res_3d_ver} and tabulated in~\cref{tab:res_3d_ver}.
        
\cmt{    
\begin{figure}[H]
    \centering
    \subfloat[][N vs Rank plot]{\includegraphics[width=0.5\columnwidth]{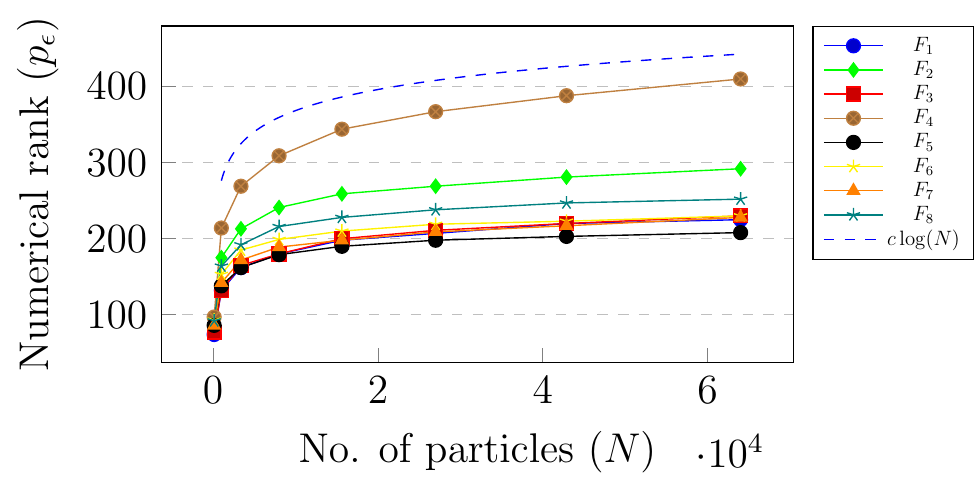}}\qquad
    \subfloat[][Numerical rank of different kernel functions]
    {\adjustbox{width=0.4\columnwidth,valign=B,raise=1.4\baselineskip}{%
      \renewcommand{\arraystretch}{1}%
        \begin{tabular}{c|cccc}\hline
            $N$ & $F_1(x,y)$ & $F_2(x,y)$ & $F_3(x,y)$ & $F_4(x,y)$ \\ \hline
            1000 & 132 & 175 & 150 & 214 \\
            3375 & 162 & 213 & 180 & 269 \\
            8000 & 180 & 241 & 205 & 309 \\
            15625 & 198	& 259 & 219 & 344 \\
            27000 & 207	& 269 & 231 & 367 \\
            42875  & 220 & 281 & 246 & 388 \\
            64000  & 225 & 292 & 253 & 410 \\
            \hline
        \end{tabular}}
    }
    \caption{Numerical rank $\bkt{tol = 10^{-12}}$ of vertex sharing interaction in $3D$}
    \label{result_3d_ver}
\end{figure} }       
    \subsection{\textbf{\textit{Rank growth of edge sharing domains}}}
        Consider edge sharing cubes \(X\) (big red cube) and \(Y\) (big black cube) as shown in~\cref{3d_edg}. The black cube $Y$ is hierarchically sub-divided using an adaptive oct tree as shown in~\cref{3d_edg}. \href{https://sites.google.com/view/dom3d/edge-sharing-domains}{\textbf{The link here provides a better $3D$ view}}. So, we can write $Y$ as
        
        \begin{equation}
            Y =  \overbrace{ Y_2 \bigcup_{j=1}^{3 \cdot 2^1} Y_{1,j}}^{\text{At level 1}} =\dots = \overbrace{Y_{\kappa+1} \bigcup_{k=1}^{\kappa} \bigcup_{j=1}^{3 \cdot 2^k} Y_{k,j}}^{\text{At level }\kappa}
        \end{equation}
    where $\kappa \sim \log_8(N)$ and $Y_{\kappa +1}$ having $2^{\kappa} = N^{1/3}$ particles. Let $\Tilde{K}_{k,j}$ be approximation of ${K_{k,j}}$, then approximation of matrix $K$ is given by
    \begin{equation}
       \Tilde{K} = \dsum_{k=1}^{\kappa} \dsum_{j=1}^{3 \cdot 2^k} \Tilde{K}_{k,j} + K_{\kappa +1} 
    \end{equation}
   We choose a $\bkt{p_{k,j}+1}\times\bkt{p_{k,j}+1}\times\bkt{p_{k,j}+1}$ tensor product grid on Chebyshev nodes in the cube $Y_{k,j}$ to obtain the polynomial interpolation of $F(\xb,\yb)$ along $\yb$. Then by~\cref{eq1} the $i^{th}$ component error in \mvp is
    \begin{equation}
        \Bigl\lvert \bkt{K_{k,j} \pmb{q}-\Tilde{K}_{k,j} \pmb{q}}_i \Bigl\lvert \leq 4M_{k,j} Q_{k,j} V_3 \frac{\rho^{-p_{k,j}}}{\rho -1}
    \end{equation}
    Now choosing $p_{k,j}$ such that the above error is less than $\delta_1$ (for some $\delta_1>0$), we obtain \\ $p_{k,j}  = \ceil{\frac{ \log \bkt{\frac{4 M_{k,j} Q_{k,j} V_3}{\delta_1 (\rho-1)}}}{\log(\rho)}} \implies \abs{\bkt{K_{k,j} \pmb{q}-\Tilde{K}_{k,j} \pmb{q}}_i} < \delta_1$. Let $p_{l,m} = \max \{p_{k,j} : k=1,2,\hdots, \kappa \text{ and } \\ j = 1,2, \hdots, 3 \cdot 2^k \}$, which corresponds to $\Tilde{K}_{l,m}$.
    Note that the rank of $K_{\kappa +1}$ is $N^{\frac{1}{3}}$. So, the rank of $\Tilde{K}$ is bounded above by
    \begin{align} \label{eq3}
        \begin{split}
            N^{\frac{1}{3}} + \sum_{k=1}^{\kappa} 3 \cdot 2^k \bkt{1 + \ceil{\frac{ \log \bkt{\frac{4 M_{l,m} Q_{l,m} V_3}{\delta_1 (\rho-1)}}}{\log(\rho)}}}^3  = N^{\frac{1}{3}} +  3 \bkt{2^{\kappa+1}-2}\bkt{1 + \ceil{\frac{ \log \bkt{\frac{4 M_{l,m} Q_{l,m} V_3}{\delta_1 (\rho-1)}}}{\log(\rho)}}}^3 \\ = N^{\frac{1}{3}} + 3 \bkt{2 N^{\frac{1}{3}}-2}\bkt{1 + \ceil{\frac{ \log \bkt{\frac{4 M_{l,m} Q_{l,m} V_3}{\delta_1 (\rho-1)}}}{\log(\rho)}}}^3 
        \end{split}
    \end{align}
    Therefore, the rank $\Tilde{K} \in \mathcal{O}\bkt{N^{\frac{1}{3}}\bkt{\log\bkt{\frac{M_{l,m} Q_{l,m}}{\delta_1}}}^3}$ and the error in \mvp
    \begin{equation}
         \Bigl\lvert \bkt{K \pmb{q} - \Tilde{K} \pmb{q}}_i \Bigl\lvert <  \sum_{k=1}^{\kappa} 3 \cdot 2^k \delta_1 = 3 \bkt{2^{\kappa +1}-2} \delta_1 < 6 N^{\frac{1}{3}} \delta_1  ,\text{ as }  \kappa \sim \log_8(N)
    \end{equation}
    If we choose $\delta_1 = \frac{\delta}{6 N^{\frac{1}{3}}}$ then
        $ \Bigl\lvert \bkt{K \pmb{q} - \Tilde{K} \pmb{q}}_i \Bigl\lvert < \delta $ and the rank of $\Tilde{K} \in$ $\mathcal{O}\bkt{N^{1/3} \bkt{\log\bkt{\frac{6N^{1/3} M_{l,m} Q_{l,m}}{\delta}}}^2}$
        
    Hence, the rank of $\Tilde{K}$ scales $\mathcal{O}\bkt{N^{1/3} \bkt{\log\bkt{\frac{N^{1/3} M_{l,m} Q_{l,m}}{\delta}}}^2}$ with 
    $\magn{K \pmb{q} - \Tilde{K} \pmb{q}}_{\infty} < \delta$. The numerical rank plots of the edge sharing interaction of the eight functions as described in~\Cref{Preliminaries} are shown in~\cref{fig:res_3d_edge} and tabulated in~\cref{tab:res_3d_edge}.

\cmt{
\begin{figure}[H]
    \centering
    \subfloat[][N vs Rank plot]{\includegraphics[width=0.5\columnwidth]{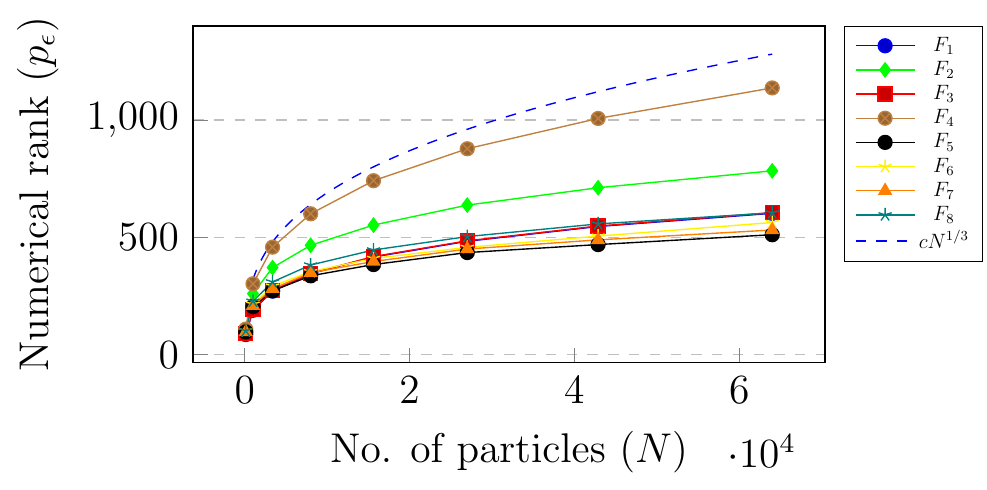}}\qquad
    \subfloat[][Numerical rank of different kernel functions]
    {\adjustbox{width=0.4\columnwidth,valign=B,raise=1.4\baselineskip}{%
      \renewcommand{\arraystretch}{1}%
        \begin{tabular}{c|cccc}\hline
            $N$ & $F_1(x,y)$ & $F_2(x,y)$ & $F_3(x,y)$ & $F_4(x,y)$ \\ \hline
            1000 & 189 & 260 & 202 & 302 \\
            3375 & 271 & 371 & 289 & 458 \\
            8000 & 345 & 466 & 366 & 600 \\
            15625 & 416	& 552 & 442 & 741 \\
            27000 & 483	& 637 & 512 & 877 \\
            42875  & 546 & 711 & 580 & 1006 \\
            64000  & 602 & 783 & 649 & 1136 \\
            \hline
        \end{tabular}}
    }
    \caption{Numerical rank $\bkt{tol = 10^{-12}}$ of edge-sharing interaction in $3D$}
    \label{result_3d_edge}
\end{figure}  }      
    \subsection{\textbf{\textit{Rank growth of face-sharing domains}}}
        Consider two face-sharing cubes \(X\) (big red cube) and \(Y\) (big black cube) as shown in~\cref{3d_fac}. The black cube $Y$ is hierarchically subdivided using an adaptive oct tree as shown~\cref{3d_fac}. \href{https://sites.google.com/view/dom3d/face-sharing-domains}{\textbf{The link here provides a better $3D$ view}}. So, we can write $Y$ as




    
        
        \begin{equation}
            Y =  \overbrace{Y_2  \bigcup_{j=1}^{4} Y_{1,j}}^{\text{At level 1}} = \dots = \overbrace{Y_{\kappa +1} \bigcup_{k=1}^{\kappa} \bigcup_{j=1}^{4^k} Y_{k,j}}^{\text{At level } \kappa}
        \end{equation}
       where $\kappa \sim \log_8(N)$ and $Y_{\kappa +1}$ having $4^{\kappa} = N^{2/3}$ particles. Let $\Tilde{K}_{k,j}$ be approximation of ${K_{k,j}}$, then approximation of matrix $K$ is
       \begin{equation}
           \Tilde{K} = \dsum_{k=1}^{\kappa} \dsum_{j=1}^{4^k} \Tilde{K}_{k,j} + K_{\kappa +1}
       \end{equation}
         We choose a $\bkt{p_{k,j}+1}\times\bkt{p_{k,j}+1}\times\bkt{p_{k,j}+1}$ tensor product grid on Chebyshev nodes in the cube $Y_{k,j}$ to obtain the polynomial interpolation of $F(\xb,\yb)$ along $\yb$. Then by~\cref{eq1} the $i^{th}$ component error in \mvp at $j^{th}$ subdivision of the level $k$ is given by (pick $p_{k,j}$ such that the error is less than $\delta_1$)
    \begin{equation}
        \Bigl\lvert \bkt{K_{k,j} \pmb{q}-\Tilde{K}_{k,j} \pmb{q}}_i \Bigl\lvert \leq 4M_{k,j} Q_{k,j} V_3 \frac{\rho^{-p_{k,j}}}{\rho -1}
    \end{equation}
     Now choosing $p_{k,j}$ such that the above error is less than $\delta_1$ (for some $\delta_1>0$), we obtain \\ $p_{k,j}  = \ceil{\frac{ \log \bkt{\frac{4 M_{k,j} Q_{k,j} V_3}{\delta_1 (\rho-1)}}}{\log(\rho)}} \implies  \Bigl\lvert \bkt{K_{k,j} \pmb{q}-\Tilde{K}_{k,j} \pmb{q}}_i \Bigl\lvert < \delta_1$. Let $p_{l,m} = \max \{p_{k,j} : k=1,2,\hdots, \kappa \text{ and } j = 1,2, \hdots, 4^k \}$, which corresponds to $\Tilde{K}_{l,m}$.
    Note that the rank of $K_{\kappa +1}$ is $N^{\frac{2}{3}}$. So, the rank of $\Tilde{K}$ is bounded above by
    \begin{align} \label{eq4}
        \begin{split}
            N^{\frac{2}{3}} + \sum_{k=1}^{\kappa} 4^k \bkt{1 + \ceil{\frac{ \log \bkt{\frac{4 M_{l,m} Q_{l,m} V_3}{\delta_1 (\rho-1)}}}{\log(\rho)}}}^3 = N^{\frac{2}{3}} +  \frac{4}{3}\bkt{4^\kappa-1} \bkt{1 + \ceil{\frac{ \log \bkt{\frac{4 M_{l,m} Q_{l,m} V_3}{\delta_1 (\rho-1)}}}{\log(\rho)}}}^3 \\ = N^{\frac{2}{3}} + \frac{4}{3}\bkt{N^\frac{2}{3}-1} \bkt{1 + \ceil{\frac{ \log \bkt{\frac{4 M_{l,m} Q_{l,m} V_3}{\delta_1 (\rho-1)}}}{\log(\rho)}}}^3 
        \end{split}
    \end{align}
  Therefore, the rank of $\Tilde{K} \in \mathcal{O}\bkt{N^{\frac{2}{3}}\bkt{\log\bkt{\frac{M_{l,m} Q_{l,m}}{\delta_1}}}^3}$  and the error in \mvp
    \begin{equation}
         \Bigl\lvert \bkt{K \pmb{q} - \Tilde{K} \pmb{q}}_i \Bigl\lvert <  \sum_{k=1}^{\kappa} 4^k \delta_1 = \frac{4}{3}\bkt{4^{\kappa}-1} \delta_1 =\frac{4}{3}\bkt{N^\frac{2}{3}-1} \delta_1 < \dfrac{4}{3} N^{\frac{2}{3}} \delta_1  ,\text{ as }  \kappa \sim \log_8(N)
    \end{equation}
    If we choose $\delta_1 =\frac{3 \delta}{4 N^{\frac{2}{3}}}$ then
        $ \Bigl\lvert \bkt{K \pmb{q} - \Tilde{K} \pmb{q}}_i \Bigl\lvert < \delta $ and the rank of $\Tilde{K} \in$ $\mathcal{O}\bkt{N^{2/3} \bkt{\log\bkt{\frac{4N^{2/3} M_{l,m} Q_{l,m}}{3 \delta}}}^3}$
        
    Hence, the rank of $\Tilde{K}$ scales $\mathcal{O}\bkt{N^{2/3} \bkt{\log\bkt{\frac{N^{2/3} M_{l,m} Q_{l,m}}{\delta}}}^3}$ with 
    $\magn{K \pmb{q} - \Tilde{K} \pmb{q}}_{\infty} < \delta$. The numerical rank plots of the face-sharing interaction of the eight functions as described in~\Cref{Preliminaries} are shown in~\cref{fig:res_3d_face,tab:res_3d_face}.
\cmt{
\begin{figure}[H]
    \centering
    \subfloat[][N vs Rank plot]{
        \includegraphics[width=0.5\columnwidth]{Images/3d_face.pdf}
    }\qquad
    \subfloat[][Numerical rank of different kernel functions]
    {\adjustbox{width=0.4\columnwidth,valign=B,raise=1.4\baselineskip}{%
      \renewcommand{\arraystretch}{1}%
        \begin{tabular}{c|cccc}\hline
            $N$ & $F_1(x,y)$ & $F_2(x,y)$ & $F_3(x,y)$ & $F_4(x,y)$ \\ \hline
            1000 & 312 & 422 & 315 & 465 \\
            3375 & 610 & 815 & 614 & 964 \\
            8000 & 1003 & 1314 & 1012 & 1623 \\
            15625 & 1491 & 1931 & 1503 & 2443 \\
            27000 & 2077 & 2641 & 2092 & 3419 \\
            42875  & 2764 & 3465 & 2781 & 4556 \\
            64000  & 3547 & 4373 & 3564 & 5751 \\
            \hline
        \end{tabular}}
    }
    \caption{Numerical rank $\bkt{tol = 10^{-12}}$ of face sharing interaction in $3D$}
    \label{result_3d_face}
\end{figure}
}